\documentclass[times,sort&compress,3p]{elsarticle}
\journal{Journal of Multivariate Analysis}

\usepackage{amsmath,amsfonts,amssymb,amsthm,graphicx,hyperref,url}
\usepackage[labelfont=bf]{caption}

\usepackage{mathtools}
\mathtoolsset{showonlyrefs=true} 
\usepackage{dsfont} 
\usepackage{mathrsfs} 
\usepackage{subcaption}
\usepackage{appendix}
\usepackage{xcolor}
\usepackage{geometry}
\theoremstyle{plain}
\newtheorem{theorem}{Theorem}
\newtheorem{proposition}{Proposition}
\newtheorem{lemma}{Lemma}
\newtheorem{corollary}{Corollary}

\theoremstyle{definition}

\newtheorem{remark}{Remark}

\newcommand{\N}{\mathbb{N}}
\newcommand{\Z}{\mathbb{Z}}
\newcommand{\R}{\mathbb{R}}
\newcommand{\PP}{\mathbb{P}}
\newcommand{\EE}{\mathbb{E}}
\newcommand{\BB}{\mathbb{B}\mathrm{ias}}
\newcommand{\VV}{\mathbb{V}\mathrm{ar}}
\newcommand{\CC}{\mathbb{C}\mathrm{ov}}

\newcommand{\bb}[1]{\boldsymbol{#1}}

\newcommand{\OO}{\mathcal O}
\newcommand{\oo}{\mathrm{o}}
\newcommand{\leqdef}{\vcentcolon=}
\newcommand{\reqdef}{=\vcentcolon}
\newcommand{\rd}{{\rm d}}

\newcommand{\ind}{\mathds{1}}

\newcommand{\e}{\varepsilon}
\DeclareMathAlphabet\mathbfcal{OMS}{cmsy}{b}{n}

\allowdisplaybreaks

\begin{document}

\begin{frontmatter}

    \title{Asymptotic properties of Dirichlet kernel density estimators}%

    \author[a1,a2]{Fr\'ed\'eric Ouimet}\ead{frederic.ouimet2@mcgill.ca}%
    \author[a3]{Raimon Tolosana-Delgado}\ead{r.tolosana@hzdr.de}%

    \address[a1]{California Institute of Technology, Pasadena, CA 91125, USA.}%
    \address[a2]{McGill University, Montreal, QC H3A 2K6, Canada.}%
    \address[a3]{Helmholtz-Zentrum Dresden-Rossendorf, Helmholtz Institute Freiberg for Resources Technology, 09599 Freiberg, Saxony, Germany.}%


    \begin{abstract}
        We study theoretically, for the first time, the Dirichlet kernel estimator introduced by \citet{doi:10.2307/2347365} for the estimation of multivariate densities supported on the $d$-dimensional simplex.
        The simplex is an important case as it is the natural domain of compositional data and has been neglected in the literature on asymmetric kernels.
        The Dirichlet kernel estimator, which generalizes the (non-modified) unidimensional Beta kernel estimator from \citet{MR1718494}, is free of boundary bias and non-negative everywhere on the simplex.
        We show that it achieves the optimal convergence rate $\OO(n^{-4/(d+4)})$ for the mean squared error and the mean integrated squared error, we prove its asymptotic normality and uniform strong consistency, and we also find an asymptotic expression for the mean integrated absolute error.
        To illustrate the Dirichlet kernel method and its favorable boundary properties, we present a case study on minerals processing.
    \end{abstract}


\end{frontmatter}

\vspace{-4mm}
\section{Introduction}\label{sec:introduction}

    Kernel smoothing or kernel density estimation is a well-known methodology to characterize (and visualize) the probability density function of a random variable or random vector in a nonparametric way. It can be considered as a bin-free alternative to histograms, and is particularly useful in multivariate cases with a low to moderate number of dimensions, where the accuracy of histograms dramatically deteriorates with the number of variables due to the curse of dimensionality. Apart from visualization purposes, density estimation can be used for nonparametric alternatives to regression and classification (both supervised and unsupervised). One of the most intuitive usages, for instance, is to construct conditional density plots (in \texttt{R} command ``\texttt{cdplot}''), to represent how the conditional probabilities of a categorical variable depends on quantitative covariables.
    This is true in particular for compositional data, but methods of density estimation on the simplex that address the well-known spill-over problem of traditional kernel estimators are very scarce in the literature, and theoretical results specific to the simplex are almost nonexistent.
    To remedy this situation, our main goal in this paper is to revisit the Dirichlet kernel estimator on the simplex introduced by \citet{doi:10.2307/2347365} and study its asymptotic properties in details.
    A case study on minerals processing presented in Section~\ref{sec:case.study} will show an explicit and elaborate use of conditional density plots using Dirichlet kernel estimators.

    \vspace{0mm}
    Nevertheless, our main contribution in this paper remains theoretical.
    We will find asymptotic expressions for the pointwise bias, the pointwise variance, the mean squared error (MSE) and the mean integrated squared error (MISE).
    These results generalize the ones for the Beta kernel in \cite{MR1718494} ($d = 1$).
    The optimal bandwidth parameters $b$, with respect to MSE and MISE, are also written explicitly.
    In practice, this can be used to implement a plug-in selection method for the bandwidth parameter.
    The asymptotic normality follows from a straightforward verification of the Lindeberg condition for double arrays, although it is completely new even for Beta kernel estimators.
    We also obtain the asymptotics of the mean integrated absolute error (MIAE) and the uniform strong consistency, which generalize the results from \citet{MR1985506}.
    To be more precise, the proof of the $L^1$ asymptotics follows the same strategy but the proof of the uniform strong consistency is completely different and represents our biggest contribution (we combine estimates on the difference of Dirichlet densities with different parameters together with a novel chaining argument).
    Our rates of convergence for the MSE and MISE are optimal, as they coincide (assuming the identification $b \approx h^2$) with the rates of convergence for the MSE and MISE of traditional multivariate kernel estimators, studied for example in \cite{MR0740865}.
    In contrast to other methods of boundary bias reduction (such as the reflection method or boundary kernels (see, e.g., \cite{MR3329609})), this property is built-in for Dirichlet kernel estimators, which makes them one of the easiest to use in the class of estimators that are asymptotically unbiased near (and on) the boundary.
    Dirichlet kernel estimators are also non-negative everywhere on their domain, which is definitely not the case of many estimators corrected for boundary bias.
    This is another reason for their desirability.
    Bandwidth selection methods and their consistency will be investigated thoroughly in upcoming work.

    \vspace{0mm}
    Here is the outline of the paper.
    In Section~\ref{sec:overview}, we present an overview of the literature of asymmetric kernels.
    In Section~\ref{sec:Dirichlet.kernels}, we define Dirichlet kernel estimators, we state some of their basic properties, and we show that Dirichlet kernels are continuous examples in the broader class of multivariate associated kernels introduced by \citet{MR3760293,Kokonendji_Some_2021}.
    In Section~\ref{sec:asymptotic.properties}, our main results are stated, which consists of the asymptotic behavior of the pointwise bias and variance (including points near the boundary of the simplex), the mean squared error, the integrated mean squared error, the mean integrated absolute error, the uniform consistency and the asymptotic normality.
    All the proofs are gathered in Section~\ref{sec:proofs}.
    In Section~\ref{sec:case.study}, the case study on minerals processing is presented.

\section{Overview of the literature}\label{sec:overview}

    Below, we give a systematic overview of the main line of articles on density estimation using Beta kernels (i.e., Dirichlet kernels with $d = 1$), and then we briefly mention several other classes of asymmetric kernels with references.
    There might be more details than the reader expects, but this is because the subject is vast and relatively important references are often disjointed or missing in the literature, which makes it hard for newcomers to get a complete chronological account of the progress in the field.

    \vspace{0mm}
    \citet{doi:10.2307/2347365} were the first to define the Dirichlet kernel estimator from \eqref{eq:Dirichlet.estimator} for the purpose of density estimation.
    The paper compared their performance empirically with an alternative approach, called the logistic-normal kernel method, where the data on the simplex is first sent to $\R^d$ via an additive log-ratio transformation and a multivariate Gaussian kernel smoothing is applied afterwards. The authors recommended Dirichlet kernels over the logistic-normal kernel method if there was a suspicion of sparseness in the data.
    To the best of our knowledge, \citet{doi:10.1016/j.cageo.2009.12.011} were, until now, the only other authors to mention Dirichlet kernels, or even kernel smoothing on the simplex, in any meaningful way.
    They compared numerically the outcome of different ways of establishing the bandwidth covariance structure of an additive logistic normal kernel: a full bandwidth matrix chosen by cross-validation, a full bandwidth matrix chosen by unconstrained plug-in method, and the logistic-normal kernel method presented in \cite{doi:10.2307/2347365} with a bandwidth matrix $H = h S$ ($S$ is the sample covariance matrix and $h$ is chosen to maximize the pseudo-likelihood).

    \vspace{0mm}
    \citet{MR1685301} were the first to study Beta kernels ($d = 1$) theoretically, and they did so in the context of smoothing for regression curves with equally spaced and fixed design points.
    The asymptotics of the pointwise bias, the integrated variance and the MISE for the estimator of the regression function were found (the optimal MISE was shown to be $\OO(n^{-4/5})$).
    These results extended to Beta kernel estimators some parts of the results from \citet{MR0858109}, who was working with the closely related Bernstein estimators.
    \citet{MR1718494} was the first author to study (unmodified) Beta kernel estimators $(\hat{f}_1)$ theoretically in the context of density estimation.
    A certain boundary Beta kernel modification, denoted by $\hat{f}_2$, was also considered.
    The asymptotics of the pointwise bias and pointwise variance of both $\hat{f}_1$ and $\hat{f}_2$ were found everywhere on $(0,1)$ as well as the MISE (the optimal MISE was shown to be $\OO(n^{-4/5})$).
    Numerical comparisons of the estimators were made with the local linear estimator of \citet{MR1192215} and \citet{doi:10.1007/BF00147776} and the non-negative modification proposed by \citet{MR1422417}, although various criticisms were raised by \citet{MR2598955}.
    \citet{MR1742101} generalized the results of \citet{MR1685301} to arbitrary collections of fixed design points using a Gasser-M{\"u}ller type estimator (\citet{doi:10.1007/BFb0098489}).
    A boundary Beta kernel modification, analogous to $\hat{f}_2$ from \citet{MR1718494}, was also considered and the asymptotic results were also extended to that regression curve estimator.
    In \cite{MR1910175}, those results were further extended to stochastic design points using a local linear smoother with Beta kernel (and Gamma kernel when the data is supported on $[0,\infty)$ instead of $[0,1]$) analogous to the traditional version proposed by \citet{MR1193323}.

    \vspace{0mm}
    \citet{MR1985506} computed the asymptotics of the MIAE for Beta kernel estimators, which extended the analogous result for traditional kernel estimators found in Theorem~2 of \citet{MR760686}.
    As pointed out by \citet[p.44]{MR3329609}, there are many reasons to estimate the MIAE as it enjoys many advantages over the MISE.
    It puts more emphasis on the tails of the target density, it is a dimensionless quantity, it is invariant to monotone changes of scale, and it is uniformly bounded (by $2$). (For a thorough study of the $L^1$ point of view, see \citet{MR780746}.)
    Another result that was proved by \citet{MR1985506} for Beta kernel estimators is the uniform strong consistency, which extended the analogous result for traditional kernel estimators found in Theorem~1 of \citet{MR859633}. However, the proof in \cite{MR1985506} is completely different.
    They apply an integration by parts trick to relate $\sup_{x\in [0,1]} |\hat{f}_{n,b}(x) - \EE[\hat{f}_{n,b}(x)]|$ to the supremum of the recentered empirical c.d.f.\ and then estimate the latter with the Dvoretzky-Kiefer-Wolfowitz inequality.
    In the context of Bernstein estimators (see, e.g., \citet{MR1910059}), a similar idea (its discrete version) was to apply a union bound on a partition of the support of the target density into small boxes and then use concentration bounds on the supremum of ``increments'' of the recentered empirical c.d.f.\ inside each box, where the width of the boxes is carefully chosen so that the bounds are summable and the result follows by the Borel-Cantelli lemma. In the present paper, we will instead apply a novel chaining argument (that might be of independent interest) and give ourselves a buffer on the boundary to avoid technical issues related to the partial derivatives of the Dirichlet density $K_{\bb{\alpha},\beta}$ with respect to $\alpha_1,\dots,\alpha_d,\beta$. It is not obvious how to generalize the proof of \citet{MR1985506} on the simplex.

    \vspace{0mm}
    \citet{doi:10.1016/j.jbankfin.2003.10.018} were the first to use Beta kernels to estimate recovery rate densities of defaulted bonds.
    They also investigated the finite sample performance of the Beta kernel density estimator by comparing the averages of integrated squared errors of Monte Carlo samples against two other methods: traditional Gaussian kernel smoothing, and a logistic transformation combined with traditional Gaussian kernel smoothing and a back transformation of the estimated density by multiplying it with the derivative of the inverse mapping (i.e., the appropriate Jacobian).
    Furthermore, they showed that the usual practice of approximating the recovery function through a Beta density calibrated with the sample mean and variance should be handled with caution as the inflexibility of the parametric approach can lead (for example) to an underestimation of the Value-at-Risk.
    \citet{Gourieroux_Monfort_2006} showed the non-consistency of the Beta kernel approach to estimate the recovery rate density 
    when there are point masses at $0$ (total loss) or $1$ (total recovery). Without point masses at $0$ or $1$, the Beta kernel approach features significant bias in finite sample according to the authors. In large sample, the method is consistent, but they showed that competing approaches (called micro-Beta and macro-Beta; these are two types of normalization of the vanilla Beta kernel estimator) can provide more accurate results when estimating the continuous part of the loss-given-default distribution.

    \vspace{0mm}
    \citet{MR2206532} derived the asymptotic behavior of Beta kernel functionals (and Gamma kernel functionals when the support of the target density $f$ is $[0,\infty)$ instead of $[0,1]$) of the form
    \begin{equation}
        \int_A \varphi(x) \big[\hat{f}_{n,b}(x) - f(x)\big] \rd x,
    \end{equation}
    where $\varphi$ is a bounded regular function and $A$ is the support of $f$, by applying a central limit theorem for degenerate $U$-statistics with variable kernel.
    The ideas are similar to those applied by \citet{MR734096} for traditional kernel estimators.

    \vspace{0mm}
    \citet{MR2756441} applied two multiplicative bias correction methods (from \citet{MR585714} and \citet{MR1354232}, respectively) to the two Beta kernel estimators from \citet{MR1718494}, which, under sufficient smoothness conditions on the target density, had an effect of reducing the pointwise bias while the order of magnitude of the pointwise variance stayed the same.
    Under the assumption that the target density is four-times continuously differentiable, the asymptotics of the pointwise bias, pointwise variance, MSE and MISE were found (the optimal rate of convergence was also shown to be $\OO(n^{-8/9})$ for the MISE and MSE inside $(0,1)$, instead of the usual $\OO(n^{-4/5})$).
    Hirukawa also investigated the numerical performance of the Beta and modified Beta kernel estimators of \citet{MR1718494} as well as them under the micro/macro normalizations from \citet{Gourieroux_Monfort_2006}, and all these combinations (except for micro) under the two aforementioned multiplicative bias correction methods.
    For all $14$ combinations that were studied, the bandwidth parameter was selected according to two methods: rule-of-thumb and plug-in.
    He concluded that the estimators corrected for bias under the JLN-method of \citet{MR1354232} had a superior performance compared to the bias-uncorrected estimators.

    \vspace{0mm}
    \citet{MR2568128} generalized the results of \citet{MR1718494,MR1794247} to the multidimensional setting.
    The kernels that they considered were the products of one-dimensional asymmetric kernels (Beta kernels, modified Beta kernels, or Gamma kernels, modified Gamma kernels, local linear kernel; depending on the support of the marginals of the target density).
    Asymptotics of the pointwise bias, pointwise variance and MISE were found.
    The authors also proved the asymptotic normality, uniform strong consistency and the almost-sure convergence of the MISE when the bandwidth parameter $b$ is selected via a least-square cross-validation method.
    The finite sample performance of the estimators were investigated by comparing the mean and standard deviation of integrated squared errors of Monte Carlo samples under various target densities. When the target density is supported on $[0,\infty)^d$, their results showed that the proposed estimators perform almost as well as the traditional Gaussian kernel estimator when there are no boundary problems, and in general, the modified Gamma and local linear estimators dominate the other estimators (i.e., Gamma, and Gaussian with and without the log-transformation).

    \vspace{0mm}
    \citet{MR2598955} showed that the performance of the Beta kernel estimator is very similar to that of the reflection estimator of \citet{MR797636}, which does not have the boundary problem only for densities exhibiting a shoulder condition at the endpoints of the support. For densities not exhibiting a shoulder condition, they showed that the performance of the Beta kernel estimator at the boundary was inferior to that of the well-known boundary kernel estimator, see, e.g., \cite{doi:10.1007/BFb0098489,MR816088,MR1649872,MR1752313} and references therein.

    \vspace{0mm}
    \citet{Bouezmarni_van_Bellegem_2011} introduced the idea of using Beta kernels to estimate the spectral density of long memory time series.
    Their technical report was recently updated and extended to include the case of short memory time series, and published as \cite{MR4148613}.
    The asymptotics of the pointwise bias and variance for the spectral density estimator were obtained, as well as the uniform weak consistency on compacts and the relative weak consistency (i.e., the convergence to one, in probability, of the ratio of the estimator to the target density).
    A cross-validation method was also studied for the selection of the bandwidth parameter $b$ following the general method of \citet{MR819597}.
    The authors show that the estimator has a better boundary behavior than traditional methods.

    \vspace{0mm}
    \citet{MR2775207} considered Beta kernel estimators for the estimation of density functions in $\beta$-H\"older spaces and under $L^p$ losses.
    They showed that the estimator is minimax whenever $0 < \beta \leq 2$ and $1 \leq p < 4$, but not minimax otherwise.
    In particular, this means that Beta kernel estimators are minimax under $L^2$ losses if the target density is twice continuously differentiable (which is the most common assumption in the literature).
    These types of results are unique in the literature on asymmetric kernels; it would be interesting to see to which extent they hold for other asymmetric kernels.
    In \cite{MR3333996}, the same authors constructed a data-driven (also called adaptative) procedure of bandwidth selection, inspired by the method of \citet{MR1147167}, that achieves the minimax rate of convergence without a priori knowledge of the regularity $\beta$ of the target density.
    They found that the procedure was competitive with the more common cross-validation methods, and the numerical computations were significantly faster.

    \vspace{0mm}
    In line with \citet{MR3299088}, \citet{MR3463548} considered an additive bias correction method of \citet{MR448691} and the nonnegative bias correction methods of \citet{MR585714} and \citet{MR1272163}, in the context of Beta kernel density estimators.
    Under the assumption that the target density is four-times continuously differentiable, the asymptotics of the pointwise bias, pointwise variance, MSE and MISE were found (the optimal rate of convergence was also shown to be $\OO(n^{-8/9})$ for the MISE and MSE inside $(0,1)$, instead of the usual $\OO(n^{-4/5})$).
    In particular, the results partially complemented/extended/corrected those in \cite{MR2756441}.
    The finite sample performance of the estimators was compared by computing the average and standard deviation of the integrated squared errors of Monte Carlo samples when the target density is a bimodal mixture of Beta densities.

    \vspace{0mm}
    Various other statistical topics related to Beta kernels are treated, for example, by
    \citet{Charpentier2006phd, 
    MR2416803, 
    Charpentier_Fermanian_Scaillet_2007, 
    doi:10.1109/ICMLC.2007.4370716, 
    MR2578075, 
    Manivong2009master, 
    MR4076244, 
    MR4302589, 
    MR4136585, 
    Hirukawa_Murtazashvili_Prokhorov_2020}. 

    \vspace{0mm}
    When the density of the observations is not supported on the compact interval $[0,1]$, we can always apply a transformation to the data that maps to $[0,1]$ (for example, $x\mapsto 1 / (1 + x)$ maps $[0,\infty)$ to $[0,1]$) and then use Beta kernels.
    Another approach is to use asymmetric kernels that match the support of the target density directly.
    For instance, we have the following six common classes of asymmetric kernels on $[0,\infty)$: Gamma, Inverse Gamma, Inverse Gaussian, Birnbaum Saunders, Log-Normal and Reciprocal Inverse Gaussian.
    Many of the results that are proved for Beta kernels have been extended to these classes (or could be extended without much trouble), so we just list the relevant papers here instead of repeating every point above:
    \begin{itemize}\setlength\itemsep{-0.1em}
        \item Gamma kernels, see, e.g., \cite{MR1794247,MR1910175,MR2179543,MR2137490,MR2206532,MR2454617,MR2568128,MR2756423,MR2606717,MR2595129,MR2801351,MR3130451,doi:10.3182/20130703-3-FR-4038.00086,doi:10.3182/20130619-3-RU-3018.00214,MR3304359,doi:10.3390/econometrics4020028,MR3520774,MR3494026,MR3843043,MR3962366,doi:10.1515/snde-2018-0001,MR4302589,doi:10.1080/02664763.2021.1881456,MR2665093,MR4147689,MR4086907,MR4136585,Ma2019master,doi:10.1007/s10958-021-05325-2};
        \item Inverse Gamma kernels, see, e.g., \cite{MR2903523,MR2929784,MR3174299,MR3544186,MR3648359};
        \item Inverse Gaussian kernels, see, e.g., \cite{MR2053071,MR2179543,MR2229885,MR2606717,doi:10.12988/pms.2014.4616,MR3131281,MR3713468,MR3937087,MR3978111,MR4136585};
        \item Birnbaum-Saunders kernels, see, e.g., \cite{Jin_Kawczak_2003,MR3040246,doi:10.1007/s00477-012-0684-8,MR3131281,MR3456321,MR3754719,MR3819800,MR4096263,MR4244643,Chekkal_Lagha_Zougab_2021};
        \item Log-Normal kernels, see, e.g., \cite{Jin_Kawczak_2003,MR2606717,doi:10.2139/ssrn.2514882,MR3540109,MR3885552};
        \item Reciprocal Inverse Gaussian kernels, see, e.g., \cite{MR2053071,MR2179543,MR3131281};
    \end{itemize}
    Other asymmetric kernel classes have been considered such as the Weibull, but not much theoretical work has been done in those cases.
    Beta kernels and the six kernel classes above are the most common in the literature.

    \vspace{0mm}
    As can be seen from the lists just above, one of the current weaknesses in the theory of asymmetric kernels is the segmentation of the results by the specific form of the kernel. Fortunately, this concern has been addressed in recent years and a theory of so-called associated kernels has been developed (through many articles written by C\'elestin Kokonendji and his co-authors) to unify the theory of asymmetric kernels with the one for traditional kernels in both the univariate and multivariate settings, as well as treating discrete and continuous variants together.
    Associated kernels were developed and studied in the discrete setting first, see, e.g.,
    \cite{MR2549110, 
    MR2511101, 
    MR2834036, 
    MR3169297, 
    MR3215725, 
    MR3337328, 
    MR3474762, 
    MR3499725, 
    MR3567789, 
    MR3566162, 
    MR3547951, 
    MR3453033, 
    MR3807369, 
    MR3760293}, 
    then came the continuous univariate associated kernels, see \citet{MR3885552}, and finally multivariate associated kernels (discrete and continuous), see \citet{MR3760293,Kokonendji_Some_2021}.
    The case of multivariate continuous associated kernels is the only one relevant to us in the present paper, so let us briefly mention the main contributions from \citet{MR3760293,Kokonendji_Some_2021}. In those papers, the definition of associated kernels was extended to the multivariate setting (the explicit definition is given around Equation~\ref{def:associated.kernels} below).
    In particular, the definition is broad enough that it covers all the asymmetric kernel classes mentioned in the paragraph above, as well as various multivariate generalizations treated in the literature and also traditional multivariate kernels (such as the ubiquitous Gaussian kernel, for example).
    In the two papers, the authors showed how associated kernels with a given correlation structure can be constructed using a variant of the mode-dispersion method, see \cite[p.115-116]{MR3760293}, and they derived asymptotic expressions for the pointwise bias and variance of the corresponding smoothing estimator, see \cite[Proposition~2.9]{MR3760293} and \cite[Proposition~2]{Kokonendji_Some_2021}.
    In \cite{MR3760293}, they also showed how to modify the estimator near the boundary to reduce the pointwise bias even more, similarly to the modified Beta kernel estimator introduced by \citet{MR1718494}.
    In \cite{Kokonendji_Some_2021} (see \cite{MR2549110} for the discrete case), a semiparametric approach adapted from \citet{MR1345205} was applied when the associated kernel is parametrized by another (possibly multidimensional) parameter $\bb{\theta}$, and asymptotics for the pointwise bias was also derived in this case. Specific examples are treated in both papers to illustrate the wide applicability of associated kernels to the currently segmented literature on asymmetric kernels and related estimators.
    In Section~\ref{sec:Dirichlet.kernels}, the definition of multivariate associated kernels will be given and we will see that Dirichlet kernels are just a special case. However, most of the asymptotic results in Section~\ref{sec:asymptotic.properties} cannot be deduced from those in \cite{MR3760293,Kokonendji_Some_2021}, so significant work remains to be done.

    \vspace{0mm}
    Various other statistical topics related to asymmetric kernels are treated, for example, in
    \cite{MR2411743, 
    MR2516437, 
    doi:10.1117/12.811578, 
    MR2587103, 
    MR2606717, 
    MR3821525, 
    doi:10.1007/s00138-009-0237-4, 
    doi:10.1109/ICMT.2011.6001976, 
    MR2903382, 
    MR2869005, 
    doi:10.1016/j.jempfin.2012.04.001, 
    MR2975812, 
    Chaubey_Li_2013, 
    MR3175804, 
    MR3178361, 
    MR3139348, 
    MR3256394, 
    MR3178355, 
    MR3200997, 
    MR3384258, 
    MR3299088, 
    MR3357933, 
    MR3350456, 
    arXiv:1512.03188, 
    doi:10.3233/JCM-170731, 
    MR3743306, 
    MR3850066, 
    MR3833873, 
    doi:10.1051/itmconf/20182300037, 
    MR3933005, 
    MR3885552, 
    MR3873902, 
    MR3907679, 
    MR3975162, 
    MR4029144, 
    Mombeni_et_al_2019_accepted, 
    MR4097810, 
    MR3979322, 
    MR4037101, 
    Ercelik_Nadar_2020, 
    doi:10.1007/s10463-020-00772-1, 
    MR4279948, 
    Chekkal_Lagha_Zougab_2021, 
    arXiv:2011.14893, 
    doi:10.1002/sta4.410, 
    arXiv:2104.04882, 
    MR4223858}. 

    \vspace{0mm}
    It should mentioned that Bernstein density estimators, studied theoretically, among other authors, by \citet{MR0397977}, \citet{MR0574548,MR0638651,MR0755576}, \citet{MR0726014}, \citet{MR0791719}, \citet{MR1293514}, \citet{MR1910059}, \citet{MR2068610}, \citet{MR2270097}, \citet{MR2351744}, \citet{MR2662607,MR2925964}, \citet{MR3174309}, \citet{MR3412755}, \citet{MR3474765}, \citet{MR4287788,arXiv:2006.11756,MR4213687}, \citet{Hanebeck2020master}, \citet{doi:10.1007/s10463-020-00783-y} and \citet{MR4130895}, share many of the same asymptotic properties (with proper reparametrization).
    As such, the literature on Bernstein estimators has parallelled that of Beta kernel estimators and other asymmetric kernel estimators in the past twenty years.
    For an overview of the vast literature on Bernstein estimators, we refer the interested reader to \citet{MR4287788}.

\section{Dirichlet kernels: definition and properties}\label{sec:Dirichlet.kernels}

    The $d$-dimensional simplex and its interior are defined by
    \begin{equation}\label{eq:def.simplex}
        \mathcal{S}_d \leqdef \big\{\bb{s}\in [0,1]^d: \|\bb{s}\|_1 \leq 1\big\} \quad \text{and} \quad \mathrm{Int}(\mathcal{S}_d) \leqdef \big\{\bb{s}\in (0,1)^d: \|\bb{s}\|_1 < 1\big\},
    \end{equation}
    where $\|\bb{s}\|_1 \leqdef \sum_{i=1}^d |s_i|$ and $d\in \N$.
    For $\alpha_1, \dots, \alpha_d, \beta > 0$, the density of the $\mathrm{Dirichlet}\hspace{0.2mm}(\bb{\alpha},\beta)$ distribution is
    \begin{equation}\label{eq:Dirichlet.density}
        K_{\bb{\alpha},\beta}(\bb{s}) \leqdef \frac{\Gamma(\|\bb{\alpha}\|_1 + \beta)}{\Gamma(\beta) \prod_{i=1}^d \Gamma(\alpha_i)} \cdot (1 - \|\bb{s}\|_1)^{\beta - 1} \prod_{i=1}^d s_i^{\alpha_i - 1}, \quad \bb{s}\in \mathcal{S}_d.
    \end{equation}
    For a given bandwidth parameter $b > 0$, and a sample of i.i.d.\ observations $\bb{X}_1,\dots,\bb{X}_n$ that are $F$ distributed ($F$ is unknown) with a density $f$ supported on $\mathcal{S}_d$, the Dirichlet kernel estimator is defined by
    \begin{equation}\label{eq:Dirichlet.estimator}
        \hat{f}_{n,b}(\bb{s}) \leqdef \frac{1}{n} \sum_{i=1}^n K_{\bb{s} / b + 1, (1 - \|\bb{s}\|_1) / b + 1}(\bb{X}_i), \quad \bb{s}\in \mathcal{S}_d.
    \end{equation}
    Two smoothing examples are given in Fig.~\ref{fig:smoothing.examples} with $d = 2$, where the two subfigures on the left-hand side illustrate the target densities, and the two subfigures on the right-hand side show the corresponding estimates.
    The reader can also see that the shape of the kernel changes with the position $\bb{s}$ on the simplex; this is in contrast with traditional estimators where the kernel is the same for every point.
    This variable smoothing allows Dirichlet kernel estimators (and more generally, asymmetric kernel estimators) to avoid the boundary bias problem of traditional kernel estimators.

    \begin{figure}[t]
        \vspace{-3mm}
        \centering
        \begin{subfigure}[b]{0.45\textwidth}
            \includegraphics[width=\textwidth]{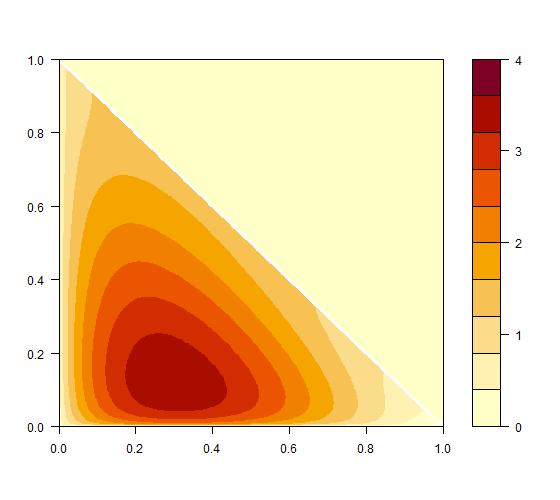}
            \label{fig:target.density.1}\vspace{-7mm}
        \end{subfigure}%
        \hspace{-4.5mm}
        \begin{subfigure}[b]{0.45\textwidth}
            \includegraphics[width=\textwidth]{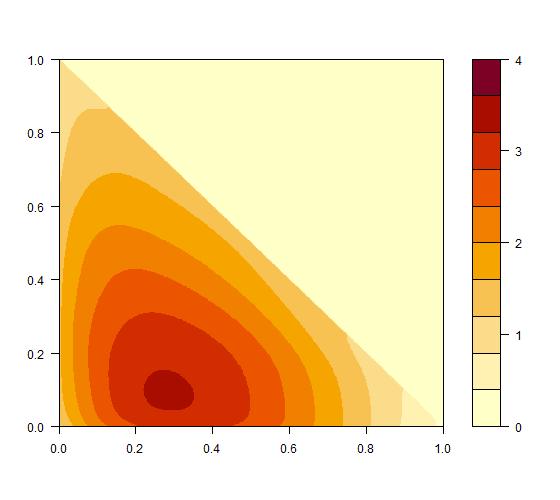}
            \label{fig:smoothing.1}\vspace{-7mm}
        \end{subfigure}%

        \vspace{-4mm}
        \noindent
        \begin{subfigure}[b]{0.45\textwidth}
            \includegraphics[width=\textwidth]{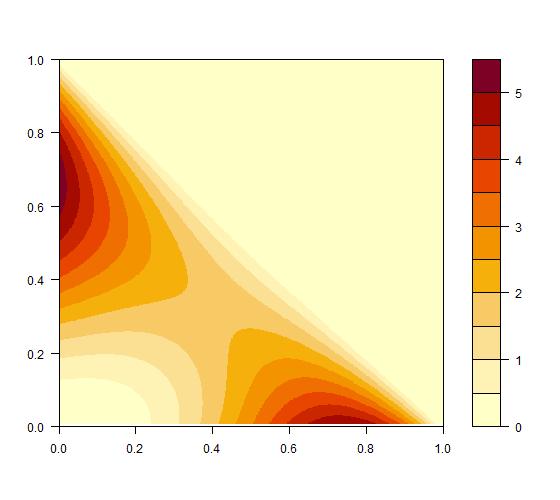}
            \label{fig:target.density.2}
        \end{subfigure}%
        \hspace{-4.5mm}
        \begin{subfigure}[b]{0.45\textwidth}
            \includegraphics[width=\textwidth]{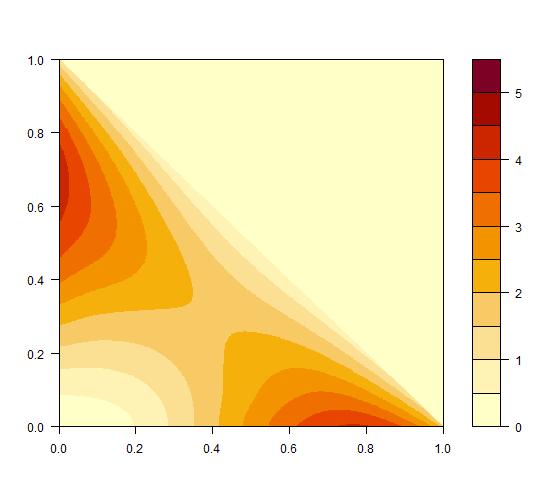}
            \label{fig:smoothing.2}
        \end{subfigure}%

        \vspace{-9mm}
        \noindent
        \caption{Two-dimensional examples of contour plots for two target densities $f$ (left) and their respective estimate $\hat{f}_{n,b}$ (right), using Dirichlet kernels with $n = 10000$ and $b = n^{-1/3}$. The first target density (top left) is the mixture $0.4 \cdot \text{Dirichlet}(1.3,1.6,1)+0.6 \cdot \text{Dirichlet}(1.7,1.2,2.5)$, whereas the second target density (bottom left) is the mixture $0.4 \cdot \text{Dirichlet}(4,1,2) + 0.6 \cdot \text{Dirichlet}(1,3,2)$.}
        \label{fig:smoothing.examples}
    \end{figure}

    \vspace{0mm}
    Note that the estimator $\hat{f}_{n,b}$ is not a proper density since it does not integrate to~$1$ exactly, but it does integrate to~$1$ asymptotically.
    To prove this, denote $\mathrm{Bulk} \leqdef \{\bb{x}\in \mathcal{S}_d : |x_i - r_i| \leq (1/b + d + 2)^{-1/2} b^{-1/6} / 2 ~\forall i\in \{1,\dots,d+1\}\}$ where $r_i \leqdef (s_i + b) / (1 + b(d + 1))$, $x_{d+1}\leqdef 1 - \|\bb{x}\|_1$, $s_{d+1}\leqdef 1 - \|\bb{s}\|_1$, then $\bb{X}_1,\dots,\bb{X}_n\in \mathrm{Bulk}$ with probability $1 - \OO(n \exp(-b^{-1/3}/2))$, as $n\to \infty$, by a straightforward union bound and concentration argument (see \cite[Equation~(21)]{doi:10.1002/sta4.410}), which is $1 - \oo(1)$ as long as $b = \oo((\log n)^{-3})$ (this is a very weak assumption given that $b_{\mathrm{opt}} \asymp n^{-2/(d+4)}$ in Theorem~\ref{thm:MISE.optimal.density} below).
    Therefore, by the multivariate normal approximation for the $\mathrm{Dirichlet}\hspace{0.2mm}(\bb{\alpha} = \bb{s} + b, \beta = 1 - \|\bb{s}\|_1 + b)$ density derived in \cite[Theorem~1]{doi:10.1002/sta4.410}, we have, as $n\to \infty$,
    \begin{align}
        \int_{\mathcal{S}_d} \hat{f}_{n,b}(\bb{s}) \rd \bb{s}
        &= \frac{1}{n} \sum_{i=1}^n \int_{\mathcal{S}_d} K_{\bb{s} / b + 1, (1 - \|\bb{s}\|_1) / b + 1}(\bb{X}_i) \rd \bb{s} = \frac{1}{n} \sum_{i=1}^n \int_{\mathcal{S}_d} K_{\bb{s} / b + 1, (1 - \|\bb{s}\|_1) / b + 1}(\bb{X}_i) \ind_{\mathrm{Bulk}}(\bb{X}_i)\rd \bb{s} + \oo_{\PP}(1) \notag \\
        &= \frac{1}{n} \sum_{i=1}^n \int_{\R^d} \frac{\exp\big(-\frac{1}{2} \bb{\delta}_{\bb{X}_i}^{\top} \Sigma_{\bb{r}}^{-1} \bb{\delta}_{\bb{X}_i}\big)}{\sqrt{(2\pi)^d \frac{(1 - \|\bb{r}\|_1) \prod_{i=1}^d r_i}{(1/b + d + 2)^d}}} \rd \bb{s} + \oo_{\PP}(1) \quad \text{with } \bb{\delta}_{\bb{X}_i} \leqdef \frac{\bb{X}_i - \bb{r}}{(1/b + d + 2)^{-1/2}}, ~~\Sigma_{\bb{r}} \leqdef \mathrm{diag}(\bb{r}) - \bb{r} \bb{r}^{\top} \notag \\
        &= \frac{1}{n} \sum_{i=1}^n \int_{\R^d} \frac{\exp(-\frac{1}{2} \bb{z}^{\top} \bb{z})}{\sqrt{(2\pi)^d}} \rd \bb{z} + \oo_{\PP}(1) = 1 + \oo_{\PP}(1),
    \end{align}
    where the second to last equality follows from the change of variables $\bb{z} = \Sigma_{\bb{r}}^{-1/2} \bb{\delta}_{\bb{X}_i}$ and the fact that the symmetric positive definite matrix $\Sigma_{\bb{r}}$ has determinant $(1 - \|\bb{r}\|_1) \prod_{i=1}^d r_i$ by \citet[Theorem~1]{MR1157720}.

    \newpage
    Under mild regularity conditions, we will prove several asymptotic results for $\hat{f}_{n,b}$ in this paper: pointwise bias and variance, mean squared error (MSE), mean integrated squared error (MISE), mean integrated absolute error (MIAE), uniform strong consistency and asymptotic normality.
    The results are stated in Section~\ref{sec:asymptotic.properties} and proved in Section~\ref{sec:proofs}.

    \vspace{0mm}
    As mentioned in Section~\ref{sec:overview}, Dirichlet kernels are continuous multivariate associated kernels.
    To see this, here is the definition, which can be found in \cite[Definition~2.1]{MR3760293} and \cite[Definition~1]{Kokonendji_Some_2021}.
    In the continuous case, multivariate associated kernels are defined as density functions $K_{\bb{x},H}$, supported on a subset of $[0,\infty)^d$ (denoted by $\mathbb{T}_d^+$), that are parametrized by points $\bb{x}$ in the support and a bandwidth matrix $H$ which is symmetric and positive definite.
    As pointed out in \cite[Table~2.1]{MR3760293}, the matrix $H$ can be full, diagonal or a Scott matrix (i.e., parametrized by a single parameter $h$), and $K_{\bb{x},H}$ must satisfy the following crucial property: For $\mathbfcal{Z}_{\bb{x},H}$ a $d$-dimensional random vector which is $K_{\bb{x},H}$-distributed, then as $H\to 0_{d\times d}^{+}$,
    \begin{equation}\label{def:associated.kernels}
        \EE[\mathbfcal{Z}_{\bb{x},H}] - \bb{x} \reqdef \bb{a}(\bb{x},H) \to \bb{0}_d, \qquad \CC(\mathbfcal{Z}_{\bb{x},H}) \reqdef B(\bb{x},H) \to 0_{d\times d}^{+}.
    \end{equation}
    (This definition extends naturally to the discrete setting, and it is also possible to have a mix of discrete and continuous components for $\mathbfcal{Z}_{\bb{x},H}$.)
    Under this definition, and under relatively weak regularity conditions on the target density $f$ (see \cite[p.167]{Kokonendji_Some_2021}), the asymptotics of the pointwise bias and variance for the corresponding estimator, denoted by $\widetilde{f}_{n,H}(\bb{x}) = n^{-1} \sum_{i=1}^n K_{\bb{x},H}(\bb{X}_i)$, were shown by \citet{MR3760293,Kokonendji_Some_2021} to be, for any given $\bb{x}\in \mathbb{T}_d^+$,
    \begin{equation}\label{eq:asymptotics.bias.variance.Kokonendji.Some}
        \begin{aligned}
            \BB[\widetilde{f}_{n,H}(\bb{x})]
            &= \nabla f(\bb{x})^{\top} \bb{a}(\bb{x},H) + \frac{1}{2} \mathrm{tr}\left\{\mathrm{Hessian}(f)(\bb{x}) \left[B(\bb{x},H) + \bb{a}(\bb{x},H)^{\top} \bb{a}(\bb{x},H)\right]\right\} + \oo\big(\mathrm{tr}(B(\bb{x},H))\big), \\
            \VV(\widetilde{f}_{n,H}(\bb{x}))
            &= n^{-1} f(\bb{x}) \int_{\mathbb{T}_d^+} K_{\bb{x},H}^2(\bb{u}) \rd \bb{u} + \oo_{\bb{x}}\left(\frac{n^{-1}}{(\det(H))^{\hspace{0.2mm}r_{\bb{x}}}}\right), \qquad (\text{$\bb{x}$ must be in the interior of $\mathbb{T}_d^{+}$ here})
        \end{aligned}
    \end{equation}
    with $r_{\bb{x}} \leqdef \sup\{p > 0 : \liminf_{n\to \infty} \int_{\mathbb{T}_d^+} K_{\bb{x},H}^2(\bb{u}) \rd \bb{u} \, (\det(H))^{\hspace{0.3mm}p} > 0\}$.

    \vspace{0mm}
    In order to obtain explicit expressions for the pointwise bias and variance for the Dirichlet kernel estimator (using Equation~\eqref{eq:asymptotics.bias.variance.Kokonendji.Some}), we can estimate the expectation and covariances of the random vector
    \begin{equation}\label{eq:def.xi.Dirichlet.r.v}
        \bb{\xi}_{\bb{s}} = (\xi_1,\dots,\xi_d) \sim \mathrm{Dirichlet}\big(\tfrac{\bb{s}}{b} + 1, \tfrac{(1 - \|\bb{s}\|_1)}{b} + 1\big), \quad \bb{s}\in \mathcal{S}_d.
    \end{equation}
    For all $i,j\in \{1,\dots,d\}$, straightforward calculations yield (see, e.g., \citet[p.39]{MR2830563} for $(\star)$):
    \vspace{-2mm}
    \begin{align}
        \EE[\xi_i]
        &\stackrel{(\star)}{=} \frac{\frac{s_i}{b} + 1}{\frac{1}{b} + d + 1} = \frac{s_i + b}{1 + b(d+1)}
        = s_i + b(1 - (d+1)s_i) + \OO(b^2), \label{eq:expectation.explicit.estimate} \\[1mm]
        \CC(\xi_i,\xi_j)
        &\stackrel{(\star)}{=} \frac{(\frac{s_i}{b} + 1)((\frac{1}{b} + d + 1) \ind_{\{i = j\}} - (\frac{s_j}{b} + 1))}{(\frac{1}{b} + d + 1)^2 (\frac{1}{b} + d + 2)}
        = \frac{b (s_i + b)(\ind_{\{i = j\}} - s_j + b (d + 1) \ind_{\{i = j\}} - b)}{(1 + b(d+1))^2 (1 + b(d+2))} \\
        &= b s_i (\ind_{\{i = j\}} - s_j) + \OO(b^2), \label{eq:covariance.explicit.estimate.first} \\[1mm]
        \EE[(\xi_i - s_i) (\xi_j - s_j)]
        &= \CC(\xi_i,\xi_j) + (\EE[\xi_i] - s_i)(\EE[\xi_j] - s_j)
        = b s_i (\ind_{\{i = j\}} - s_j) + \OO(b^2). \label{eq:covariance.explicit.estimate}
    \end{align}
    This shows that $\hat{f}_{n,b}$ is a multivariate associated kernel supported on $\mathbb{T}_d^+ = \mathcal{S}_d$ according to Definition~2.1 in \cite{MR3760293} (alternatively, Definition~1 in \cite{Kokonendji_Some_2021}).
    (In fact, the expression \eqref{eq:covariance.explicit.estimate} even shows that our estimator $\hat{f}_{n,b}$ could be derived asymptotically, a posteriori, from the mode-dispersion method described in \cite[p.115-116]{MR3760293}.)
    Consequently, under the assumption that $f$ is twice continuously differentiable on $\mathcal{S}_d$, we can get the asymptotics of the pointwise bias and variance from \eqref{eq:asymptotics.bias.variance.Kokonendji.Some} with $\bb{a}(\bb{x}, H) = [b(1 - (d+1)s_i) + \OO(b^2)]_{i=1}^d$ and $B(\bb{x},H) = b \, (\mathrm{diag}(\bb{s}) - \bb{s} \bb{s}^{\top}) + \OO(b^2)$.
    The explicit expression we obtain for the pointwise bias is written in Theorem~\ref{thm:bias.var.density} (it is just a special case of the broader results obtained by \citet{MR3760293,Kokonendji_Some_2021}).
    However, to be clear, the expression for the pointwise variance in Theorem~\ref{thm:bias.var.density} is more precise than the one we obtain from the above argument (the technical bound on the main part of the variance in Lemma~\ref{lem:A.b.x.asymptotics} in particular is necessary to obtain the asymptotics of the MISE rigorously).
    Also, our expression for the pointwise variance in Theorem~\ref{thm:bias.var.density} only assumes that $f$ is Lipschitz continuous instead of twice continuously differentiable, and we even prove the asymptotics near the boundary, which does not follow from the results in \citet{MR3760293,Kokonendji_Some_2021}.

\section{Main results}\label{sec:asymptotic.properties}

For each result in this section, one of the following two assumptions will be used:
\begin{align}
    &\bullet \quad \text{The density $f$ is Lipschitz continuous on $\mathcal{S}_d$.} \label{eq:assump:f.density} \\
    &\bullet \quad \text{The density $f$ is twice continuously differentiable on $\mathcal{S}_d$.} \label{eq:assump:f.density.2}
\end{align}
    Also, here are some notations we will use throughout the rest of the paper. The notation $u = \OO(v)$ means that $\limsup |u / v| < C < \infty$ as $b\to 0$ or $n\to \infty$, depending on the context.
    The positive constant $C$ can depend on the target density $f$ and the dimension $d$, but no other variable unless explicitly written as a subscript. The most common occurrence is a local dependence of the asymptotics with a given point $\bb{s}$ on the simplex, in which case we would write $u = \OO_{\bb{s}}(v)$.
    In a similar fashion, the notation $u = \oo(v)$ means that $\lim |u / v| = 0$ as $b\to 0$ or $n\to \infty$.
    Subscripts indicate which parameters the convergence rate can depend on.
    The symbol $\mathscr{D}$ over an arrow `$\longrightarrow$' will denote the convergence in law (or distribution).
    We will use the shorthand $[d] \leqdef \{1,\dots,d\}$ in several places.
    The bandwidth parameter $b = b(n)$ is always implicitly a function of the number of observations, the only exception being in Lemma~\ref{lem:A.b.x.asymptotics} and its proof.
    Finally, we denote the expectation of $\hat{f}_{n,b}(\bb{s})$ by
    \begin{equation}\label{eq:f.b.star.representation.1}
        f_b(\bb{s}) \leqdef \EE\big[\hat{f}_{n,b}(\bb{s})\big] = \EE[K_{\bb{s} / b + 1, (1 - \|\bb{s}\|_1) / b + 1}(\bb{X})] = \int_{\mathcal{S}_d} f(\bb{x}) K_{\bb{s} / b + 1, (1 - \|\bb{s}\|_1) / b + 1}(\bb{x}) \rd \bb{x}.
    \end{equation}
    Alternatively, notice that if $\bb{\xi}_{\bb{s}}\sim \mathrm{Dirichlet}(\bb{s} / b + 1, (1 - \|\bb{s}\|_1) / b + 1)$, then we also have the representation
    \begin{equation}\label{eq:f.b.star.representation.2}
        f_b(\bb{s}) = \EE[f(\bb{\xi}_{\bb{s}})].
    \end{equation}

The asymptotics of the pointwise bias and variance for Beta kernel estimators were first computed by \citet{MR1718494}.
The theorem below extends this to the multidimensional setting.

\begin{theorem}[Pointwise bias and variance]\label{thm:bias.var.density}
    Assume that \eqref{eq:assump:f.density.2} holds.
    We have, as $n\to \infty$ and uniformly for $\bb{s}\in \mathcal{S}_d$,
    \begin{equation}\label{eq:thm:bias.var.density.eq.bias}
        \BB[\hat{f}_{n,b}(\bb{s})] = f_b(\bb{s}) - f(\bb{s}) = b \, g(\bb{s}) + \oo(b),
    \end{equation}
    where
    \begin{equation}\label{eq:def:b.x}
        g(\bb{s}) \leqdef \sum_{i\in [d]} (1 - (d+1)s_i) \frac{\partial}{\partial s_i} f(\bb{s}) + \frac{1}{2} \sum_{i,j\in [d]} s_i (\ind_{\{i = j\}} - s_j) \frac{\partial^2}{\partial s_i \partial s_j} f(\bb{s}).
    \end{equation}
    Assume that \eqref{eq:assump:f.density} holds instead.
    For every subset of indices $\mathcal{J}\subseteq [d]$, denote
    \begin{equation}\label{eq:def.psi}
        \psi(\bb{s}) \leqdef \psi_{\emptyset}(\bb{s}) \quad \text{and} \quad \psi_{\mathcal{J}}(\bb{s}) \leqdef \bigg[(4\pi)^{d - |\mathcal{J}|} \cdot (1 - \|\bb{s}\|_1) \prod_{i\in [d]\backslash\mathcal{J}} s_i\bigg]^{-1/2}.
    \end{equation}
    Then, for any $\bb{s}\in \mathrm{Int}(\mathcal{S}_d)$, any subset $\emptyset \neq \mathcal{J}\subseteq [d]$, and any $\bb{\kappa}\in (0,\infty)^d$, we have, as $n\to \infty$,
    \begin{equation}\label{eq:thm:bias.var.density.eq.var}
        \VV(\hat{f}_{n,b}(\bb{s})) =
        \begin{cases}
            n^{-1} b^{-d/2} \cdot (\psi(\bb{s}) f(\bb{s}) + \OO_{\bb{s}}(b^{1/2})), &\mbox{if } s_i / b \to \infty ~\forall i\in [d] ~\text{and } (1 - \|\bb{s}\|_1) / b \to \infty, \\[3mm]
            n^{-1} b^{-(d + |\mathcal{J}|)/2} \cdot \Big(\psi_{\mathcal{J}}(\bb{s}) f(\bb{s}) \prod_{i\in \mathcal{J}} \frac{\Gamma(2\kappa_i + 1)}{2^{2\kappa_i + 1} \Gamma^2(\kappa_i + 1)} + \OO_{\bb{\kappa},\bb{s}}(b^{1/2})\Big), &\mbox{if } s_i / b \to \kappa_i ~\forall i\in \mathcal{J}\hspace{-0.5mm}, \, s_i / b\to \infty ~\forall i\in [d]\backslash \mathcal{J} \\[-0.5mm]
            &\text{and } (1 - \|\bb{s}\|_1) / b \to \infty,
        \end{cases}
    \end{equation}
\end{theorem}

This means that the pointwise variance is $\OO_{\bb{s}}(n^{-1} b^{-d/2})$ in the interior of the simplex and it gets multiplied by a factor $b^{-1/2}$ every time we go near the boundary in one of the $d$ dimensions. If we are near an edge of dimension $d - |\mathcal{J}|$, then the pointwise variance is $\OO_{\bb{s}}(n^{-1} b^{-(d + |\mathcal{J}|)/2})$.

\begin{corollary}[Mean squared error]\label{cor:bias.var.implies.MSE.density}
    Assume that \eqref{eq:assump:f.density.2} holds.
    We have, as $n\to \infty$ and for each $\bb{s}\in \mathrm{Int}(\mathcal{S}_d)$,
    \begin{equation}
        \begin{aligned}
            \mathrm{MSE}[\hat{f}_{n,b}(\bb{s})]
            &\leqdef \EE\Big[\big|\hat{f}_{n,b}(\bb{s}) - f(\bb{s})\big|^2\Big] \\
            &= \VV(\hat{f}_{n,b}(\bb{s})) + \big(\BB[\hat{f}_{n,b}(\bb{s})]\big)^2 = n^{-1} b^{-d/2} \psi(\bb{s}) f(\bb{s}) + b^2 g^2(\bb{s}) + \OO_{\bb{s}}(n^{-1} b^{-d/2 + 1/2}) + \oo(b^2).
        \end{aligned}
    \end{equation}
    In particular, if $f(\bb{s}) \cdot g(\bb{s}) \neq 0$, the asymptotically optimal choice of $b$, with respect to $\mathrm{MSE}$, is
    \begin{equation}\label{eq:b.opt.MSE}
        b_{\mathrm{opt}}(\bb{s}) = n^{-2/(d+4)} \left[\frac{d}{4} \cdot \frac{\psi(\bb{s}) f(\bb{s})}{g^2(\bb{s})}\right]^{2/(d+4)},
    \end{equation}
    with
    \begin{equation}
        \begin{aligned}
            \mathrm{MSE}[\hat{f}_{n,b_{\mathrm{opt}}}]
            &= n^{-4 / (d+4)} \left[\frac{1 + \frac{d}{4}}{\big(\frac{d}{4}\big)^{\frac{d}{d+4}}}\right] \frac{\big(\psi(\bb{s}) f(\bb{s})\big)^{4 / (d+4)}}{\big(g^2(\bb{s})\big)^{-d / (d+4)}} + \oo_{\bb{s}}(n^{-4/(d+4)}).
        \end{aligned}
    \end{equation}
    More generally, if $n^{2/(d+4)} \, b\to \lambda$ for some $\lambda > 0$ as $n\to \infty$, then
    \begin{equation}
        \begin{aligned}
            \mathrm{MSE}[\hat{f}_{n,b}(\bb{s})]
            &= n^{-4 / (d+4)} \big[\lambda^{-d/2} \psi(\bb{s}) f(\bb{s}) + \lambda^2 g^2(\bb{s})\big] + \oo_{\bb{s}}(n^{-4/(d+4)}).
        \end{aligned}
    \end{equation}
\end{corollary}

By integrating the MSE and showing that the contribution coming from points near the boundary is negligible, we obtain the following result.

\begin{theorem}[Mean integrated squared error]\label{thm:MISE.optimal.density}
    Assume that \eqref{eq:assump:f.density.2} holds.
    We have, as $n\to \infty$,
    \begin{equation}\label{def:MISE.density}
        \begin{aligned}
            \mathrm{MISE}[\hat{f}_{n,b}] \leqdef \int_{\mathcal{S}_d} \EE\Big[\big|\hat{f}_{n,b}(\bb{s}) - f(\bb{s})\big|^2\Big] \rd \bb{s} = n^{-1} b^{-d/2} \int_{\mathcal{S}_d} \psi(\bb{s}) f(\bb{s}) \rd \bb{s} + b^2 \int_{\mathcal{S}_d} g^2(\bb{s}) \rd \bb{s} + \oo(n^{-1} b^{-d/2}) + \oo(b^2).
        \end{aligned}
    \end{equation}
    In particular, if $\int_{\mathcal{S}_d} g^2(\bb{s}) \rd \bb{s} > 0$, the asymptotically optimal choice of $b$, with respect to $\mathrm{MISE}$, is
    \begin{equation}\label{eq:b.opt.MISE}
        b_{\mathrm{opt}} = n^{-2/(d+4)} \left[\frac{d}{4} \cdot \frac{\int_{\mathcal{S}_d} \psi(\bb{s}) f(\bb{s}) \rd \bb{s}}{\int_{\mathcal{S}_d} g^2(\bb{s}) \rd \bb{s}}\right]^{2/(d+4)},
    \end{equation}
    with
    \begin{equation}
        \begin{aligned}
            \mathrm{MISE}[\hat{f}_{n,b_{\mathrm{opt}}}]
            &= n^{-4 / (d+4)} \left[\frac{1 + \frac{d}{4}}{\big(\frac{d}{4}\big)^{\frac{d}{d+4}}}\right] \frac{\big(\int_{\mathcal{S}_d} \psi(\bb{s}) f(\bb{s}) \rd \bb{s}\big)^{4 / (d+4)}}{\big(\int_{\mathcal{S}_d} g^2(\bb{s}) \rd \bb{s}\big)^{-d / (d+4)}} + \oo(n^{-4/(d+4)}).
        \end{aligned}
    \end{equation}
    More generally, if $n^{2/(d+4)} \, b \to \lambda$ for some $\lambda > 0$ as $n\to \infty$, then
    \begin{equation}
        \begin{aligned}
            \mathrm{MISE}[\hat{f}_{n,b}]
            &= n^{-4 / (d+4)} \left[\lambda^{-d/2} \int_{\mathcal{S}_d} \psi(\bb{s}) f(\bb{s}) \rd \bb{s} + \lambda^2 \int_{\mathcal{S}_d} g^2(\bb{s}) \rd \bb{s}\right] + \oo(n^{-4/(d+4)}).
        \end{aligned}
    \end{equation}
\end{theorem}

The following theorem is the analogue of the $L^1$ asymptotics first proved for traditional univariate kernel estimators by \citet{MR955204}, and then extended to the multivariate setting by \citet{MR1118245}.
In the context of Beta kernels, the result can be found in Theorem~1 of \cite{MR1985506}.
As pointed out by \citet[Section~2.3.2]{MR3329609}, the MIAE enjoys many advantages over the MISE.
It puts more emphasis on the tails of the target density, it is a dimensionless quantity, it is invariant to monotone changes of scale, and it is uniformly bounded (by $2$).
For a thorough study of the $L^1$ point of view in the kernel smoothing theory, we refer the reader to \citet{MR780746}.

\begin{theorem}[Mean integrated absolute error]\label{thm:asymptotic.L1.bound}
    Assume that \eqref{eq:assump:f.density.2} holds.
    We have, as $n\to \infty$,
    \begin{equation}\label{eq:L1.asymp}
        \begin{aligned}
            \mathrm{MIAE}[\hat{f}_{n,b}]
            &\leqdef \int_{\mathcal{S}_d} \EE\big|\hat{f}_{n,b}(\bb{s}) - f(\bb{s})\big| \rd \bb{s} \\
            &= \int_{\mathcal{S}_d} w(\bb{s}) \, \EE\hspace{0.3mm}\Bigg[\bigg|Z - \frac{b \, g(\bb{s})}{w(\bb{s})}\bigg|\Bigg] \rd \bb{s} + \OO(n^{-1} b^{-d/2}) + \oo(n^{-1/2} b^{-d/4}) + \oo(b),
        \end{aligned}
    \end{equation}
    where $w(\bb{s}) \leqdef n^{-1/2} b^{-d/4} \sqrt{\psi(\bb{s}) f(\bb{s})}$, $\psi$ and $g$ are defined in \eqref{eq:def.psi} and \eqref{eq:def:b.x}, and $Z\sim \mathcal{N}(0,1)$.
    If $n^{1/2} b^{\hspace{0.2mm}d/4}\to \infty$, then we have the bound
    \begin{equation}\label{eq:L1.asymp.bound}
        \begin{aligned}
            \mathrm{MIAE}[\hat{f}_{n,b}]
            &\leq n^{-1/2} b^{-d/4} \sqrt{\frac{2}{\pi}} \int_{\mathcal{S}_d} \sqrt{\psi(\bb{s}) f(\bb{s})} \rd \bb{s} + b \int_{\mathcal{S}_d} |g(\bb{s})| \rd \bb{s} + \oo(n^{-1/2} b^{-d/4}) + \oo(b),
        \end{aligned}
    \end{equation}
    In particular, if $\int_{\mathcal{S}_d} |g(\bb{s})| \rd \bb{s} > 0$, the asymptotically optimal choice of $b$, with respect to the mean integrated absolute error bound \eqref{eq:L1.asymp.bound}, is
    \begin{equation}\label{eq:b.opt.L1}
        b_{\mathrm{opt}} = n^{-2/(d+4)} \left[\frac{d}{4} \sqrt{\frac{2}{\pi}} \cdot \frac{\int_{\mathcal{S}_d} \sqrt{\psi(\bb{s}) f(\bb{s})} \rd \bb{s}}{\int_{\mathcal{S}_d} |g(\bb{s})| \rd \bb{s}}\right]^{4/(d+4)},
    \end{equation}
    with
    \begin{equation}
        \begin{aligned}
            \mathrm{MIAE}[\hat{f}_{n,b_{\mathrm{opt}}}]
            &\leq n^{-2 / (d+4)} \left[\frac{1 + \frac{d}{4}}{\big(\frac{d}{4}\big)^{\frac{d}{d+4}}}\right] \frac{\big(\sqrt{\frac{2}{\pi}} \int_{\mathcal{S}_d} \sqrt{\psi(\bb{s}) f(\bb{s})} \rd \bb{s}\big)^{4 / (d+4)}}{\big(\int_{\mathcal{S}_d} |g(\bb{s})| \rd \bb{s}\big)^{-d / (d+4)}} + \oo(n^{-2/(d+4)}).
        \end{aligned}
    \end{equation}
    More generally, if $n^{2/(d+4)} \, b \to \lambda$ for some $\lambda > 0$ as $n\to \infty$, then
    \begin{equation}
        \begin{aligned}
            \mathrm{MIAE}[\hat{f}_{n,b}]
            &\leq n^{-2 / (d+4)} \left[\lambda^{-d/4} \int_{\mathcal{S}_d} \sqrt{\psi(\bb{s}) f(\bb{s})} \rd \bb{s} + \lambda \int_{\mathcal{S}_d} |g(\bb{s})| \rd \bb{s}\right] + \oo(n^{-2/(d+4)}).
        \end{aligned}
    \end{equation}
\end{theorem}

The uniform strong consistency was proved for traditional multivariate kernel estimators by \citet{MR859633} and for Beta kernel estimators by \citet{MR1985506}.
In order to keep a control on the partial derivatives of the Dirichlet density $K_{\bb{\alpha},\beta}$ with respect to the parameters $\alpha_1,\dots,\alpha_d,\beta$ in our proof, we add a small buffer to the boundary that goes to zero as $n\to \infty$.
Our proof is completely different from the proof of the case $d = 1$ by \citet{MR1985506} (it is not clear how to generalize it) and instead relies on a novel chaining argument.

\begin{theorem}[Uniform strong consistency]\label{thm:Theorem.3.1.Babu.Canty.Chaubey}
    Assume that \eqref{eq:assump:f.density} holds.
    We have, as $n\to \infty$,
    \begin{equation}\label{eq:thm:Theorem.3.1.Babu.Canty.Chaubey}
        \sup_{\bb{s}\in \mathcal{S}_d} |f_b(\bb{s}) - f(\bb{s})| = \OO(b^{1/2}).
    \end{equation}
    Furthermore, for $\delta > 0$, define
    \begin{equation}\label{eq:def.S.delta}
        \mathcal{S}_d(\delta) \leqdef \big\{\bb{s}\in \mathcal{S}_d: 1 - \|\bb{s}\|_1 \geq \delta ~\text{and}~ s_i \geq \delta \, \, \forall i\in [d]\big\}.
    \end{equation}
    Then, if $b^{-d} \leq n$ as $n\to \infty$, we have
    \begin{equation}
        \sup_{\bb{s}\in \mathcal{S}_d(b d)} |\hat{f}_{n,b}(\bb{s}) - f(\bb{s})| = \OO\left(\frac{|\log b| (\log n)^{3/2}}{b^{\hspace{0.2mm}d + 1/2} \sqrt{n}}\right) + \OO(b^{1/2}), \quad \text{a.s.}
    \end{equation}
    In particular, if $|\log b|^2 \, b^{-2d - 1} = \oo(n / (\log n)^3)$ as $n\to \infty$, then
    \begin{equation}
        \sup_{\bb{s}\in \mathcal{S}_d(b d)} |\hat{f}_{n,b}(\bb{s}) - f(\bb{s})|\to 0, \quad \text{a.s.}
    \end{equation}
\end{theorem}

\newpage
A straightforward verification of the Lindeberg condition for double arrays yields the asymptotic normality.
This result was never proved even for Beta kernel estimators.

\begin{theorem}[Asymptotic normality]\label{thm:Theorem.3.2.and.3.3.Babu.Canty.Chaubey}
    Assume that \eqref{eq:assump:f.density} holds.
    Let $\bb{s}\in \mathrm{Int}(\mathcal{S}_d)$ be such that $f(\bb{s}) > 0$.
    If $n^{1/2} b^{\hspace{0.2mm}d/4}\to \infty$ as $n\to \infty$ and $b\to 0$, then
    \begin{equation}\label{eq:thm:Theorem.3.2.and.3.3.Babu.Canty.Chaubey.Prop.1}
        n^{1/2} b^{\hspace{0.2mm}d/4} (\hat{f}_{n,b}(\bb{s}) - f_b(\bb{s})) \stackrel{\mathscr{D}}{\longrightarrow} \mathcal{N}(0,\psi(\bb{s}) f(\bb{s})).
    \end{equation}
    If we also have $n^{1/2} b^{\hspace{0.2mm}d/4 + 1/2}\to 0$ as $n\to \infty$ and $b\to 0$, then \eqref{eq:thm:Theorem.3.1.Babu.Canty.Chaubey} of Theorem~\ref{thm:Theorem.3.1.Babu.Canty.Chaubey} implies
    \begin{equation}
        n^{1/2} b^{\hspace{0.2mm}d/4} (\hat{f}_{n,b}(\bb{s}) - f(\bb{s})) \stackrel{\mathscr{D}}{\longrightarrow} \mathcal{N}(0,\psi(\bb{s}) f(\bb{s})).
    \end{equation}
    Independently of the above rates for $n$ and $b$, if we assume \eqref{eq:assump:f.density.2} instead and $n^{2/(d+4)} \, b \to \lambda$ for some $\lambda > 0$ as $n\to \infty$ and $b\to 0$, then the pointwise bias result in Theorem~\ref{thm:bias.var.density} implies
    \begin{equation}\label{eq:thm:Theorem.3.2.and.3.3.Babu.Canty.Chaubey.Thm.3.3}
        n^{2 / (d+4)} (\hat{f}_{n,b}(\bb{s}) - f(\bb{s})) \stackrel{\mathscr{D}}{\longrightarrow} \mathcal{N}(\lambda \, g(\bb{s}), \lambda^{-d/2} \psi(\bb{s}) f(\bb{s})).
    \end{equation}
\end{theorem}

\begin{remark}
    The rate of convergence for the traditional $d$-dimensional kernel density estimator with i.i.d.\ data and bandwidth $h$ is $\OO_{\bb{s}}(n^{-1/2} h^{-d/2})$ in Theorem~3.1.15 of \citet{MR0740865}, whereas $\hat{f}_{n,b}$ converges at a rate of $\OO_{\bb{s}}(n^{-1/2} b^{-d/4})$.
    Hence, the relation between the bandwidth of $\hat{f}_{n,b}$ and the bandwidth of the traditional multivariate kernel density estimator is $b \approx h^2$.
\end{remark}

\section{Case study: minerals processing}\label{sec:case.study}

Minerals processing is a branch of engineering that deals with the design and optimization of systems for beneficiation of valuable minerals out of ore rock materials. These ores are first comminuted to generate simple particles (formed by the least number of minerals possible, ideally, only one mineral), and then subjected to physical/physico-chemical separation processes. Here, an input stream of mixed particles (or feed) is separated into two or more product streams with, hopefully, purer composition. These processes can be studied at the particle level, in which case one attributes (or estimates) for each particle its probability of landing on each one of the possible output streams, as a function of one or more of its properties: for two output streams and one single property, these functions are called Tromp curves \cite{Tromp_1937}. \citet{doi:10.1016/j.mineng.2019.03.026} generalized the tool in a nonparametric way to two properties $(X,Y)$, by estimating for each output stream $i$ the probability density function $\hat{f}_{i;b_x, b_y}(x,y)$ using a bivariate kernel with bandwidth $(b_x,b_y)$, and then forcing at each point $(x,y)$ the set of estimated densities to sum to one. The procedure is essentially equivalent to linear discriminant analysis where the hypothesis of normality of the covariables is superseded by using kernel density estimates.

\vspace{0mm}
Kernel density estimates on the simplex allow the obvious generalization of this idea to obtain discrimination rules (in general) and multidimensional Tromp maps (in this specific case) for compositional covariables. To illustrate the idea, we use freely available particle data from an Apatite flotation experiment \cite{doi:10.14278/rodare.543}, for which we have available information about 2,825,712 particles, split into two output streams (the value stream and the waste stream). Flotation is a powerful physico-chemical concentration process that is particularly sensitive to the surface composition (i.e., the proportion of each mineral on the surface of a particle). Out of the 25 minerals considered, we formed 5 groups depending on their characteristics and process behaviour in contrast with the value mineral (see \citet{doi:10.1016/j.mineng.2021.107054} for details): ``Apatite'' (the value mineral), ``semi-soluble salts'' (with a similar chemical behavior than Apatite), ``phyllosilicates'' or sheet silicates (with a specific hydrodynamic behaviour), ``other silicates'' (waste, non-floaters), and ``other minerals'' (mostly sulfides, fast floating minerals often polluting the value stream).

\vspace{0mm}
Due to the preliminary comminution step, most of these particles are monomineralic ($>86\%$) or bimineralic ($>98\%$). Three or less minerals show $>99.8\%$ of the particles. Shortly, most of these compositions are plagued with zeros, making it an ideal case for Dirichlet kernel methods (recall that the estimator does not spill over the simplex and it is asymptotically unbiased with an error that is uniform on the simplex by Theorem~\ref{thm:bias.var.density}, including the boundary). We estimated the density (using Equation~\ref{eq:Dirichlet.estimator}) for each $\mathcal{S}_2$ side of the $\mathcal{S}_4$ simplex that involve Apatite, the value mineral. For each ternary diagram, a regular grid of $100\times100$ nodes was constructed, those observations being zero for all three components were filtered out, and a small $\epsilon$ was added to all selected data, i.e., we computed (Equation~\ref{eq:Dirichlet.estimator}) with $\mathbf{X}_i' = \alpha \mathbf{X}_i+\epsilon \mathbf{1}$ with $\alpha$ such that  $\mathbf{1}^{\top}\mathbf{X}_i'=1$, i.e., the data are re-closed to sum to 1. This was done to avoid the collapse of the Dirichlet kernels when an observation is exactly zero. $\epsilon$ was chosen to be half the grid step. This treatment is reasonable both given the observation mechanism of this data \cite{doi:10.14278/rodare.543,doi:10.1016/j.mineng.2021.107054} and the regularity and the bias-stability properties of the estimator proven in Theorem~\ref{thm:bias.var.density} and Theorem~\ref{thm:Theorem.3.1.Babu.Canty.Chaubey}. The resulting density estimates are displayed in Fig.~\ref{fig:concentrate-tailings}, both in raw and log scale, showing a strong concentration at the vertices and sides of the simplex: note the intense color of the pixel nearest to (0,0), corresponding to pure Apatite particles. Notice that the estimated density is strictly contained within the simplex, and it is not spilling over beyond its boundaries: this is particularly important given the strong concentration of data at the sides of the simplex.

\begin{figure}[htb]
    \centering
    \includegraphics[width=0.95\textwidth]{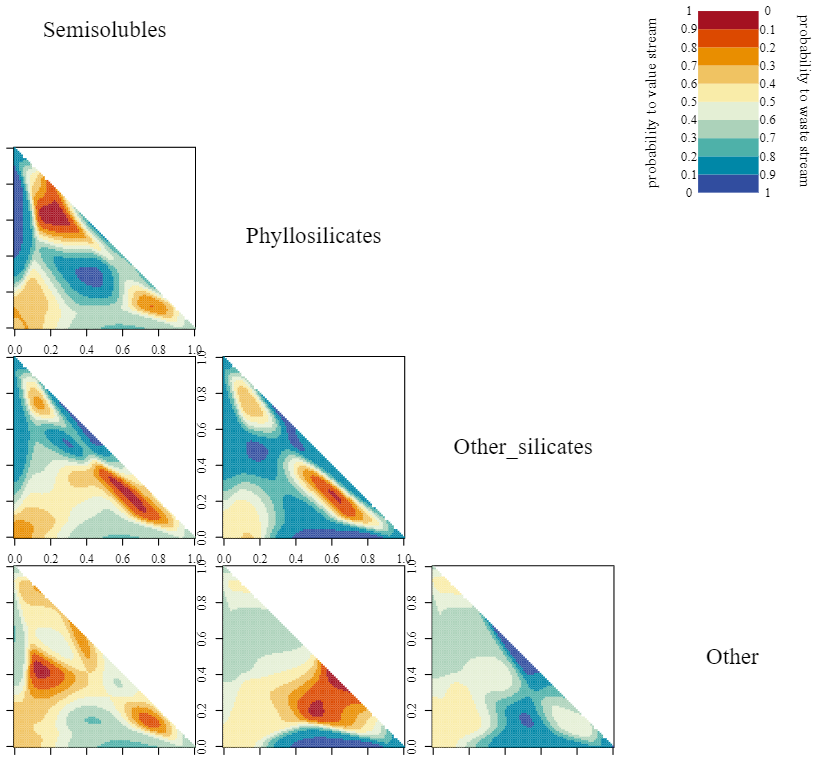}
    \caption{Tromp curves on the simplex for the system Apatite-semisolubles-phyllosilicates-other silicates-other minerals (red: high probability of going to value stream; blue: high probability of going to waste stream).}
    \label{fig:Tromp}
\end{figure}

\begin{figure}
    \centering
    \begin{subfigure}[b]{0.65\textwidth}
        \includegraphics[width=1\linewidth]{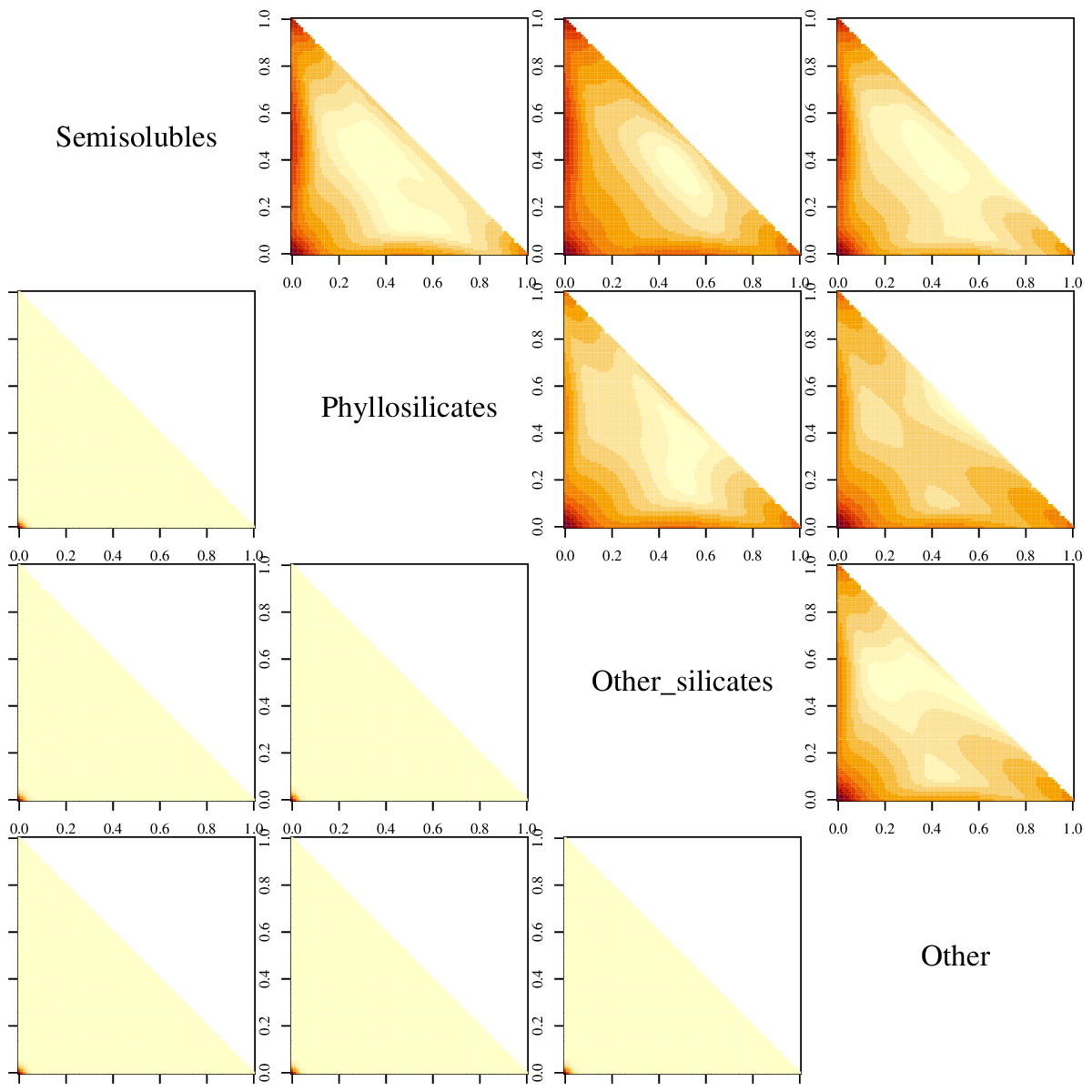}
    \end{subfigure}
    \begin{subfigure}[b]{0.65\textwidth}
        \includegraphics[width=1\linewidth]{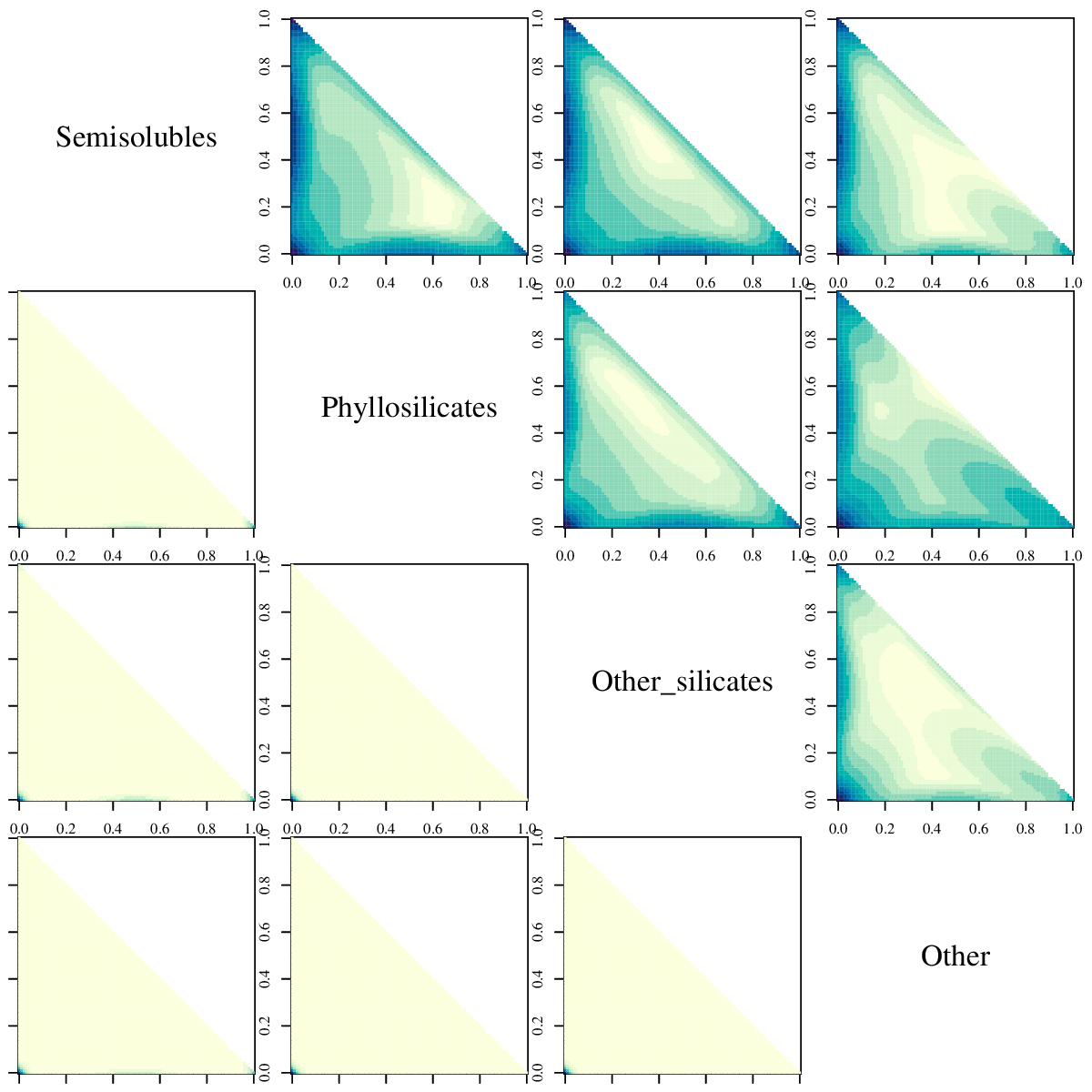}
    \end{subfigure}
    \caption{Dirichlet kernel density estimates of particles on the value stream (\emph{concentrate}, top, yellow to red) and on the waste stream (\emph{tailings}, bottom, yellow to blue): the more intense the red/blue, the higher the density. Upper triangle plots show the log-density, lower triangle plots show the raw density. Each diagram shows data on the simplex (Apatite - variable on the column - variable on the row).}
    \label{fig:concentrate-tailings}
\end{figure}

\newpage
Fig.~\ref{fig:Tromp} shows the results of the calculations for the Tromp simplicial maps, i.e.,
\begin{equation}
    T_b(\bb{s}) = \frac{\hat{f}^{(v)}_{n,b}(\bb{s})}{\hat{f}^{(v)}_{n,b}(\bb{s})+\hat{f}^{(w)}_{n,b}(\bb{s})},
\end{equation}
where the superindexes $^{(v)}$ and $^{(w)}$ represent the density for the value stream and for the waste stream, respectively.
It can be seen that particles formed by less than $\sim 70\%$ Apatite  quickly develop high probabilities of landing in the waste stream. Particles with high proportion of Apatite ($>80\%$) land in the value stream. The complexity of the dependence of the Tromp map on the proportion of phyllosilicates (seen in all diagrams involving this component) suggests the interplay of several factors on top of just the Apatite purity, like, e.g., for particles mostly formed by these sheet silicates, where their hydrodynamics are dominated by their platy shape, and they stay longer in suspension and are transferred to the value stream. Finally, the diagrams involving sulphide minerals (``other'') show that the probabilities of these minerals to report to a specific output stream are much nearer to 0.5-0.5, indicating that the process is not less selective with respect to these minerals.

\newpage
\section{Proofs}\label{sec:proofs}

    \subsection{Proof of Theorem~\ref{thm:bias.var.density}}

    The expression for the pointwise bias follows from the last paragraph in Section~\ref{sec:Dirichlet.kernels}.
    To obtain the asymptotics of the pointwise variance in the interior of the simplex, $\VV(\hat{f}_{n,b}(\bb{s})) = n^{-1} b^{-d/2} \cdot (\psi(\bb{s}) f(\bb{s}) + \OO_{\bb{s}}(b^{1/2}))$, we could again refer to \citet{MR3760293,Kokonendji_Some_2021} if we assumed \eqref{eq:assump:f.density.2}, but here we need a slightly more precise result to get the asymptotics of the MISE later and we work under the weaker assumption \eqref{eq:assump:f.density} that $f$ is Lipschitz continuous on $\mathcal{S}_d$.
    We also want the asymptotics near the boundary, which does not follow from the results of \citet{MR3760293,Kokonendji_Some_2021}.
    For these reasons, we provide a proof below.
    First, note that we can write
    \begin{equation}\label{eq:thm:bias.var.density.begin.variance}
        \hat{f}_{n,b}(\bb{s}) - f_b(\bb{s}) = \frac{1}{n} \sum_{i=1}^n Y_{i,b}(\bb{s}),
    \end{equation}
    where the random variables
    \begin{equation}\label{eq:Y.i.b.random.field}
        Y_{i,b}(\bb{s}) \leqdef K_{\frac{\bb{s}}{b} + 1, \frac{1 - \|\bb{s}\|_1}{b} + 1}(\bb{X}_i) - f_b(\bb{s}), ~~1 \leq i \leq n, \quad \text{are i.i.d.}
    \end{equation}
    Hence, if $\bb{\gamma}_{\bb{s}} \sim \mathrm{Dirichlet}(2\bb{s} / b + 1, 2(1 - \|\bb{s}\|_1) / b + 1)$, then
    \begin{align}\label{eq:thm:bias.var.density.variance.expression}
        \VV(\hat{f}_{n,b}(\bb{s}))
        &= n^{-1} \EE\left[K_{\bb{s} / b + 1, (1 - \|\bb{s}\|_1) / b + 1}(\bb{X})^2\right] - n^{-1} \big(f_b(\bb{s})\big)^2 = n^{-1} A_b(\bb{s}) \, \EE[f(\bb{\gamma}_{\bb{s}})] - \OO(n^{-1}) \notag \\[0.5mm]
        &= n^{-1} A_b(\bb{s}) \, (f(\bb{s}) + \OO(b^{1/2})) - \OO(n^{-1}),
    \end{align}
    where
    \begin{align}\label{eq:def.A.b.s}
        A_b(\bb{s})
        &\leqdef \frac{\Gamma(2(1 - \|\bb{s}\|_1) / b + 1) \prod_{i\in [d]} \Gamma(2 s_i / b + 1)}{\Gamma^2((1 - \|\bb{s}\|_1) / b + 1) \prod_{i\in [d]} \Gamma^2(s_i / b + 1)} \cdot \frac{\Gamma^2(1 / b + d + 1)}{\Gamma(2 / b + d + 1)},
    \end{align}
    and where the last line in \eqref{eq:thm:bias.var.density.variance.expression} follows from the Lipschitz continuity of $f$, the Cauchy-Schwarz inequality and the analogue of \eqref{eq:covariance.explicit.estimate} for $\bb{\gamma}_{\bb{s}}$:
    \begin{align}
        \EE[f(\bb{\gamma}_{\bb{s}})] - f(\bb{s})
        &= \sum_{i\in [d]} \OO\Big(\EE\big[|\gamma_i - s_i|\big]\Big) \leq \sum_{i\in [d]} \OO\bigg(\hspace{-0.5mm}\sqrt{\EE\big[|\gamma_i - s_i|^2\big]}\bigg) = \OO(b^{1/2}).
    \end{align}
    The conclusion of Theorem~\ref{thm:bias.var.density} follows from \eqref{eq:thm:bias.var.density.variance.expression} and Lemma~\ref{lem:A.b.x.asymptotics} below.

    \begin{lemma}\label{lem:A.b.x.asymptotics}
        We have, as $b\to 0$ and uniformly for $\bb{s}\in \mathcal{S}_d$,
        \begin{equation}\label{eq:lem:A.b.x.asymptotics.first.bound}
            0 < A_b(\bb{s}) \leq \frac{b^{(d + 1) / 2} \, (1 / b + d)^{d + 1/2}}{(4\pi)^{d/2} \sqrt{(1 - \|\bb{s}\|_1) \prod_{i\in [d]} s_i}} \, (1 + \OO(b)).
        \end{equation}
        Furthermore, for any subset $\emptyset \neq \mathcal{J}\subseteq [d]$, and any $\bb{\kappa}\in (0,\infty)^d$,
        \begin{equation}
            A_b(\bb{s}) =
            \begin{cases}
                b^{-d/2} \, \psi(\bb{s}) (1 + \OO_{\bb{s}}(b)), &\mbox{if } s_i / b \to \infty ~\forall i\in [d] ~\text{and}~ (1 - \|\bb{s}\|_1) / b \to \infty, \\
                b^{-(d + |\mathcal{J}|)/2} \psi_{\mathcal{J}}(\bb{s}) \prod_{i\in \mathcal{J}} \frac{\Gamma(2\kappa_i + 1)}{2^{2\kappa_i + 1} \Gamma^2(\kappa_i + 1)} \cdot (1 + \OO_{\bb{\kappa},\bb{s}}(b)), &\mbox{if } s_i / b \to \kappa_i ~\forall i\in \mathcal{J} ~\text{and}~ s_i / b\to \infty ~\forall i\in [d]\backslash \mathcal{J} \\[-0.5mm]
                &\text{and}~ (1 - \|\bb{s}\|_1) / b \to \infty,
            \end{cases}
        \end{equation}
        where $\psi$ and $\psi_{\mathcal{J}}$ are defined as in \eqref{eq:def.psi}.
    \end{lemma}

    \begin{proof}[\bf Proof of Lemma~\ref{lem:A.b.x.asymptotics}]
        If we denote
        \begin{equation}
            S_b(\bb{s}) \leqdef \frac{R^2((1 - \|\bb{s}\|_1) / b) \prod_{i\in [d]} R^2(s_i / b)}{R(2(1 - \|\bb{s}\|_1) / b) \prod_{i\in [d]} R(2 s_i / b)} \cdot \frac{R(2 / b + d)}{R^2(1 / b + d)},
        \end{equation}
        where
        \begin{equation}\label{eq:def.R}
            R(z) \leqdef \frac{\sqrt{2\pi} e^{-z} z^{z + 1/2}}{\Gamma(z + 1)}, \quad z\geq 0,
        \end{equation}
        then, for all $\bb{s}\in \mathrm{Int}(\mathcal{S}_d)$, we have
        \begin{align}\label{eq:lem:A.b.x.asymptotics.begin}
            A_b(\bb{s})
            &= \frac{2^{2(1 - \|\bb{s}\|_1) / b + 1/2} \prod_{i\in [d]} 2^{2s_i / b + 1/2}}{(2\pi)^{(d+1)/2} \sqrt{(1 - \|\bb{s}\|_1) / b} \, \prod_{i\in [d]} \sqrt{s_i / b}} \cdot \frac{\sqrt{2\pi} \, e^{-d} (1 / b + d)^{2 / b + 2d + 1}}{(2 / b + d)^{2 / b + d + 1/2}} \cdot S_b(\bb{s}) \notag \\
            &= \frac{b^{(d + 1) / 2} \, (1 / b + d)^{d + 1/2}}{(4\pi)^{d/2} \sqrt{(1 - \|\bb{s}\|_1) \prod_{i\in [d]} s_i}} \cdot \bigg(\frac{2 / b + 2d}{2 / b + d}\bigg)^{2/b + d + 1/2} e^{-d} \cdot S_b(\bb{s}).
        \end{align}
        It well-known that $R(z) < 1$ for all $z\geq 1$, see, e.g., Theorem~2.2 in \cite{MR3684463}.
        We also know that $z\mapsto R(z)$ is increasing on $(1,\infty)$.
        Indeed, by the standard relation $(\Gamma'/\Gamma)(z+1) = 1 / z + (\Gamma'/\Gamma)(z)$ and Lemma~2 in \cite{MR162751}, we have
        \begin{equation*}
            \frac{\rd}{\rd z} \log R(z) = \log z + \frac{1}{2z} - \frac{\Gamma'(z + 1)}{\Gamma(z + 1)} = \log z - \frac{1}{2z} - \frac{\Gamma'(z)}{\Gamma(z)} > 0, \quad \text{for all } z > 1.
        \end{equation*}
        Putting these results on $R$ together, we see that, uniformly for $\bb{s}\in \mathcal{S}_d$,
        \begin{equation}\label{eq:upper.bound.Stirling}
            0 < S_b(\bb{s}) \leq \frac{R(2 / b + d)}{R^2(1 / b + d)}\stackrel{\text{Stirling}}{=} 1 + \OO(b), \quad b\to 0.
        \end{equation}
        Equation \eqref{eq:lem:A.b.x.asymptotics.first.bound} then follows from \eqref{eq:lem:A.b.x.asymptotics.begin}, \eqref{eq:upper.bound.Stirling} and the standard exponential approximation
        \begin{equation}\label{eq:exp.approx}
            \bigg(\frac{2 / b + 2d}{2 / b + d}\bigg)^{2/b + d + 1/2} = \bigg(1 + \frac{d}{2 / b + d}\bigg)^{2/b + d + 1/2} = e^d \, (1 + \OO(b)),
        \end{equation}
        see, e.g., \cite[p.70]{MR0167642}.

        To prove the second claim of the lemma, note that $S_b(\bb{s}) = 1 + \OO_{\bb{s}}(b)$ by Stirling's formula, so as $s_i / b \to \infty ~\forall i\in [d]$ and $(1 - \|\bb{s}\|_1) / b \to \infty$, we have, from \eqref{eq:lem:A.b.x.asymptotics.begin} and \eqref{eq:exp.approx},
        \begin{equation}\label{eq:lem:A.b.x.asymptotics.first.case.end}
            A_b(\bb{s}) = \frac{b^{-d/2} (1 + b d)^{d + 1/2}}{(4\pi)^{d/2} \sqrt{(1 - \|\bb{s}\|_1) \prod_{i\in [d]} s_i}} \cdot (1 + \OO_{\bb{s}}(b)) = \frac{b^{-d/2} (1 + \OO_{\bb{s}}(b))}{(4\pi)^{d/2} \sqrt{(1 - \|\bb{s}\|_1) \prod_{i\in [d]} s_i}}.
        \end{equation}
        Next, let $\emptyset \neq \mathcal{J}\subseteq [d]$ and $\bb{\kappa}\in (0,\infty)^d$.
        If $s_i / b \to \kappa_i$ for all $i\in \mathcal{J}$, $s_i / b\to \infty$ for all $i\in [d]\backslash \mathcal{J}$ and $(1 - \|\bb{s}\|_1) / b \to \infty$, then, from \eqref{eq:def.A.b.s},
        \begin{equation}
            \begin{aligned}
                A_b(\bb{s})
                &= \prod_{i\in \mathcal{J}} \frac{\Gamma(2\kappa_i + 1)}{\Gamma^2(\kappa_i + 1)} \cdot \frac{(1 + \OO_{\bb{\kappa},\bb{s}}(b)) ~ 2^{2(1 - \|\bb{s}\|_1) / b + 1/2} \prod_{i\in [d]\backslash \mathcal{J}} 2^{2s_i / b + 1/2}}{(2\pi)^{(d - |\mathcal{J}| + 1)/2} \sqrt{(1 - \|\bb{s}\|_1) / b} \, \prod_{i\in [d]\backslash \mathcal{J}} \sqrt{s_i / b}} \cdot \frac{\sqrt{2\pi} \, e^{-d} (1 / b + d)^{2 / b + 2d + 1}}{(2 / b + d)^{2 / b + d + 1/2}} \cdot S_b^{\mathcal{J}}(\bb{s}) \\
                &= \prod_{i\in \mathcal{J}} \frac{\Gamma(2\kappa_i + 1)}{\Gamma^2(\kappa_i + 1)} \cdot \frac{(1 + \OO_{\bb{\kappa},\bb{s}}(b)) ~ b^{(d - |\mathcal{J}| + 1)/2} \, (1 / b + d)^{d + 1/2}}{2^{d/2} \prod_{i\in \mathcal{J}} 2^{2\kappa_i + 1/2} (2\pi)^{(d - |\mathcal{J}|)/2} \sqrt{(1 - \|\bb{s}\|_1) \prod_{i\in [d]\backslash \mathcal{J}} s_i}} \cdot \bigg(\frac{2 / b + 2d}{2 / b + d}\bigg)^{2/b + d + 1/2} e^{-d} \cdot S_b^{\mathcal{J}}(\bb{s}),
            \end{aligned}
        \end{equation}
        where
        \begin{equation}
            S_b^{\mathcal{J}}(\bb{s}) \leqdef \frac{R^2((1 - \|\bb{s}\|_1) / b) \prod_{i\in [d]\backslash \mathcal{J}} R^2(s_i / b)}{R(2(1 - \|\bb{s}\|_1) / b) \prod_{i\in [d]\backslash \mathcal{J}} R(2 s_i / b)} \cdot \frac{R(2 / b + d)}{R^2(1 / b + d)}.
        \end{equation}
        Similarly to \eqref{eq:lem:A.b.x.asymptotics.first.case.end}, Stirling's formula and \eqref{eq:exp.approx} imply
        \begin{equation}
            A_b(\bb{s}) = \prod_{i\in \mathcal{J}} \frac{\Gamma(2\kappa_i + 1)}{2^{2\kappa_i + 1} \Gamma^2(\kappa_i + 1)} \cdot \frac{b^{-(d + |\mathcal{J}|)/2} (1 + \OO_{\bb{\kappa},\bb{s}}(b))}{(4\pi)^{(d - |\mathcal{J}|)/2} \sqrt{(1 - \|\bb{s}\|_1) \prod_{i\in [d]\backslash \mathcal{J}} s_i}}.
        \end{equation}
        This concludes the proof of Lemma~\ref{lem:A.b.x.asymptotics} and Theorem~\ref{thm:bias.var.density}.
    \end{proof}

    \subsection{Proof of Theorem~\ref{thm:MISE.optimal.density}}

        By the bound \eqref{eq:lem:A.b.x.asymptotics.first.bound}, the fact that $f$ is uniformly bounded (it is continuous on $\mathcal{S}_d$), the almost-everywhere convergence in \eqref{eq:lem:A.b.x.asymptotics.first.case.end}, and the dominated convergence theorem, we have
        \begin{equation}\label{eq:special.MISE}
            b^{\hspace{0.2mm}d/2} \int_{\mathcal{S}_d} A_b(\bb{s}) f(\bb{s}) \rd \bb{s} = \int_{\mathcal{S}_d} \psi(\bb{s}) f(\bb{s}) \rd \bb{s} + \oo(1).
        \end{equation}
        Therefore, the expressions for the pointwise variance in \eqref{eq:thm:bias.var.density.variance.expression} (using \eqref{eq:special.MISE}) and the pointwise bias in \eqref{eq:thm:bias.var.density.eq.bias} yield
        \begin{equation}
            \mathrm{MISE}[\hat{f}_{n,b}] = \int_{\mathcal{S}_d} \VV(\hat{f}_{n,b}(\bb{s})) + \BB[\hat{f}_{n,b}(\bb{s})]^2 \rd \bb{s} = n^{-1} b^{-d/2} \int_{\mathcal{S}_d} \psi(\bb{s}) f(\bb{s}) \rd \bb{s} + b^2 \int_{\mathcal{S}_d} g^2(\bb{s}) \rd \bb{s} + \oo(n^{-1} b^{-d/2}) + \oo(b^2).
        \end{equation}
        This ends the proof.

    \subsection{Proof of Theorem~\ref{thm:asymptotic.L1.bound}}

        By Lemma~2 in \cite{MR760686}, if $\xi_1,\dots,\xi_n$ is an i.i.d.\ sequence of random variables with $\EE[|\xi_1|^3] < \infty$, then
        \begin{equation}\label{eq:Devroye.bound}
            \sup_{a\in \R} \left|\EE\Big[\big|\overline{\xi}_n - \EE[\overline{\xi}_n] - a \sqrt{\VV(\overline{\xi}_n)}\big|\Big] - \sqrt{\VV(\overline{\xi}_n)} \, \EE\big[|Z - a|\big]\right| \leq \frac{c_0 \, \EE\big[|\xi_1 - \EE[\xi_1]|^3\big]}{n \, \VV(\xi_1)},
        \end{equation}
        where $\overline{\xi}_n \leqdef \frac{1}{n} \sum_{i=1}^n \xi_i$, $Z\sim \mathcal{N}(0,1)$, and $c_0 > 0$ is a universal constant.
        By applying this result with $\xi_i \leqdef K_{\bb{s}/b + 1, (1 - \|\bb{s}\|_1)/b + 1}(\bb{X}_i)$ (here, $\bb{s}\in \mathrm{Int}(\mathcal{S}_d)$ is fixed) and $a^{\star}(\bb{s}) \leqdef (f(\bb{s}) - \EE[\hat{f}_{n,b}(\bb{s})]) / [\VV(\hat{f}_{n,b}(\bb{s}))]^{1/2}$, we can show
        \begin{equation}\label{eq:thm:asymptotic.L1.bound.beginning}
            \left|\EE\Big[\big|\hat{f}_{n,b}(\bb{s}) - f(\bb{s})\big|\Big] - \sqrt{\VV(\hat{f}_{n,b}(\bb{s}))} \, \EE\big[|Z - a^{\star}(\bb{s})|\big]\right| \leq c_1 \, n^{-1} b^{-d/2} \psi(\bb{s}),
        \end{equation}
        for another constant $c_1 = c_1(d) > 0$ that depends only on $d$.
        Indeed, to get the last inequality, note that, as $n\to \infty$,
        \begin{equation}\label{eq:Jensen.third.moment.xi.1}
            \frac{\EE\big[|\xi_1 - \EE[\xi_1]|^3\big]}{\VV(\xi_1)} \leq \frac{4 \, \big\{\EE[\xi_1^3] + (\EE[\xi_1])^3\big\}}{\EE[\xi_1^2] - (\EE[\xi_1])^2} = 4 \cdot \frac{\EE[\xi_1^3]}{\EE[\xi_1^2]} + \OO(1),
        \end{equation}
        by applying Jensen's inequality.
        Similarly to the proof of Theorem~\ref{thm:bias.var.density} (which includes Lemma~\ref{lem:A.b.x.asymptotics}), we have
        \begin{equation}\label{eq:Jensen.third.moment.xi.1.next}
            \frac{\EE[\xi_1^3]}{\EE[\xi_1^2]} = \widetilde{A}_{b}(\bb{s}) \, (1 + \OO(b^{1/2})),
        \end{equation}
        where
        \begin{equation}\label{eq:def.A.b.s.three}
            \widetilde{A}_b(\bb{s}) \leqdef \frac{\Gamma(3(1 - \|\bb{s}\|_1) / b + 1)}{\Gamma(2(1 - \|\bb{s}\|_1) / b + 1) \Gamma((1 - \|\bb{s}\|_1) / b + 1)} \cdot \frac{\prod_{i\in [d]} \Gamma(3 s_i / b + 1)}{\prod_{i\in [d]} \big\{\Gamma(2 s_i / b + 1) \Gamma(s_i / b + 1)\big\}} \cdot \frac{\Gamma(2 / b + d + 1) \Gamma(1 / b + d + 1)}{\Gamma(3 / b + d + 1)}.
        \end{equation}
        Following the first part of the proof of Lemma~\ref{lem:A.b.x.asymptotics}, it is straightforward to show that
        \begin{align}\label{eq:bound.third.moment.xi.1}
            \widetilde{A}_b(\bb{s})
            &\leq \frac{3^{3(1 - \|\bb{s}\|_1) / b + 1/2}}{(2\pi)^{(d+1)/2} 2^{2(1 - \|\bb{s}\|_1) / b + 1/2} \sqrt{(1 - \|\bb{s}\|_1) / b}} \cdot \frac{\prod_{i\in [d]} 3^{3s_i / b + 1/2}}{\prod_{i\in [d]} 2^{2s_i / b + 1/2} \, \prod_{i\in [d]} \sqrt{s_i / b}} \notag \\
            &\quad\cdot \frac{\sqrt{2\pi} \, e^{-d} (2 / b + d)^{2 / b + d + 1/2} (1 / b + d)^{1 / b + d + 1/2}}{(3 / b + d)^{3 / b + d + 1/2}} \cdot (1 + \OO(b)) \notag \\
            &= \frac{b^{(d+1)/2} (3/b + d)^{d + 1/2} \, e^d}{\pi^{d/2} \, 3^{3d/2 + 1/2} \sqrt{(1 - \|\bb{s}\|_1) \prod_{i\in [d]} s_i}} \cdot (1 + \OO(b)) \cdot \bigg(\frac{3 / b + 3d / 2}{3 / b + d}\bigg)^{2/b + d + 1/2} e^{-d} \cdot \bigg(\frac{3 / b + 3d}{3 / b + d}\bigg)^{1/b + d + 1/2} e^{-d} \notag \\
            &= \frac{b^{-d/2} (1 + \OO(b))}{(3 \pi)^{d/2} \sqrt{(1 - \|\bb{s}\|_1) \prod_{i\in [d]} s_i}}.
        \end{align}
        Hence, putting \eqref{eq:Devroye.bound}, \eqref{eq:Jensen.third.moment.xi.1}, \eqref{eq:Jensen.third.moment.xi.1.next} and \eqref{eq:bound.third.moment.xi.1} together proves \eqref{eq:thm:asymptotic.L1.bound.beginning}.

        Now, by \eqref{eq:thm:asymptotic.L1.bound.beginning}, the triangle inequality and the fact that $\psi$ is integrable on $\mathcal{S}_d$ yield
        \begin{equation}\label{eq:thm:asymptotic.L1.bound.next}
            \Bigg|\mathrm{MIAE}[\hat{f}_{n,b}] - \int_{\mathcal{S}_d} w(\bb{s}) \, \EE\bigg|Z - \frac{b \, g(\bb{s})}{w(\bb{s})}\bigg| \rd \bb{s}\Bigg| \leq c_2 \, n^{-1} b^{-d/2} + \int_{\mathcal{S}_d} \left|\sqrt{\VV(\hat{f}_{n,b}(\bb{s}))} \, \EE|Z - a^{\star}(\bb{s})| - w(\bb{s}) \, \EE\bigg|Z - \frac{b \, g(\bb{s})}{w(\bb{s})}\bigg|\right|\rd \bb{s}.
        \end{equation}
        where $w(\bb{s}) \leqdef n^{-1/2} b^{-d/4} \sqrt{\psi(\bb{s}) f(\bb{s})}$ and $c_2 = c_2(d) > 0$ is a constant that depends only on $d$.
        It was shown in Lemma~7 of \citet{MR760686} that, for all $u,w\geq 0$ and all $v,z\in \R$,
        \begin{equation}
            \left|u \, \EE\Big|Z + \frac{v}{u}\Big| - w \, \EE\Big|Z - \frac{z}{w}\Big|\right| \leq \sqrt{\frac{2}{\pi}} \, |u - w| + |v - z|,
        \end{equation}
        so the right-hand side of \eqref{eq:thm:asymptotic.L1.bound.next} is bounded from above by
        \begin{equation}
            \begin{aligned}
                &c_2 \, n^{-1} b^{-d/2} + \int_{\mathcal{S}_d} \bigg|\sqrt{\VV(\hat{f}_{n,b}(\bb{s}))} - \frac{\sqrt{\psi(\bb{s}) f(\bb{s})}}{n^{1/2} \, b^{\hspace{0.2mm}d/4}}\bigg| \rd \bb{s} + \int_{\mathcal{S}_d} \big|\BB[\hat{f}_{n,b}(\bb{s})] - b \, g(\bb{s})\big| \rd \bb{s}.
            \end{aligned}
        \end{equation}
        By the expression \eqref{eq:thm:bias.var.density.variance.expression} for the pointwise variance (using Lemma~\ref{lem:A.b.x.asymptotics}), and the pointwise bias in \eqref{eq:thm:bias.var.density.eq.bias}, the above is $\OO(n^{-1} b^{-d/2}) + \oo(n^{-1/2} b^{-d/4}) + \oo(b)$, which proves \eqref{eq:L1.asymp}. The bound \eqref{eq:L1.asymp.bound} is a direct consequence of \eqref{eq:L1.asymp} together with the trivial bound $\EE|Z - u| \leq \sqrt{2/\pi} + |u|$.
        This ends the proof.

    \subsection{Proof of Theorem~\ref{thm:Theorem.3.1.Babu.Canty.Chaubey}}

    This is the most technical proof, so here is the idea.
    The first three lemmas below bound, uniformly, the Dirichlet density (Lemma~\ref{lem:Dirichlet.density.bound}), the partial derivatives of the Dirichlet density with respect to the parameters $\alpha_1,\dots,\alpha_d$ and $\beta$ (Lemma~\ref{lem:Dirichlet.density.bound.derivatives}), and then the absolute difference of densities (pointwise and under expectations) that have different parameters (Lemma~\ref{lem:differences.of.Dirichlet.densities}).
    This is then used to show continuity estimates for the random field $\bb{s}\mapsto Y_{i,b}(\bb{s})$ from \eqref{eq:Y.i.b.random.field} (Proposition~\ref{prop:continuity.estimate}), meaning that we get a control on the probability that $Y_{i,b}(\bb{s})$ and $Y_{i,b}(\bb{s}')$ are too far apart when $\bb{s}$ and $\bb{s}'$ are close.
    The proof of Proposition~\ref{prop:continuity.estimate} relies on a novel chaining argument.
    From this, we easily deduce large deviation bounds for the supremum of $Y_{i,b}(\bb{s})$ over points $\bb{s}'$ that are inside a small hypercube of width $2b$ centered at $\bb{s}$ (Corollary~\ref{cor:large.deviation}).
    Since $\hat{f}_{n,b}(\bb{s}) - f_b(\bb{s}) = \frac{1}{n} \sum_{i=1}^n Y_{i,b}(\bb{s})$, we can estimate tail probabilities for the supremum of $|\hat{f}_{n,b} - f_b|$ over $\mathcal{S}_d(b d)$ by a union bound over the suprema on a collection of small hypercubes that partitions $\mathcal{S}_d(b d)$ and apply a large deviation bound from Corollary~\ref{cor:large.deviation} to each one of them.

    In the first lemma, we bound the density of the $\mathrm{Dirichlet}\hspace{0.2mm}(\bb{\alpha},\beta)$ distribution from \eqref{eq:Dirichlet.density}.
    \begin{lemma}\label{lem:Dirichlet.density.bound}
        If $\alpha_1, \dots, \alpha_d, \beta \geq 2$, then
        \begin{equation}\label{eq:lem:Dirichlet.density.bound}
            \sup_{\bb{s}\in \mathcal{S}_d} K_{\bb{\alpha},\beta}(\bb{s}) \leq \sqrt{\frac{\|\bb{\alpha}\|_1 + \beta - 1}{(\beta - 1) \prod_{i\in [d]} (\alpha_i - 1)}} ~ (\|\bb{\alpha}\|_1 + \beta - d - 1)^d.
        \end{equation}
    \end{lemma}

    \begin{proof}[\bf Proof of Lemma~\ref{lem:Dirichlet.density.bound}]
        Whenever $\alpha_1, \dots, \alpha_d, \beta \geq 2$, the Dirichlet density $K_{\bb{\alpha},\beta}$ is well-known to maximize at $\bb{s}^{\star} = (\bb{\alpha} - 1) / (\|\bb{\alpha}\|_1 + \beta - d - 1)$.
        At this point, we have
        \begin{equation}\label{eq:Dirichlet.density.at.mode}
            K_{\bb{\alpha},\beta}(\bb{s}^{\star}) = \frac{\Gamma(\|\bb{\alpha}\|_1 + \beta)}{\Gamma(\beta) \prod_{i\in [d]} \Gamma(\alpha_i)} \cdot \frac{(\beta - 1)^{\beta - 1} \prod_{i\in [d]} (\alpha_i - 1)^{\alpha_i - 1}}{(\|\bb{\alpha}\|_1 + \beta - d - 1)^{\|\bb{\alpha}\|_1 + \beta - d - 1}}.
        \end{equation}
        From Theorem~2.2 in \cite{MR3684463}, we also know that, for all $y\geq 2$,
        \begin{equation}
            \sqrt{2\pi} e^{-y + 1} (y - 1)^{y - 1 + \frac{1}{2}} \leq \Gamma(y) \leq \tfrac{7}{5} \cdot \sqrt{2\pi} e^{-y + 1} (y - 1)^{y - 1 + \frac{1}{2}}.
        \end{equation}
        Therefore, \eqref{eq:Dirichlet.density.at.mode} is
        \begin{align}\label{eq:Dirichlet.density.bound}
            &\leq \frac{\frac{7}{5} \sqrt{2\pi} e^{-\|\bb{\alpha}\|_1 - \beta + 1} (\|\bb{\alpha}\|_1 + \beta - 1)^{\|\bb{\alpha}\|_1 + \beta - 1 + \frac{1}{2}}}{\sqrt{2\pi} e^{-\beta + 1} (\beta - 1)^{\beta - 1 + \frac{1}{2}} \prod_{i\in [d]} \sqrt{2\pi} e^{-\alpha_i + 1} (\alpha_i - 1)^{\alpha_i - 1 + \frac{1}{2}}} \cdot \frac{(\beta - 1)^{\beta - 1} \prod_{i\in [d]} (\alpha_i - 1)^{\alpha_i - 1}}{(\|\bb{\alpha}\|_1 + \beta - d - 1)^{\|\bb{\alpha}\|_1 + \beta - d - 1}} \notag \\
            &= \tfrac{7}{5} \, (2\pi)^{-d/2} \cdot e^{-d} \bigg(1 - \frac{d}{\|\bb{\alpha}\|_1 + \beta - 1}\bigg)^{\hspace{-0.5mm}-(\|\bb{\alpha}\|_1 + \beta - 1)} \cdot \sqrt{\frac{\|\bb{\alpha}\|_1 + \beta - 1}{(\beta - 1) \prod_{i\in [d]} (\alpha_i - 1)}} \cdot (\|\bb{\alpha}\|_1 + \beta - d - 1)^d \notag \\
            &\leq \tfrac{7}{5} \, (2\pi)^{-d/2} \cdot e^{\frac{2}{5}d} \cdot \sqrt{\frac{\|\bb{\alpha}\|_1 + \beta - 1}{(\beta - 1) \prod_{i\in [d]} (\alpha_i - 1)}} \cdot (\|\bb{\alpha}\|_1 + \beta - d - 1)^d \notag \\
            &\leq \sqrt{\frac{\|\bb{\alpha}\|_1 + \beta - 1}{(\beta - 1) \prod_{i\in [d]} (\alpha_i - 1)}} \cdot (\|\bb{\alpha}\|_1 + \beta - d - 1)^d,
        \end{align}
        where we used our assumption $\alpha_1, \dots, \alpha_d, \beta \geq 2$ and the fact that $(1 - y)^{-1} \leq e^{\frac{7}{5}y}$ for $y\in [0,1/2]$ to obtain the second inequality.
    \end{proof}

    In the second lemma, we bound the partial derivatives of the $\mathrm{Dirichlet}\hspace{0.2mm}(\bb{\alpha},\beta)$ density with respect to its parameters.

    \begin{lemma}\label{lem:Dirichlet.density.bound.derivatives}
        If $\alpha_1, \dots, \alpha_d, \beta \geq 2$, then for all $\bb{s}\in \mathrm{Int}(\mathcal{S}_d)$,
        \begin{align}
            \left|\frac{\partial}{\partial \alpha_j} K_{\bb{\alpha},\beta}(\bb{s})\right|
            &\leq \Big\{|\log(\|\bb{\alpha}\|_1 + \beta)| + |\log(\alpha_j)| + |\log s_j|\Big\} \cdot \sqrt{\frac{\|\bb{\alpha}\|_1 + \beta - 1}{(\beta - 1) \prod_{i\in [d]} (\alpha_i - 1)}} ~ (\|\bb{\alpha}\|_1 + \beta - d - 1)^d, \label{eq:lem:Dirichlet.density.bound.derivatives.eq.1} \\
            \left|\frac{\partial}{\partial \beta} K_{\bb{\alpha},\beta}(\bb{s})\right|
            &\leq \Big\{|\log(\|\bb{\alpha}\|_1 + \beta)| + |\log(\beta)| + |\log(1 - \|\bb{s}\|_1)|\Big\} \cdot \sqrt{\frac{\|\bb{\alpha}\|_1 + \beta - 1}{(\beta - 1) \prod_{i\in [d]} (\alpha_i - 1)}} ~ (\|\bb{\alpha}\|_1 + \beta - d - 1)^d. \label{eq:lem:Dirichlet.density.bound.derivatives.eq.2}
        \end{align}
    \end{lemma}

    \begin{proof}[\bf Proof of Lemma~\ref{lem:Dirichlet.density.bound.derivatives}]
        The digamma function $\psi(z) \leqdef \Gamma'(z) / \Gamma(z)$ satisfies $|\psi(z)| < |\log(z)|$ for all $z \geq 2$ (see, e.g., Lemma~2 in \cite{MR162751}).
        Hence, for all $j\in [d]$ and all $\bb{s}\in \mathrm{Int}(\mathcal{S}_d)$,
        \begin{align}\label{eq:lem:Dirichlet.density.bound.derivatives.begin}
            \left|\frac{\partial}{\partial \alpha_j} K_{\bb{\alpha},\beta}(\bb{s})\right| = \Big|\big(\psi(\|\bb{\alpha}\|_1 + \beta) - \psi(\alpha_j) + \log s_j\big) \, K_{\bb{\alpha},\beta}(\bb{s})\Big| \leq \Big\{|\log(\|\bb{\alpha}\|_1 + \beta)| + |\log(\alpha_j)| + |\log s_j|\Big\} \, K_{\bb{\alpha},\beta}(\bb{s}).
        \end{align}
        The conclusion \eqref{eq:lem:Dirichlet.density.bound.derivatives.eq.1} follows from Lemma~\ref{lem:Dirichlet.density.bound}.
        The proof of \eqref{eq:lem:Dirichlet.density.bound.derivatives.eq.2} is virtually identical, and thus omitted.
    \end{proof}

    As a consequence of Lemma~\ref{lem:Dirichlet.density.bound.derivatives} and the multivariate mean value theorem, we can control the absolute difference of two Dirichlet densities with different parameters, pointwise and under expectations.

    \begin{lemma}\label{lem:differences.of.Dirichlet.densities}
        If $\alpha_1, \dots, \alpha_d, \beta, \alpha_1', \dots, \alpha_d', \beta' \geq 2$, and $\bb{X}$ is $F$ distributed with a bounded density $f$ supported on $\mathcal{S}_d$, then
        \begin{equation}\label{lem:differences.of.Dirichlet.densities.claim.1}
            \begin{aligned}
                \EE\big[|K_{\bb{\alpha}',\beta'}(\bb{X}) - K_{\bb{\alpha},\beta}(\bb{X})|\big]
                &\leq 3 \, (d + 1) \, \|f\|_{\infty} \sqrt{\tfrac{\|\bb{\alpha} \vee \bb{\alpha}'\|_1 + (\beta \vee \beta') - 1}{((\beta \wedge \beta') - 1) \prod_{i\in [d]} ((\alpha_i \wedge \alpha_i') - 1)}} \cdot \big(\|\bb{\alpha} \vee \bb{\alpha}'\|_1 + (\beta \vee \beta') - d - 1\big)^d \\
                &\quad\cdot \log\big(\|\bb{\alpha} \vee \bb{\alpha}'\|_1 + (\beta \vee \beta')\big) \cdot \|(\bb{\alpha}',\beta') - (\bb{\alpha},\beta)\|_{\infty},
            \end{aligned}
        \end{equation}
        where $\bb{\alpha} \vee \bb{\alpha}' \leqdef (\max\{\alpha_i,\alpha_i'\})_{i\in [d]}$, $\beta \vee \beta' \leqdef \max\{\beta,\beta'\}$, and $\beta \wedge \beta' \leqdef \min\{\beta,\beta'\}$.
        Furthermore, let
        \begin{equation}\label{eq:def.S.delta.second}
            \mathcal{S}_d(\delta) \leqdef \big\{\bb{s}\in \mathcal{S}_d: 1 - \|\bb{s}\|_1 \geq \delta ~\text{and}~ s_i \geq \delta \, \, \forall i\in [d]\big\}, \quad \delta > 0.
        \end{equation}
        Then, for $0 < \delta \leq e^{-1}$, we have
        \begin{equation}\label{lem:differences.of.Dirichlet.densities.claim.2}
            \begin{aligned}
                \max_{\bb{s}\in \mathcal{S}_d(\delta)} |K_{\bb{\alpha}',\beta'}(\bb{s}) - K_{\bb{\alpha},\beta}(\bb{s})|
                &\leq 3 \, (d + 1) \, \|f\|_{\infty} |\log \delta| \cdot \sqrt{\tfrac{\|\bb{\alpha} \vee \bb{\alpha}'\|_1 + (\beta \vee \beta') - 1}{((\beta \wedge \beta') - 1) \prod_{i\in [d]} ((\alpha_i \wedge \alpha_i') - 1)}} \cdot \big(\|\bb{\alpha} \vee \bb{\alpha}'\|_1 + (\beta \vee \beta') - d - 1\big)^d \\
                &\quad\cdot \log\big(\|\bb{\alpha} \vee \bb{\alpha}'\|_1 + (\beta \vee \beta')\big) \cdot \|(\bb{\alpha}',\beta') - (\bb{\alpha},\beta)\|_{\infty}.
            \end{aligned}
        \end{equation}
    \end{lemma}

    \begin{proof}[\bf Proof of Lemma~\ref{lem:differences.of.Dirichlet.densities}]
        By the triangle inequality and the multivariate mean value theorem,
        \begin{equation}
            \EE\big[|K_{\bb{\alpha}',\beta'}(\bb{X}) - K_{\bb{\alpha},\beta}(\bb{X})|\big] \leq \|f\|_{\infty} \int_{\mathrm{Int}(\mathcal{S}_d)} \Bigg\{\left|\Big.\frac{\partial}{\partial \beta} K_{\bb{\alpha},\beta}(\bb{s})\Big|_{(\bb{\alpha},\beta) = (\bb{\alpha}_{\bb{s}},\beta_{\bb{s}})}\right| |\beta - \beta'| + \sum_{j\in [d]} \bigg| \Big.\frac{\partial}{\partial \alpha_j} K_{\bb{\alpha},\beta}(\bb{s})\Big|_{(\bb{\alpha},\beta) = (\bb{\alpha}_{\bb{s}},\beta_{\bb{s}})}\bigg| |\alpha_j - \alpha_j'|\Bigg\} \rd \bb{s},
        \end{equation}
        where, for every $\bb{s}\in \mathrm{Int}(\mathcal{S}_d)$, $(\bb{\alpha}_{\bb{s}},\beta_{\bb{s}})$ is some point on the line segment joining $(\bb{\alpha},\beta)$ and $(\bb{\alpha}',\beta')$. Now, by the estimates in Lemma~\ref{lem:Dirichlet.density.bound.derivatives}, the above is
        \begin{equation}\label{eq:lem:differences.of.Dirichlet.densities.claim.1.eq.end}
            \begin{aligned}
                &\leq \|f\|_{\infty} \sqrt{\tfrac{\|\bb{\alpha} \vee \bb{\alpha}'\|_1 + (\beta \vee \beta') - 1}{((\beta \wedge \beta') - 1) \prod_{i\in [d]} ((\alpha_i \wedge \alpha_i') - 1)}} \cdot \big(\|\bb{\alpha} \vee \bb{\alpha}'\|_1 + (\beta \vee \beta') - d - 1\big)^d \, \|(\bb{\alpha}',\beta') - (\bb{\alpha},\beta)\|_{\infty} \\
                &\qquad\cdot \Bigg\{2 (d + 1) \log\big(\|\bb{\alpha} \vee \bb{\alpha}'\|_1 + (\beta \vee \beta')\big) + \int_{\mathcal{S}_d} |\log(1 - \|\bb{s}\|_1)| \rd \bb{s} + \sum_{j\in [d]} \int_{\mathcal{S}_d} |\log s_j| \rd \bb{s}\Bigg\}.
            \end{aligned}
        \end{equation}
        The integrals are bounded by $1$ since
        \begin{equation}
            \int_{\mathcal{S}_d} |\log(1 - \|\bb{s}\|_1)| \rd \bb{s} = \int_{\mathcal{S}_d} |\log s_j| \rd \bb{s} \leq \int_0^1 |\log s_j| \rd s_j = 1.
        \end{equation}
        Together with \eqref{eq:lem:differences.of.Dirichlet.densities.claim.1.eq.end}, this proves \eqref{lem:differences.of.Dirichlet.densities.claim.1}.
        The proof of the second claim \eqref{lem:differences.of.Dirichlet.densities.claim.2} follows from a simpler argument (without the integrals), and is left to the reader.
    \end{proof}

    \begin{proposition}[Continuity estimates]\label{prop:continuity.estimate}
        Recall from \eqref{eq:Y.i.b.random.field} that
        \begin{equation}\label{eq:def.Y.i.b}
            Y_{i,b}(\bb{s}) \leqdef K_{\frac{\bb{s}}{b} + 1, \frac{1 - \|\bb{s}\|_1}{b} + 1}(\bb{X}_i) - \EE\left[K_{\frac{\bb{s}}{b} + 1, \frac{1 - \|\bb{s}\|_1}{b} + 1}(\bb{X}_i)\right], \quad 1 \leq i \leq n.
        \end{equation}
        Let $\bb{s}\in \mathcal{S}_d(b(d+1))$, $n \geq 1$, $0 < b < (e^{-16\sqrt{2}} \wedge d^{-1})$, $0 < a \leq e^{-1} \|f\|_{\infty} |\log b| / b^{\hspace{0.2mm}d + 1/2}$, and take the unique
        \begin{equation}\label{eq:cond.delta}
            \delta\in (0,e^{-1}] \quad \text{that satisfies} \quad \delta |\log \delta| = \frac{b^{\hspace{0.2mm}d + 1/2} a}{\|f\|_{\infty} |\log b|}.
        \end{equation}
        Then, for all $h\in \R$,
        \begin{equation}\label{eq:prop:continuity.estimate}
            \PP\left(\sup_{\bb{s}'\in \bb{s} + [-b,b]^d} \bigg|\frac{1}{n} \sum_{i=1}^n Y_{i,b}(\bb{s}')\bigg| \geq h + 2a, \bigg|\frac{1}{n} \sum_{i=1}^n Y_{i,b}(\bb{s})\bigg| \leq h\right) \leq C_{f,d} \exp\left(-\frac{1}{100^2d^{\hspace{0.2mm} 4} \|f\|_{\infty}^2} \cdot \bigg(\frac{n^{1/2} \, b^{\hspace{0.2mm}d + 1/2} a}{|\log \delta| \, |\log b|}\bigg)^2\right),
        \end{equation}
        where $C_{f,d} > 0$ is a constant that depends only on the density $f$ and the dimension $d$.
    \end{proposition}

    \begin{proof}[\bf Proof of Proposition~\ref{prop:continuity.estimate}]
        By a union bound, the probability in \eqref{eq:prop:continuity.estimate} is
        \begin{equation}\label{eq:prop:continuity.estimate.begin}
            \begin{aligned}
                &\leq \PP\left(\Bigg\{\sup_{\bb{s}'\in \bb{s} + [-b,b]^d} \bigg|\frac{1}{n} \sum_{i=1}^n \big(Y_{i,b}(\bb{s}') - Y_{i,b}(\bb{s})\big) \ind_{\{\bb{X}_i\in \mathcal{S}_d \backslash \mathcal{S}_d(\delta)\}}\bigg| \geq a\Bigg\} \cap \Bigg\{\sum_{i=1}^n \ind_{\{\bb{X}_i\in \mathcal{S}_d \backslash \mathcal{S}_d(\delta)\}} \leq  n \cdot 4 \, \|f\|_{\infty} \delta\Bigg\}\right) \\
                &\quad+ \PP\left(\sum_{i=1}^n \ind_{\{\bb{X}_i\in \mathcal{S}_d \backslash \mathcal{S}_d(\delta)\}} \geq n \cdot 4 \, \|f\|_{\infty} \delta\right) + \PP\left(\sup_{\bb{s}'\in \bb{s} + [-b,b]^d} \bigg|\frac{1}{n} \sum_{i=1}^n \big(Y_{i,b}(\bb{s}') - Y_{i,b}(\bb{s})\big) \ind_{\{\bb{X}_i\in \mathcal{S}_d(\delta)\}}\bigg| \geq a\right) \reqdef (A) + (B) + (C).
            \end{aligned}
        \end{equation}

        In order to bound the term $(A)$ in \eqref{eq:prop:continuity.estimate.begin}, note that our assumption $\bb{s}\in \mathcal{S}_d(b(d+1))$ and $\bb{s}' = \bb{s} + [-b,b]^d$ imply $\bb{s},\bb{s}'\in \mathcal{S}_d(b)$, which in turn implies
        \begin{equation}
            \alpha_1 = \frac{s_1}{b} + 1, \dots, \, \alpha_d = \frac{s_d}{b} + 1, \, \beta = \frac{1 - \|\bb{s}\|_1}{b} + 1 \geq 2, \qquad \alpha_1' = \frac{s_1'}{b} + 1, \dots, \, \alpha_d' = \frac{s_d'}{b} + 1, \, \beta' = \frac{1 - \|\bb{s}'\|_1}{b} + 1 \geq 2,
        \end{equation}
        and thus
        \begin{equation}\label{eq:tech.eq.term.A}
            \sqrt{\frac{\|\bb{\alpha}\|_1 + \beta - 1}{(\beta - 1) \prod_{i\in [d]} (\alpha_i - 1)}} \leq \sqrt{\|\bb{\alpha}\|_1 + \beta - 1} = \sqrt{b^{-1} + d}, \qquad \sqrt{\frac{\|\bb{\alpha}'\|_1 + \beta' - 1}{(\beta' - 1) \prod_{i\in [d]} (\alpha_i' - 1)}} \leq \sqrt{\|\bb{\alpha}'\|_1 + \beta' - 1} = \sqrt{b^{-1} + d}.
        \end{equation}
        Together with our assumption in \eqref{eq:cond.delta} and the upper bound on the Dirichlet density in Lemma~\ref{lem:Dirichlet.density.bound}, we have, on the event $\big\{\sum_{i=1}^n \ind_{\scriptscriptstyle \{\bb{X}_i\in \mathcal{S}_d \backslash \mathcal{S}_d(\delta)\}} \leq n \cdot 4 \, \|f\|_{\infty} \delta\big\}$,
        \begin{equation}\label{eq:second.term.A}
            \bigg|\frac{1}{n} \sum_{i=1}^n \big(Y_{i,b}(\bb{s}') - Y_{i,b}(\bb{s})\big) \ind_{\{\bb{X}_i\in \mathcal{S}_d \backslash \mathcal{S}_d(\delta)\}}\bigg| \leq 4 \cdot 4 \, \|f\|_{\infty} \delta \cdot b^{-d} \sqrt{b^{-1} + d} \leq \frac{16 \sqrt{1 + b d}}{|\log \delta| \, |\log b|} \, a.
        \end{equation}
        Since $0 < \delta \leq e^{-1}$ and $0 < b < (e^{-16\sqrt{2}} \wedge d^{-1})$ by assumption, the above is $< a$, which means that
        \begin{equation}\label{eq:estimate.prob.A}
            (A) = 0.
        \end{equation}

        The term $(B)$ is the probability that there are ``too many bad observations'' (i.e., too many $\bb{X}_i$'s near the boundary of the simplex, where the partial derivatives of the Dirichlet density with respect to $\alpha_1,\dots,\alpha_d$ and $\beta$ explode).
        We will control this term with a concentration bound.
        First, note that the volume of $\mathcal{S}_d \backslash \mathcal{S}_d(\delta)$ is at most $2 d \delta / d!$.
        Indeed, $\mathcal{S}_d(\delta)$ has the shape of a simplex of side-length $1 - 2\delta$ inside $\mathcal{S}_d$, so
        \begin{equation}\label{eq:volume.simplex}
            d! \cdot \text{Volume}(\mathcal{S}_d\backslash \mathcal{S}_d(\delta)) = 1 - (1 - 2 \delta)^d \leq 1 - (1 + d \cdot (- 2 \delta)) = 2 d \delta,
        \end{equation}
        where we used the inequality $(1 + x)^n \geq 1 + n x$, which valid for all $n\in \N$ and $x\geq -1$.
        From \eqref{eq:volume.simplex} and the fact that $\|f\|_{\infty}$ is finite ($f$ is continuous by assumption and $\mathcal{S}_d$ is compact), we get that
        \begin{equation}
            \EE\big[\ind_{\{\bb{X}_i\in \mathcal{S}_d \backslash \mathcal{S}_d(\delta)\}}\big] \leq \frac{2 \|f\|_{\infty}}{(d-1)!} \, \delta.
        \end{equation}
        By applying Hoeffding's inequality and condition \eqref{eq:cond.delta}, we obtain
        \begin{equation}\label{eq:estimate.prob.B}
            (B) \leq \exp\left(-2n \cdot \bigg((2 (d-1)! - 1) \cdot \frac{2 \|f\|_{\infty}}{(d-1)!} \, \delta\bigg)^2\right) \leq \exp\left(-2 \, \bigg(\frac{n^{1/2} \, b^{\hspace{0.2mm}d + 1/2} a}{|\log \delta| \, |\log b|}\bigg)^2\right).
        \end{equation}

        Now, in order to bound the third probability in \eqref{eq:prop:continuity.estimate.begin}, the main idea of the proof is to decompose the supremum with a chaining argument and apply concentration bounds on the increments at each level of the $d$-dimensional tree.
        With the notation $\mathcal{H}_k \leqdef 2^{-k} \cdot b \, \Z^d$, we have the embedded sequence of lattice points
        \begin{equation}
            \mathcal{H}_0 \subseteq \mathcal{H}_1 \subseteq \dots \subseteq \mathcal{H}_k \subseteq \dots \subseteq \R^d.
        \end{equation}
        Hence, for $\bb{s}\in \mathcal{S}_d(b(d+1))$ fixed, and for any $\bb{s}'\in \bb{s} + [-b,b]^d$, let $(\bb{s}_k)_{k\in \N_0}$ be a sequence that satisfies
        \begin{equation}
            \bb{s}_0 = \bb{s}, ~\quad \bb{s}_k - \bb{s}\in \mathcal{H}_k \cap [-b, b]^d, ~\quad \lim_{k\to\infty} \|\bb{s}_k - \bb{s}'\|_{\infty} = 0,
        \end{equation}
        and
        \begin{equation}
            (\bb{s}_{k+1})_i = (\bb{s}_k)_i \pm 2^{-k-1} b, \quad \text{for all } i\in [d].
        \end{equation}
        Since the map $\bb{s}\mapsto \frac{1}{n} \sum_{i=1}^n Y_{i,b}(\bb{s})$ is almost-surely continuous,
        \begin{equation}
            \bigg|\frac{1}{n} \sum_{i=1}^n \big(Y_{i,b}(\bb{s}') - Y_{i,b}(\bb{s})\big) \ind_{\{\bb{X}_i\in \mathcal{S}_d(\delta)\}}\bigg| \leq \sum_{k=0}^{\infty} \bigg|\frac{1}{n} \sum_{i=1}^n \big(Y_{i,b}(\bb{s}_{k+1}) - Y_{i,b}(\bb{s}_k)\big)  \ind_{\{\bb{X}_i\in \mathcal{S}_d(\delta)\}}\bigg|,
        \end{equation}
        and since, $\sum_{k=0}^{\infty} \frac{1}{2(k+1)^2} \leq 1$, we have the inclusion of events,
        \small
        \begin{equation}
            \left\{\sup_{\bb{s}'\in \bb{s} + [-b,b]^d} \bigg|\frac{1}{n} \sum_{i=1}^n \big(Y_{i,b}(\bb{s}') - Y_{i,b}(\bb{s})\big) \ind_{\{\bb{X}_i\in \mathcal{S}_d(\delta)\}}\bigg| \geq a\right\} \, \subseteq \, \bigcup_{k=0}^{\infty} \hspace{-3mm}\bigcup_{\substack{\bb{s}_k \in \bb{s} + \mathcal{H}_k \cap [-b,b]^d \\ (\bb{s}_{k+1})_i = (\bb{s}_k)_i \pm 2^{-k - 1} b ~\forall i\in [d]}} \hspace{-4mm}\left\{\bigg|\frac{1}{n} \sum_{i=1}^n \big(Y_{i,b}(\bb{s}_{k+1}) - Y_{i,b}(\bb{s}_k)\big) \ind_{\{\bb{X}_i\in \mathcal{S}_d(\delta)\}}\bigg| \geq \frac{a}{2 (k + 1)^2}\right\}.
        \end{equation}
        \normalsize
        By a union bound and the fact that $|\mathcal{H}_k \cap [-b,b]^d| \leq 2^{(k+2)d}$,
        \begin{equation}\label{eq:prop:continuity.estimate.before.Bernstein}
            (C) \leq \sum_{k=0}^{\infty} ~ 2^{(k+2)d} ~\cdot~ 2^d \hspace{-5mm} \sup_{\substack{\bb{s}_k \in \bb{s} + \mathcal{H}_k \cap [-b,b]^d \\ (\bb{s}_{k+1})_i = (\bb{s}_k)_i \pm 2^{-k - 1} b ~\forall i\in [d]}} \PP\left(\bigg|\frac{1}{n} \sum_{i=1}^n \big(Y_{i,b}(\bb{s}_{k+1}) - Y_{i,b}(\bb{s}_k)\big) \ind_{\{\bb{X}_i\in \mathcal{S}_d(\delta)\}}\bigg| \geq \frac{a}{2 (k + 1)^2}\right).
        \end{equation}
        By Azuma's inequality (see, e.g., Theorem~1.3.1 in \cite{MR1422018}), Lemma~\ref{lem:differences.of.Dirichlet.densities} (Note that $\bb{s}\in \mathcal{S}_d(b(d+1))$ and $\bb{s}'\in \bb{s} + [-b,b]^d$ imply $\bb{s}_k\in \mathcal{S}_d(b)$ for all $k\in \N_0$, so that
        \begin{equation*}
            \alpha_1 = \frac{(\bb{s}_k)_1}{b} + 1, \dots, \, \alpha_d = \frac{(\bb{s}_k)_d}{b} + 1, \, \beta = \frac{1 - \|\bb{s}_k\|_1}{b} + 1 \geq 2, \quad \text{for all } k\in \N_0.)
        \end{equation*}
        and \eqref{eq:def.Y.i.b}, the above is
        \begin{align}\label{eq:prop:continuity.estimate.before.Bernstein.next}
            &\leq \sum_{k=0}^{\infty} ~ 2^{(k+3)d} \cdot 2 \, \exp\left(-\frac{n a^2}{8 (k + 1)^4} \cdot \left(25 \, d^{\hspace{0.2mm} 2} \|f\|_{\infty} \frac{|\log \delta| \, |\log b|}{b^{\hspace{0.2mm}d + 1/2} \, 2^{k + 1}}\right)^{\hspace{-0.5mm}-2\hspace{0.5mm}}\right) \notag \\
            &\leq \sum_{k=0}^{\infty} ~ 2^{(k+3)d} \cdot 2 \, \exp\left(-\frac{2^{2k-1}}{25^2 d^{\hspace{0.2mm} 4} \|f\|_{\infty}^2\, (k+1)^4} \cdot \bigg(\frac{n^{1/2} \, b^{\hspace{0.2mm}d + 1/2} a}{|\log \delta| \, |\log b|}\bigg)^2\right).
        \end{align}
        The minimum of $k\mapsto 0.99 \cdot 2^{2k-1} (k+1)^{-4}$ on $\N_0$ is larger than say $1/16$, so we deduce
        \begin{equation}\label{eq:estimate.prob.C}
            (C) \leq C_{f,d} \exp\left(-\frac{1}{100^2 d^{\hspace{0.2mm} 4} \|f\|_{\infty}^2} \cdot \bigg(\frac{n^{1/2} \, b^{\hspace{0.2mm}d + 1/2} a}{|\log \delta| \, |\log b|}\bigg)^2\right),
        \end{equation}
        for some large constant $C_{f,d} > 0$.
        Putting \eqref{eq:estimate.prob.A}, \eqref{eq:estimate.prob.B} and \eqref{eq:estimate.prob.C} together in \eqref{eq:prop:continuity.estimate.begin} concludes the proof of Proposition~\ref{prop:continuity.estimate}.
    \end{proof}

    \begin{corollary}[Large deviation estimates]\label{cor:large.deviation}
        Recall $Y_{i,b}(\bb{s})$ from \eqref{eq:def.Y.i.b}.
        Let $\bb{s}\in \mathcal{S}_d(b(d+1))$, $n \geq 100^6 d^{\hspace{0.2mm} 6}$, $n^{-1/d} \leq b \leq (e^{-16\sqrt{2}} \wedge d^{-1})$, $0 < a \leq e^{-1} \|f\|_{\infty} |\log b| / b^{\hspace{0.2mm}d + 1/2}$, and take the unique
        \begin{equation}\label{eq:cond.delta.2}
            \delta\in (0,e^{-1}] \quad \text{that satisfies} \quad \delta \, |\log \delta| = \frac{b^{\hspace{0.2mm}d + 1/2} a}{\|f\|_{\infty} |\log b|}.
        \end{equation}
        Then, we have
        \begin{equation}\label{eq:cor:large.deviation}
            \PP\left(\sup_{\bb{s}'\in \bb{s} + [-b,b]^d} \bigg|\frac{1}{n} \sum_{i=1}^n Y_{i,b}(\bb{s}')\bigg| \geq 3a\right) \leq C_{f,d} \exp\left(-\frac{1}{100^2d^{\hspace{0.2mm} 4} \|f\|_{\infty}^2} \cdot \bigg(\frac{n^{1/2} \, b^{\hspace{0.2mm}d + 1/2} a}{|\log \delta| \, |\log b|}\bigg)^2\right),
        \end{equation}
        where $C_{f,d} > 0$ is a constant that depends only on the density $f$ and the dimension $d$.
    \end{corollary}

    \begin{proof}[\bf Proof of Corollary~\ref{cor:large.deviation}]
        By a union bound, the probability in \eqref{eq:cor:large.deviation} is
        \begin{equation}
            \leq \PP\left(\sup_{\bb{s}'\in \bb{s} + [-b,b]^d} \bigg|\frac{1}{n} \sum_{i=1}^n Y_{i,b}(\bb{s}')\bigg| \geq 3a, \bigg|\frac{1}{n} \sum_{i=1}^n Y_{i,b}(\bb{s})\bigg| \leq a\right) + \PP\left(\bigg|\frac{1}{n} \sum_{i=1}^n Y_{i,b}(\bb{s})\bigg| \geq a\right).
        \end{equation}
        The first probability is bounded using Proposition~\ref{prop:continuity.estimate}.
        We get the same bound on the second probability by applying Azuma's inequality and Lemma~\ref{lem:differences.of.Dirichlet.densities}, as we did in \eqref{eq:prop:continuity.estimate.before.Bernstein.next}.
    \end{proof}

    We are now ready to prove Theorem~\ref{thm:Theorem.3.1.Babu.Canty.Chaubey}.
    On the one hand, the Lipschitz continuity of $f$, Jensen's inequality and \eqref{eq:covariance.explicit.estimate}, imply that, uniformly for $\bb{s}\in \mathcal{S}_d$,
    \begin{equation}\label{eq:thm:Theorem.3.1.Babu.Canty.Chaubey.control.T.m}
        f_b(\bb{s}) - f(\bb{s}) = \EE[f(\bb{\xi}_{\bb{s}})] - f(\bb{s}) = \sum_{i\in [d]} \OO\Big(\EE\big[|\xi_i - s_i|\big]\Big) \leq \sum_{i\in [d]} \OO\bigg(\hspace{-0.5mm}\sqrt{\EE\big[|\xi_i - s_i|^2\big]}\bigg) = \OO(b^{1/2}).
    \end{equation}
    On the other hand, recall from \eqref{eq:thm:bias.var.density.begin.variance} that
    \begin{equation}\label{eq:f.n.b.minus.f.b.star}
        \hat{f}_{n,b}(\bb{s}) - f_b(\bb{s}) = \frac{1}{n} \sum_{i=1}^n Y_{i,b}(\bb{s}).
    \end{equation}
    By a union bound over the suprema on hypercubes of width $2b$ centered at each $\bb{s}\in 2b \, \mathbb{Z}^d \cap \mathcal{S}_d(b(d+1))$, and the large deviation estimates in Corollary~\ref{cor:large.deviation} with
    \begin{equation}\label{eq:choice.a}
        a = 100 \, d^{\hspace{0.2mm} 2} \frac{(\log n)^{3/2}}{\sqrt{n}} \cdot \frac{\|f\|_{\infty} |\log b|}{b^{\hspace{0.2mm}d + 1/2}}
    \end{equation}
    (the upper bound condition on $a$ is satisfied as long as $100 \, d^{\hspace{0.2mm} 2} (\log n)^{3/2} / \sqrt{n} \leq e^{-1}$, which is valid if $n\geq 100^6 d^{\hspace{0.2mm} 6}$ for example)
    and the unique $\delta\in (0,e^{-1}]$ that satisfies
    \begin{equation}\label{eq:cond.delta.verif}
        \delta |\log \delta| = \frac{b^{\hspace{0.2mm}d + 1/2} a}{\|f\|_{\infty} |\log b|} \stackrel{\eqref{eq:choice.a}}{=} 100 \, d^{\hspace{0.2mm} 2} \frac{(\log n)^{3/2}}{\sqrt{n}},
    \end{equation}
    we have
    \begin{align}\label{eq:thm:Theorem.3.1.Babu.Canty.Chaubey.eq.end}
        \PP\left(\sup_{\bb{s}\in \mathcal{S}_d(b d)} |\hat{f}_{n,b}(\bb{s}) - f_b(\bb{s})| \geq 3a\right)
        &\leq \sum_{\bb{s}\in 2b \, \mathbb{Z}^d \cap \mathcal{S}_d(b(d+1))} \hspace{-3mm} \PP\left(\sup_{\bb{s}'\in \bb{s} + [-b,b]^d} \Big|\frac{1}{n} \sum_{i=1}^n Y_{i,b}(\bb{s}')\Big| \geq 3a\right) \notag \\[-0.5mm]
        &\leq b^{-d} \cdot C_{f,d} \exp\left(-\frac{1}{100^2d^{\hspace{0.2mm} 4} \|f\|_{\infty}^2} \cdot \bigg(\frac{n^{1/2} \, b^{\hspace{0.2mm}d + 1/2} a}{|\log \delta| \, |\log b|}\bigg)^2\right) \leq b^{-d} \cdot C_{f,d} \exp\left(- \frac{(\log n)^3}{|\log \delta|^2}\right).
    \end{align}
    The condition imposed on $\delta$ in \eqref{eq:cond.delta.verif} implies
    \begin{equation}\label{eq:result.for.delta}
        n^{-1/2} \leq \delta \leq e^{-1}, \quad (\text{and thus } |\log \delta|\leq \tfrac{1}{2} \log n)
    \end{equation} because the function $x\mapsto x |\log x|$ is increasing on $(0,e^{-1}]$.
    Using \eqref{eq:result.for.delta} in \eqref{eq:thm:Theorem.3.1.Babu.Canty.Chaubey.eq.end}, we get
    \begin{equation}
        \PP\left(\sup_{\bb{s}\in \mathcal{S}_d(b d)} |\hat{f}_{n,b}(\bb{s}) - f_b(\bb{s})| \geq 3a\right) \leq C_{f,d} \exp\left(d |\log b| - 4 \log n\right).
    \end{equation}
    Since we assumed that $b \geq n^{-1/d}$, the above is $\leq C_{f,d} \, n^{-3}$, which is summable.
    By our choice of $a$ in \eqref{eq:choice.a} and the Borel-Cantelli lemma, we obtain
    \begin{equation}
        \sup_{\bb{s}\in \mathcal{S}_d(b d)} |\hat{f}_{n,b}(\bb{s}) - f_b(\bb{s})| = \OO\left(\frac{|\log b| (\log n)^{3/2}}{b^{\hspace{0.2mm}d + 1/2} \sqrt{n}}\right), \quad \text{a.s.}
    \end{equation}
    Together with \eqref{eq:thm:Theorem.3.1.Babu.Canty.Chaubey.control.T.m}, the conclusion follows.

    \subsection{Proof of Theorem~\ref{thm:Theorem.3.2.and.3.3.Babu.Canty.Chaubey}}

        By \eqref{eq:f.n.b.minus.f.b.star}, the asymptotic normality of $n^{1/2} b^{\hspace{0.2mm}d/4} (\hat{f}_{n,b}(\bb{s}) - f_b(\bb{s}))$ will be proved if we verify the following Lindeberg condition for double arrays (see, e.g., Section~1.9.3 in \cite{MR0595165}):
        For every $\e > 0$,
        \begin{equation}\label{eq:prop:Proposition.1.Babu.Canty.Chaubey.Lindeberg.condition}
            s_b^{-2} \, \EE\left[|Y_{1,b}(\bb{s})|^2 \, \ind_{\{|Y_{1,b}(\bb{s})| > \e n^{1/2} s_b\}}\right] \longrightarrow 0, \quad n\to \infty,
        \end{equation}
        where $s_b^2 \leqdef \EE\big[|Y_{1,b}(\bb{s})|^2\big]$ and $b = b(n)\to 0$.
        From Lemma~\ref{lem:Dirichlet.density.bound}, we know that
        \begin{equation}
            |Y_{1,b}(\bb{s})| = \OO\big(\psi(\bb{s}) \, b^{\hspace{0.2mm}d/2} \cdot b^{-d}\big) = \OO_{\bb{s}}(b^{-d/2}),
        \end{equation}
        and we also know that $s_b = b^{-d/4} \sqrt{\psi(\bb{s}) f(\bb{s})} \, (1 + \oo_{\bb{s}}(1))$ when $f$ is Lipschitz continuous, by the proof of Theorem~\ref{thm:bias.var.density}, so
        \begin{equation}\label{eq:prop:Proposition.1.Babu.Canty.Chaubey.Lindeberg.condition.verify}
            \frac{|Y_{1,b}(\bb{s})|}{n^{1/2} s_b} = \OO_{\bb{s}}(n^{-1/2} \, b^{\hspace{0.2mm}d/4} \, b^{-d/2}) = \OO_{\bb{s}}(n^{-1/2} b^{-d/4}) \longrightarrow 0,
        \end{equation}
        whenever $n^{1/2} b^{\hspace{0.2mm}d/4}\to \infty$ as $n\to \infty$ and $b\to 0$.
        Under this condition, \eqref{eq:prop:Proposition.1.Babu.Canty.Chaubey.Lindeberg.condition} holds (since for any given $\e > 0$, the indicator function is equal to $0$ for $n$ large enough, independently of $\omega$) and thus
        \begin{equation}
            \begin{aligned}
                n^{1/2} b^{\hspace{0.2mm}d/4} (\hat{f}_{n,b}(\bb{s}) - f_b(\bb{s}))
                &= n^{1/2} b^{\hspace{0.2mm}d/4} \cdot \frac{1}{n} \sum_{i=1}^n Y_{i,m} \stackrel{\mathscr{D}}{\longrightarrow} \mathcal{N}(0,\psi(\bb{s}) f(\bb{s})).
            \end{aligned}
        \end{equation}
        This ends the proof.

\section*{Acknowledgments}

F.\ Ouimet is supported by a postdoctoral fellowship from the NSERC (PDF) and the FRQNT (B3X supplement).
We thank the Editor, the Associate Editor and the referees for their insightful remarks which led to improvements in the presentation of this paper.

\section*{Author contributions}

\begin{itemize}\setlength\itemsep{0em}
    \item F.\ Ouimet: writing of the original draft and editing, review of the literature, conceptualization, theoretical results and proofs; responsible for Sections~\ref{sec:overview}, \ref{sec:Dirichlet.kernels}, \ref{sec:asymptotic.properties} and \ref{sec:proofs}, and parts of Section~\ref{sec:introduction}.
    \item R.\ Tolosana-Delgado: writing of the case study and the practical motivations in the introduction; responsible for Section~\ref{sec:case.study} and parts of Section~\ref{sec:introduction}.
\end{itemize}

\begin{appendices}
    \section{Supplementary data}
    \texttt{R} codes related to this article can be found online at \url{https://doi.org/10.1016/j.jmva.2021.104832}.
\end{appendices}

%
%

\phantomsection
\addcontentsline{toc}{chapter}{References}

\bibliographystyle{myjmva}
\bibliography{Ouimet_2021_review_Bernstein_and_asymmetric_kernels_bib}

\begin{thebibliography}{193}
\expandafter\ifx\csname natexlab\endcsname\relax\def\natexlab#1{#1}\fi
\providecommand{\bibinfo}[2]{#2}
\ifx\xfnm\relax \def\xfnm[#1]{\unskip,\space#1}\fi
\bibitem[{Abdous and Kokonendji(2009)}]{MR2511101}
\bibinfo{author}{B.~Abdous}, \bibinfo{author}{C.~C. Kokonendji},
  \bibinfo{title}{Consistency and asymptotic normality for discrete
  associated-kernel estimator}, \bibinfo{journal}{Afr. Diaspora J. Math.
  (N.S.)} \bibinfo{volume}{8} (\bibinfo{year}{2009}) \bibinfo{pages}{63--70}.
  \bibinfo{note}{\href{http://www.ams.org/mathscinet-getitem?mr=MR2511101}{MR2511101}}.
\bibitem[{Abramowitz and Stegun(1964)}]{MR0167642}
\bibinfo{author}{M.~Abramowitz}, \bibinfo{author}{I.~A. Stegun},
  \bibinfo{title}{Handbook of {M}athematical {F}unctions with {F}ormulas,
  {G}raphs, and {M}athematical {T}ables}, volume~\bibinfo{volume}{55} of
  \text{\bibinfo{series}{National Bureau of Standards Applied Mathematics
  Series}}, \bibinfo{publisher}{For sale by the Superintendent of Documents,
  U.S. Government Printing Office, Washington, D.C.}, \bibinfo{year}{1964}.
  \bibinfo{note}{\href{http://www.ams.org/mathscinet-getitem?mr=MR0167642}{MR0167642}}.
\bibitem[{Aitchison and Lauder(1985)}]{doi:10.2307/2347365}
\bibinfo{author}{J.~Aitchison}, \bibinfo{author}{I.~J. Lauder},
  \bibinfo{title}{Kernel density estimation for compositional data},
  \bibinfo{journal}{J. Roy. Statist. Soc. Ser. C} \bibinfo{volume}{34}
  (\bibinfo{year}{1985}) \bibinfo{pages}{129--137}.
  \bibinfo{note}{\href{https://doi.org/10.2307/2347365}{doi:10.2307/2347365}}.
\bibitem[{Babu et~al.(2002)Babu, Canty and Chaubey}]{MR1910059}
\bibinfo{author}{G.~J. Babu}, \bibinfo{author}{A.~J. Canty},
  \bibinfo{author}{Y.~P. Chaubey}, \bibinfo{title}{Application of {B}ernstein
  polynomials for smooth estimation of a distribution and density function},
  \bibinfo{journal}{J. Statist. Plann. Inference} \bibinfo{volume}{105}
  (\bibinfo{year}{2002}) \bibinfo{pages}{377--392}.
  \bibinfo{note}{\href{http://www.ams.org/mathscinet-getitem?mr=MR1910059}{MR1910059}}.
\bibitem[{Babu and Chaubey(2006)}]{MR2270097}
\bibinfo{author}{G.~J. Babu}, \bibinfo{author}{Y.~P. Chaubey},
  \bibinfo{title}{Smooth estimation of a distribution and density function on a
  hypercube using {B}ernstein polynomials for dependent random vectors},
  \bibinfo{journal}{Statist. Probab. Lett.} \bibinfo{volume}{76}
  (\bibinfo{year}{2006}) \bibinfo{pages}{959--969}.
  \bibinfo{note}{\href{http://www.ams.org/mathscinet-getitem?mr=MR2270097}{MR2270097}}.
\bibitem[{Bat{\i}r(2017)}]{MR3684463}
\bibinfo{author}{N.~Bat{\i}r}, \bibinfo{title}{Bounds for the gamma function},
  \bibinfo{journal}{Results Math.} \bibinfo{volume}{72} (\bibinfo{year}{2017})
  \bibinfo{pages}{865--874}.
  \bibinfo{note}{\href{http://www.ams.org/mathscinet-getitem?mr=MR3684463}{MR3684463}}.
\bibitem[{Belaid et~al.(2016{\natexlab{a}})Belaid, Adjabi, Kokonendji and
  Zougab}]{MR3547951}
\bibinfo{author}{N.~Belaid}, \bibinfo{author}{S.~Adjabi},
  \bibinfo{author}{C.~C. Kokonendji}, \bibinfo{author}{N.~Zougab},
  \bibinfo{title}{Bayesian local bandwidth selector in multivariate associated
  kernel estimator for joint probability mass functions}, \bibinfo{journal}{J.
  Stat. Comput. Simul.} \bibinfo{volume}{86}
  (\bibinfo{year}{2016}{\natexlab{a}}) \bibinfo{pages}{3667--3681}.
  \bibinfo{note}{\href{http://www.ams.org/mathscinet-getitem?mr=MR3547951}{MR3547951}}.
\bibitem[{Belaid et~al.(2016{\natexlab{b}})Belaid, Adjabi, Zougab and
  Kokonendji}]{MR3566162}
\bibinfo{author}{N.~Belaid}, \bibinfo{author}{S.~Adjabi},
  \bibinfo{author}{N.~Zougab}, \bibinfo{author}{C.~C. Kokonendji},
  \bibinfo{title}{Bayesian bandwidth selection in discrete multivariate
  associated kernel estimators for probability mass functions},
  \bibinfo{journal}{J. Korean Statist. Soc.} \bibinfo{volume}{45}
  (\bibinfo{year}{2016}{\natexlab{b}}) \bibinfo{pages}{557--567}.
  \bibinfo{note}{\href{http://www.ams.org/mathscinet-getitem?mr=MR3566162}{MR3566162}}.
\bibitem[{Belalia(2016)}]{MR3474765}
\bibinfo{author}{M.~Belalia}, \bibinfo{title}{On the asymptotic properties of
  the {B}ernstein estimator of the multivariate distribution function},
  \bibinfo{journal}{Statist. Probab. Lett.} \bibinfo{volume}{110}
  (\bibinfo{year}{2016}) \bibinfo{pages}{249--256}.
  \bibinfo{note}{\href{http://www.ams.org/mathscinet-getitem?mr=MR3474765}{MR3474765}}.
\bibitem[{Bertin et~al.(2019)Bertin, El~Kolei and Klutchnikoff}]{MR4029144}
\bibinfo{author}{K.~Bertin}, \bibinfo{author}{S.~El~Kolei},
  \bibinfo{author}{N.~Klutchnikoff}, \bibinfo{title}{Adaptive density
  estimation on bounded domains}, \bibinfo{journal}{Ann. Inst. Henri
  Poincar\'{e} Probab. Stat.} \bibinfo{volume}{55} (\bibinfo{year}{2019})
  \bibinfo{pages}{1916--1947}.
  \bibinfo{note}{\href{http://www.ams.org/mathscinet-getitem?mr=MR4029144}{MR4029144}}.
\bibitem[{Bertin and Klutchnikoff(2011)}]{MR2775207}
\bibinfo{author}{K.~Bertin}, \bibinfo{author}{N.~Klutchnikoff},
  \bibinfo{title}{Minimax properties of beta kernel estimators},
  \bibinfo{journal}{J. Statist. Plann. Inference} \bibinfo{volume}{141}
  (\bibinfo{year}{2011}) \bibinfo{pages}{2287--2297}.
  \bibinfo{note}{\href{http://www.ams.org/mathscinet-getitem?mr=MR2775207}{MR2775207}}.
\bibitem[{Bertin and Klutchnikoff(2014)}]{MR3333996}
\bibinfo{author}{K.~Bertin}, \bibinfo{author}{N.~Klutchnikoff},
  \bibinfo{title}{Adaptive estimation of a density function using beta
  kernels}, \bibinfo{journal}{ESAIM Probab. Stat.} \bibinfo{volume}{18}
  (\bibinfo{year}{2014}) \bibinfo{pages}{400--417}.
  \bibinfo{note}{\href{http://www.ams.org/mathscinet-getitem?mr=MR3333996}{MR3333996}}.
\bibitem[{Bertin et~al.(2020)Bertin, Klutchnikoff, L\'{e}on and
  Prieur}]{MR4097810}
\bibinfo{author}{K.~Bertin}, \bibinfo{author}{N.~Klutchnikoff},
  \bibinfo{author}{J.~R. L\'{e}on}, \bibinfo{author}{C.~Prieur},
  \bibinfo{title}{Adaptive density estimation on bounded domains under mixing
  conditions}, \bibinfo{journal}{Electron. J. Stat.} \bibinfo{volume}{14}
  (\bibinfo{year}{2020}) \bibinfo{pages}{2198--2237}.
  \bibinfo{note}{\href{http://www.ams.org/mathscinet-getitem?mr=MR4097810}{MR4097810}}.
\bibitem[{Bouezmarni and van Bellegem(2011)}]{Bouezmarni_van_Bellegem_2011}
\bibinfo{author}{T.~Bouezmarni}, \bibinfo{author}{S.~van Bellegem},
  \bibinfo{title}{Nonparametric beta kernel estimator for long memory time
  series}, \bibinfo{journal}{CORE Discussion Paper}  (\bibinfo{year}{2011})
  \bibinfo{pages}{1--20}. \bibinfo{note}{[URL]~
  \href{https://ideas.repec.org/p/cor/louvco/2011004.html}{https://ideas.repec.org/p/cor/louvco/2011004.html}}.
\bibitem[{Bouezmarni et~al.(2011)Bouezmarni, El~Ghouch and
  Mesfioui}]{MR2801351}
\bibinfo{author}{T.~Bouezmarni}, \bibinfo{author}{A.~El~Ghouch},
  \bibinfo{author}{M.~Mesfioui}, \bibinfo{title}{Gamma kernel estimators for
  density and hazard rate of right-censored data}, \bibinfo{journal}{J. Probab.
  Stat.}  (\bibinfo{year}{2011}) \bibinfo{pages}{Art. ID 937574, 16 pp}.
  \bibinfo{note}{\href{http://www.ams.org/mathscinet-getitem?mr=MR2801351}{MR2801351}}.
\bibitem[{Bouezmarni and Rolin(2003)}]{MR1985506}
\bibinfo{author}{T.~Bouezmarni}, \bibinfo{author}{J.-M. Rolin},
  \bibinfo{title}{Consistency of the beta kernel density function estimator},
  \bibinfo{journal}{Canad. J. Statist.} \bibinfo{volume}{31}
  (\bibinfo{year}{2003}) \bibinfo{pages}{89--98}.
  \bibinfo{note}{\href{http://www.ams.org/mathscinet-getitem?mr=MR1985506}{MR1985506}}.
\bibitem[{Bouezmarni and Rolin(2007)}]{MR2351744}
\bibinfo{author}{T.~Bouezmarni}, \bibinfo{author}{J.-M. Rolin},
  \bibinfo{title}{{B}ernstein estimator for unbounded density function},
  \bibinfo{journal}{J. Nonparametr. Stat.} \bibinfo{volume}{19}
  (\bibinfo{year}{2007}) \bibinfo{pages}{145--161}.
  \bibinfo{note}{\href{http://www.ams.org/mathscinet-getitem?mr=MR2351744}{MR2351744}}.
\bibitem[{Bouezmarni and Rombouts(2008)}]{MR2454617}
\bibinfo{author}{T.~Bouezmarni}, \bibinfo{author}{J.~V.~K. Rombouts},
  \bibinfo{title}{Density and hazard rate estimation for censored and
  {$\alpha$}-mixing data using gamma kernels}, \bibinfo{journal}{J.
  Nonparametr. Stat.} \bibinfo{volume}{20} (\bibinfo{year}{2008})
  \bibinfo{pages}{627--643}.
  \bibinfo{note}{\href{http://www.ams.org/mathscinet-getitem?mr=MR2454617}{MR2454617}}.
\bibitem[{Bouezmarni and Rombouts(2009)}]{MR2665093}
\bibinfo{author}{T.~Bouezmarni}, \bibinfo{author}{J.~V.~K. Rombouts},
  \bibinfo{title}{Semiparametric multivariate density estimation for positive
  data using copulas}, \bibinfo{journal}{Comput. Statist. Data Anal.}
  \bibinfo{volume}{53} (\bibinfo{year}{2009}) \bibinfo{pages}{2040--2054}.
  \bibinfo{note}{\href{http://www.ams.org/mathscinet-getitem?mr=MR2665093}{MR2665093}}.
\bibitem[{Bouezmarni and Rombouts(2010{\natexlab{a}})}]{MR2568128}
\bibinfo{author}{T.~Bouezmarni}, \bibinfo{author}{J.~V.~K. Rombouts},
  \bibinfo{title}{Nonparametric density estimation for multivariate bounded
  data}, \bibinfo{journal}{J. Statist. Plann. Inference} \bibinfo{volume}{140}
  (\bibinfo{year}{2010}{\natexlab{a}}) \bibinfo{pages}{139--152}.
  \bibinfo{note}{\href{http://www.ams.org/mathscinet-getitem?mr=MR2568128}{MR2568128}}.
\bibitem[{Bouezmarni and Rombouts(2010{\natexlab{b}})}]{MR2756423}
\bibinfo{author}{T.~Bouezmarni}, \bibinfo{author}{J.~V.~K. Rombouts},
  \bibinfo{title}{Nonparametric density estimation for positive time series},
  \bibinfo{journal}{Comput. Statist. Data Anal.} \bibinfo{volume}{54}
  (\bibinfo{year}{2010}{\natexlab{b}}) \bibinfo{pages}{245--261}.
  \bibinfo{note}{\href{http://www.ams.org/mathscinet-getitem?mr=MR2756423}{MR2756423}}.
\bibitem[{Bouezmarni and Scaillet(2005)}]{MR2179543}
\bibinfo{author}{T.~Bouezmarni}, \bibinfo{author}{O.~Scaillet},
  \bibinfo{title}{Consistency of asymmetric kernel density estimators and
  smoothed histograms with application to income data},
  \bibinfo{journal}{Econom. Theor.} \bibinfo{volume}{21} (\bibinfo{year}{2005})
  \bibinfo{pages}{390--412}.
  \bibinfo{note}{\href{http://www.ams.org/mathscinet-getitem?mr=MR2179543}{MR2179543}}.
\bibitem[{Bouezmarni et~al.(2020)Bouezmarni, Van~Bellegem and
  Rabhi}]{MR4148613}
\bibinfo{author}{T.~Bouezmarni}, \bibinfo{author}{S.~Van~Bellegem},
  \bibinfo{author}{Y.~Rabhi}, \bibinfo{title}{Nonparametric beta kernel
  estimator for long and short memory time series}, \bibinfo{journal}{Canad. J.
  Statist.} \bibinfo{volume}{48} (\bibinfo{year}{2020})
  \bibinfo{pages}{582--595}.
  \bibinfo{note}{\href{http://www.ams.org/mathscinet-getitem?mr=MR4148613}{MR4148613}}.
\bibitem[{Brown and Chen(1999)}]{MR1685301}
\bibinfo{author}{B.~M. Brown}, \bibinfo{author}{S.~X. Chen},
  \bibinfo{title}{Beta-{B}ernstein smoothing for regression curves with compact
  support}, \bibinfo{journal}{Scand. J. Statist.} \bibinfo{volume}{26}
  (\bibinfo{year}{1999}) \bibinfo{pages}{47--59}.
  \bibinfo{note}{\href{http://www.ams.org/mathscinet-getitem?mr=MR1685301}{MR1685301}}.
\bibitem[{Chac\'on et~al.(2011)Chac\'on, Mateu-Figueras and
  Mart\'in-Fern\'andez}]{doi:10.1016/j.cageo.2009.12.011}
\bibinfo{author}{J.~E. Chac\'on}, \bibinfo{author}{G.~Mateu-Figueras},
  \bibinfo{author}{J.~A. Mart\'in-Fern\'andez}, \bibinfo{title}{Gaussian
  kernels for density estimation with compositional data},
  \bibinfo{journal}{Computers \& Geosciences} \bibinfo{volume}{37}
  (\bibinfo{year}{2011}) \bibinfo{pages}{702--711}.
  \bibinfo{note}{\href{https://doi.org/10.1016/j.cageo.2009.12.011}{doi:10.1016/j.cageo.2009.12.011}}.
\bibitem[{Charpentier(2006)}]{Charpentier2006phd}
\bibinfo{author}{A.~Charpentier}, \bibinfo{title}{Dependence {S}tructures and
  {L}imiting {R}esults, with {A}pplications in {F}inance and {I}nsurance},
  \bibinfo{type}{Ph{D} thesis}, Katholieke Universiteit Leuven,
  \bibinfo{year}{2006}. \bibinfo{note}{[URL]~
  \url{https://tel.archives-ouvertes.fr/file/index/docid/82892/filename/thesis.pdf}}.
\bibitem[{Charpentier et~al.(2007)Charpentier, Fermanian and
  Scaillet}]{Charpentier_Fermanian_Scaillet_2007}
\bibinfo{author}{A.~Charpentier}, \bibinfo{author}{J.-D. Fermanian},
  \bibinfo{author}{O.~Scaillet}, \bibinfo{title}{The estimation of copulas:
  theory and practice}, in: \bibinfo{editor}{J.~Rank} (Ed.),
  \bibinfo{booktitle}{Copulas: from theory to application in finance},
  \bibinfo{publisher}{London: Risk Books}, \bibinfo{year}{2007}, pp.
  \bibinfo{pages}{35--64}. \bibinfo{note}{\newline [URL]~
  \url{https://archive-ouverte.unige.ch/unige:41917}}.
\bibitem[{Charpentier and Flachaire(2015)}]{doi:10.2139/ssrn.2514882}
\bibinfo{author}{A.~Charpentier}, \bibinfo{author}{E.~Flachaire},
  \bibinfo{title}{Log-transform kernel density estimation of income
  distribution}, \bibinfo{journal}{L’Actualité économique, Revue
  d’analyse économique} \bibinfo{volume}{91} (\bibinfo{year}{2015})
  \bibinfo{pages}{141--159}.
  \bibinfo{note}{\href{https://doi.org/10.2139/ssrn.2514882}{doi:10.2139/ssrn.2514882}}.
\bibitem[{Charpentier and Oulidi(2010)}]{MR2578075}
\bibinfo{author}{A.~Charpentier}, \bibinfo{author}{A.~Oulidi},
  \bibinfo{title}{Beta kernel quantile estimators of heavy-tailed loss
  distributions}, \bibinfo{journal}{Stat. Comput.} \bibinfo{volume}{20}
  (\bibinfo{year}{2010}) \bibinfo{pages}{35--55}.
  \bibinfo{note}{\href{http://www.ams.org/mathscinet-getitem?mr=MR2578075}{MR2578075}}.
\bibitem[{Chaubey et~al.(2012{\natexlab{a}})Chaubey, Dewan and Li}]{MR2869005}
\bibinfo{author}{Y.~P. Chaubey}, \bibinfo{author}{I.~Dewan},
  \bibinfo{author}{J.~Li}, \bibinfo{title}{An asymmetric kernel estimator of
  density function for stationary associated sequences},
  \bibinfo{journal}{Comm. Statist. Simulation Comput.} \bibinfo{volume}{41}
  (\bibinfo{year}{2012}{\natexlab{a}}) \bibinfo{pages}{554--572}.
  \bibinfo{note}{\href{http://www.ams.org/mathscinet-getitem?mr=MR2869005}{MR2869005}}.
\bibitem[{Chaubey and Li(2013)}]{Chaubey_Li_2013}
\bibinfo{author}{Y.~P. Chaubey}, \bibinfo{author}{J.~Li},
  \bibinfo{title}{Asymmetric kernel density estimator for length biased data},
  in: \bibinfo{booktitle}{Contemporary Topics in Mathematics and Statistics
  with Applications}, volume~\bibinfo{volume}{1}, \bibinfo{publisher}{Asian
  Books Private Ltd}, \bibinfo{year}{2013}, p. \bibinfo{pages}{28 pp.}
\bibitem[{Chaubey et~al.(2012{\natexlab{b}})Chaubey, Li, Sen and
  Sen}]{MR2975812}
\bibinfo{author}{Y.~P. Chaubey}, \bibinfo{author}{J.~Li},
  \bibinfo{author}{A.~Sen}, \bibinfo{author}{P.~K. Sen}, \bibinfo{title}{A new
  smooth density estimator for non-negative random variables},
  \bibinfo{journal}{J. Indian Statist. Assoc.} \bibinfo{volume}{50}
  (\bibinfo{year}{2012}{\natexlab{b}}) \bibinfo{pages}{83--104}.
  \bibinfo{note}{\href{http://www.ams.org/mathscinet-getitem?mr=MR2975812}{MR2975812}}.
\bibitem[{Chekkal et~al.(2021)Chekkal, Lagha and
  Zougab}]{Chekkal_Lagha_Zougab_2021}
\bibinfo{author}{S.~Chekkal}, \bibinfo{author}{K.~Lagha},
  \bibinfo{author}{N.~Zougab}, \bibinfo{title}{Generalized
  {B}irnbaum–{S}aunders kernel for hazard rate function estimation},
  \bibinfo{journal}{Comm. Statist. Simulation Comput.}  (\bibinfo{year}{2021})
  \bibinfo{pages}{1--16}.
  \bibinfo{note}{\href{https://www.doi.org/10.1080/03610918.2021.1887228}{doi:10.1080/03610918.2021.1887228}}.
\bibitem[{Chen(1999)}]{MR1718494}
\bibinfo{author}{S.~X. Chen}, \bibinfo{title}{Beta kernel estimators for
  density functions}, \bibinfo{journal}{Comput. Statist. Data Anal.}
  \bibinfo{volume}{31} (\bibinfo{year}{1999}) \bibinfo{pages}{131--145}.
  \bibinfo{note}{\href{http://www.ams.org/mathscinet-getitem?mr=MR1718494}{MR1718494}}.
\bibitem[{Chen(2000{\natexlab{a}})}]{MR1742101}
\bibinfo{author}{S.~X. Chen}, \bibinfo{title}{Beta kernel smoothers for
  regression curves}, \bibinfo{journal}{Statist. Sinica} \bibinfo{volume}{10}
  (\bibinfo{year}{2000}{\natexlab{a}}) \bibinfo{pages}{73--91}.
  \bibinfo{note}{\href{http://www.ams.org/mathscinet-getitem?mr=MR1742101}{MR1742101}}.
\bibitem[{Chen(2000{\natexlab{b}})}]{MR1794247}
\bibinfo{author}{S.~X. Chen}, \bibinfo{title}{Probability density function
  estimation using gamma kernels}, \bibinfo{journal}{Ann. Inst. Statist. Math}
  \bibinfo{volume}{52} (\bibinfo{year}{2000}{\natexlab{b}})
  \bibinfo{pages}{471--480}.
  \bibinfo{note}{\href{http://www.ams.org/mathscinet-getitem?mr=MR1794247}{MR1794247}}.
\bibitem[{Chen(2002)}]{MR1910175}
\bibinfo{author}{S.~X. Chen}, \bibinfo{title}{Local linear smoothers using
  asymmetric kernels}, \bibinfo{journal}{Ann. Inst. Statist. Math.}
  \bibinfo{volume}{54} (\bibinfo{year}{2002}) \bibinfo{pages}{312--323}.
  \bibinfo{note}{\href{http://www.ams.org/mathscinet-getitem?mr=MR1910175}{MR1910175}}.
\bibitem[{Comte and Genon-Catalot(2012)}]{MR2903382}
\bibinfo{author}{F.~Comte}, \bibinfo{author}{V.~Genon-Catalot},
  \bibinfo{title}{Convolution power kernels for density estimation},
  \bibinfo{journal}{J. Statist. Plann. Inference} \bibinfo{volume}{142}
  (\bibinfo{year}{2012}) \bibinfo{pages}{1698--1715}.
  \bibinfo{note}{\href{http://www.ams.org/mathscinet-getitem?mr=MR2903382}{MR2903382}}.
\bibitem[{Devroye and Gy\"{o}rfi(1985)}]{MR780746}
\bibinfo{author}{L.~Devroye}, \bibinfo{author}{L.~Gy\"{o}rfi},
  \bibinfo{title}{Nonparametric {D}ensity {E}stimation: {T}he $L_1$ {V}iew},
  Wiley Series in Probability and Mathematical Statistics,
  \bibinfo{publisher}{John Wiley \& Sons, Inc., New York},
  \bibinfo{year}{1985}.
  \bibinfo{note}{\href{http://www.ams.org/mathscinet-getitem?mr=MR780746}{MR780746}}.
\bibitem[{Devroye and Penrod(1984)}]{MR760686}
\bibinfo{author}{L.~Devroye}, \bibinfo{author}{C.~S. Penrod},
  \bibinfo{title}{Distribution-free lower bounds in density estimation},
  \bibinfo{journal}{Ann. Statist.} \bibinfo{volume}{12} (\bibinfo{year}{1984})
  \bibinfo{pages}{1250--1262}.
  \bibinfo{note}{\href{http://www.ams.org/mathscinet-getitem?mr=MR760686}{MR760686}}.
\bibitem[{Devroye and Penrod(1986)}]{MR859633}
\bibinfo{author}{L.~Devroye}, \bibinfo{author}{C.~S. Penrod},
  \bibinfo{title}{The strong uniform convergence of multivariate variable
  kernel estimates}, \bibinfo{journal}{Canad. J. Statist.} \bibinfo{volume}{14}
  (\bibinfo{year}{1986}) \bibinfo{pages}{211--219}.
  \bibinfo{note}{\href{http://www.ams.org/mathscinet-getitem?mr=MR859633}{MR859633}}.
\bibitem[{Dobrovidov and
  Markovich(2013{\natexlab{a}})}]{doi:10.3182/20130703-3-FR-4038.00086}
\bibinfo{author}{A.~V. Dobrovidov}, \bibinfo{author}{L.~A. Markovich},
  \bibinfo{title}{Data-driven bandwidth choice for gamma kernel estimates of
  density derivatives on the positive semi-axis}, \bibinfo{journal}{IFAC
  Proceedings Volumes} \bibinfo{volume}{46}
  (\bibinfo{year}{2013}{\natexlab{a}}) \bibinfo{pages}{500--505}.
  \bibinfo{note}{\href{https://doi.org/10.3182/20130703-3-FR-4038.00086}{doi:10.3182/20130703-3-FR-4038.00086}}.
\bibitem[{Dobrovidov and
  Markovich(2013{\natexlab{b}})}]{doi:10.3182/20130619-3-RU-3018.00214}
\bibinfo{author}{A.~V. Dobrovidov}, \bibinfo{author}{L.~A. Markovich},
  \bibinfo{title}{Nonparametric gamma kernel estimators of density derivatives
  on positive semi-axis}, \bibinfo{journal}{IFAC Proc. Vol.}
  \bibinfo{volume}{46} (\bibinfo{year}{2013}{\natexlab{b}})
  \bibinfo{pages}{910--915}.
  \bibinfo{note}{\href{https://doi.org/10.3182/20130619-3-RU-3018.00214}{doi:10.3182/20130619-3-RU-3018.00214}}.
\bibitem[{Er\c{c}elik and Nadar(2020{\natexlab{a}})}]{Ercelik_Nadar_2020}
\bibinfo{author}{E.~Er\c{c}elik}, \bibinfo{author}{M.~Nadar}, \bibinfo{title}{A
  new kernel estimator based on scaled inverse chi-squared density function},
  \bibinfo{journal}{Am. J. Math. Manag. Sci.}
  (\bibinfo{year}{2020}{\natexlab{a}}) \bibinfo{pages}{1--14}.
  \bibinfo{note}{\href{https://www.doi.org/10.1080/01966324.2020.1854138}{doi:10.1080/01966324.2020.1854138}}.
\bibitem[{Er\c{c}elik and Nadar(2020{\natexlab{b}})}]{MR4037101}
\bibinfo{author}{E.~Er\c{c}elik}, \bibinfo{author}{M.~Nadar},
  \bibinfo{title}{Nonparametric density estimation based on beta prime kernel},
  \bibinfo{journal}{Comm. Statist. Theory Methods} \bibinfo{volume}{49}
  (\bibinfo{year}{2020}{\natexlab{b}}) \bibinfo{pages}{325--342}.
  \bibinfo{note}{\href{http://www.ams.org/mathscinet-getitem?mr=MR4037101}{MR4037101}}.
\bibitem[{Fan and Gijbels(1992)}]{MR1193323}
\bibinfo{author}{J.~Fan}, \bibinfo{author}{I.~Gijbels},
  \bibinfo{title}{Variable bandwidth and local linear regression smoothers},
  \bibinfo{journal}{Ann. Statist.} \bibinfo{volume}{20} (\bibinfo{year}{1992})
  \bibinfo{pages}{2008--2036}.
  \bibinfo{note}{\href{http://www.ams.org/mathscinet-getitem?mr=MR1193323}{MR1193323}}.
\bibitem[{Fauzi and Maesono(2020)}]{MR4147689}
\bibinfo{author}{R.~R. Fauzi}, \bibinfo{author}{Y.~Maesono},
  \bibinfo{title}{New type of gamma kernel density estimator},
  \bibinfo{journal}{J. Korean Statist. Soc.} \bibinfo{volume}{49}
  (\bibinfo{year}{2020}) \bibinfo{pages}{882--900}.
  \bibinfo{note}{\href{http://www.ams.org/mathscinet-getitem?mr=MR4147689}{MR4147689}}.
\bibitem[{F\'{e}(2014)}]{MR3256394}
\bibinfo{author}{E.~F\'{e}}, \bibinfo{title}{Estimation and inference in
  regression discontinuity designs with asymmetric kernels},
  \bibinfo{journal}{J. Appl. Stat.} \bibinfo{volume}{41} (\bibinfo{year}{2014})
  \bibinfo{pages}{2406--2417}.
  \bibinfo{note}{\href{http://www.ams.org/mathscinet-getitem?mr=MR3256394}{MR3256394}}.
\bibitem[{Fernandes and Grammig(2005)}]{MR2137490}
\bibinfo{author}{M.~Fernandes}, \bibinfo{author}{J.~Grammig},
  \bibinfo{title}{Nonparametric specification tests for conditional duration
  models}, \bibinfo{journal}{J. Econometrics} \bibinfo{volume}{127}
  (\bibinfo{year}{2005}) \bibinfo{pages}{35--68}.
  \bibinfo{note}{\href{http://www.ams.org/mathscinet-getitem?mr=MR2137490}{MR2137490}}.
\bibitem[{Fernandes et~al.(2015)Fernandes, Mendes and Scaillet}]{MR3357933}
\bibinfo{author}{M.~Fernandes}, \bibinfo{author}{E.~F. Mendes},
  \bibinfo{author}{O.~Scaillet}, \bibinfo{title}{Testing for symmetry and
  conditional symmetry using asymmetric kernels}, \bibinfo{journal}{Ann. Inst.
  Statist. Math.} \bibinfo{volume}{67} (\bibinfo{year}{2015})
  \bibinfo{pages}{649--671}.
  \bibinfo{note}{\href{http://www.ams.org/mathscinet-getitem?mr=MR3357933}{MR3357933}}.
\bibitem[{Fernandes and Monteiro(2005)}]{MR2206532}
\bibinfo{author}{M.~Fernandes}, \bibinfo{author}{P.~K. Monteiro},
  \bibinfo{title}{Central limit theorem for asymmetric kernel functionals},
  \bibinfo{journal}{Ann. Inst. Statist. Math.} \bibinfo{volume}{57}
  (\bibinfo{year}{2005}) \bibinfo{pages}{425--442}.
  \bibinfo{note}{\href{http://www.ams.org/mathscinet-getitem?mr=MR2206532}{MR2206532}}.
\bibitem[{Funke and Hirukawa(2019)}]{MR3907679}
\bibinfo{author}{B.~Funke}, \bibinfo{author}{M.~Hirukawa},
  \bibinfo{title}{Nonparametric estimation and testing on discontinuity of
  positive supported densities: a kernel truncation approach},
  \bibinfo{journal}{Econom. Stat.} \bibinfo{volume}{9} (\bibinfo{year}{2019})
  \bibinfo{pages}{156--170}.
  \bibinfo{note}{\href{http://www.ams.org/mathscinet-getitem?mr=MR3907679}{MR3907679}}.
\bibitem[{Funke and Hirukawa(2021)}]{MR4302589}
\bibinfo{author}{B.~Funke}, \bibinfo{author}{M.~Hirukawa}, \bibinfo{title}{Bias
  correction for local linear regression estimation using asymmetric kernels
  via the skewing method}, \bibinfo{journal}{Econom. Stat.}
  \bibinfo{volume}{20} (\bibinfo{year}{2021}) \bibinfo{pages}{109--130}.
  \bibinfo{note}{\href{http://www.ams.org/mathscinet-getitem?mr=MR4302589}{MR4302589}}.
\bibitem[{Funke and Kawka(2015)}]{MR3384258}
\bibinfo{author}{B.~Funke}, \bibinfo{author}{R.~Kawka},
  \bibinfo{title}{Nonparametric density estimation for multivariate bounded
  data using two non-negative multiplicative bias correction methods},
  \bibinfo{journal}{Comput. Statist. Data Anal.} \bibinfo{volume}{92}
  (\bibinfo{year}{2015}) \bibinfo{pages}{148--162}.
  \bibinfo{note}{\href{http://www.ams.org/mathscinet-getitem?mr=MR3384258}{MR3384258}}.
\bibitem[{Gasser and M{\"u}ller(1979)}]{doi:10.1007/BFb0098489}
\bibinfo{author}{T.~Gasser}, \bibinfo{author}{H.-G. M{\"u}ller},
  \bibinfo{title}{Kernel estimation of regression functions}, in:
  \bibinfo{booktitle}{Smoothing Techniques for Curve Estimation},
  \bibinfo{publisher}{Springer Berlin Heidelberg}, \bibinfo{year}{1979}, pp.
  \bibinfo{pages}{23--68}.
  \bibinfo{note}{\href{https://doi.org/10.1007/BFb0098489}{doi:10.1007/BFb0098489}}.
\bibitem[{Gasser et~al.(1985)Gasser, M\"{u}ller and Mammitzsch}]{MR816088}
\bibinfo{author}{T.~Gasser}, \bibinfo{author}{H.-G. M\"{u}ller},
  \bibinfo{author}{V.~Mammitzsch}, \bibinfo{title}{Kernels for nonparametric
  curve estimation}, \bibinfo{journal}{J. Roy. Statist. Soc. Ser. B}
  \bibinfo{volume}{47} (\bibinfo{year}{1985}) \bibinfo{pages}{238--252}.
  \bibinfo{note}{\href{http://www.ams.org/mathscinet-getitem?mr=MR816088}{MR816088}}.
\bibitem[{Gawronski(1985)}]{MR0791719}
\bibinfo{author}{W.~Gawronski}, \bibinfo{title}{Strong laws for density
  estimators of {B}ernstein type}, \bibinfo{journal}{Period. Math. Hungar}
  \bibinfo{volume}{16} (\bibinfo{year}{1985}) \bibinfo{pages}{23--43}.
  \bibinfo{note}{\href{http://www.ams.org/mathscinet-getitem?mr=MR0791719}{MR0791719}}.
\bibitem[{Gawronski and Stadtm\"uller(1980)}]{MR0574548}
\bibinfo{author}{W.~Gawronski}, \bibinfo{author}{U.~Stadtm\"uller},
  \bibinfo{title}{On density estimation by means of {P}oisson's distribution},
  \bibinfo{journal}{Scand. J. Statist.} \bibinfo{volume}{7}
  (\bibinfo{year}{1980}) \bibinfo{pages}{90--94}.
  \bibinfo{note}{\href{http://www.ams.org/mathscinet-getitem?mr=MR0574548}{MR0574548}}.
\bibitem[{Gawronski and Stadtm\"uller(1981)}]{MR0638651}
\bibinfo{author}{W.~Gawronski}, \bibinfo{author}{U.~Stadtm\"uller},
  \bibinfo{title}{Smoothing histograms by means of lattice and continuous
  distributions}, \bibinfo{journal}{Metrika} \bibinfo{volume}{28}
  (\bibinfo{year}{1981}) \bibinfo{pages}{155--164}.
  \bibinfo{note}{\href{http://www.ams.org/mathscinet-getitem?mr=MR0638651}{MR0638651}}.
\bibitem[{Gawronski and Stadtm\"uller(1984)}]{MR0755576}
\bibinfo{author}{W.~Gawronski}, \bibinfo{author}{U.~Stadtm\"uller},
  \bibinfo{title}{Linear combinations of iterated generalized {B}ernstein
  functions with an application to density estimation}, \bibinfo{journal}{Acta
  Sci. Math.} \bibinfo{volume}{47} (\bibinfo{year}{1984})
  \bibinfo{pages}{205--221}.
  \bibinfo{note}{\href{http://www.ams.org/mathscinet-getitem?mr=MR0755576}{MR0755576}}.
\bibitem[{Geenens(2020)}]{doi:10.1007/s10463-020-00772-1}
\bibinfo{author}{G.~Geenens}, \bibinfo{title}{Mellin-{M}eijer-kernel density
  estimation on $\mathbb{R}^+$}, \bibinfo{journal}{Ann. Inst. Stat. Math.}
  (\bibinfo{year}{2020}) \bibinfo{pages}{25 pp.}
  \bibinfo{note}{\href{https://www.doi.org/10.1007/s10463-020-00772-1}{doi:10.1007/s10463-020-00772-1}}.
\bibitem[{Gospodinov and Hirukawa(2012)}]{doi:10.1016/j.jempfin.2012.04.001}
\bibinfo{author}{N.~Gospodinov}, \bibinfo{author}{M.~Hirukawa},
  \bibinfo{title}{Nonparametric estimation of scalar diffusion models of
  interest rates using asymmetric kernels}, \bibinfo{journal}{J. Empir.
  Finance} \bibinfo{volume}{19} (\bibinfo{year}{2012})
  \bibinfo{pages}{595--609}.
  \bibinfo{note}{\href{https://doi.org/10.1016/j.jempfin.2012.04.001}{doi:10.1016/j.jempfin.2012.04.001}}.
\bibitem[{Gouri\'eroux and Monfort(2006)}]{Gourieroux_Monfort_2006}
\bibinfo{author}{C.~Gouri\'eroux}, \bibinfo{author}{A.~Monfort},
  \bibinfo{title}{(non) consistency of the beta kernel estimator for recovery
  rate distribution}, \bibinfo{journal}{CREST Discussion Paper}
  (\bibinfo{year}{2006}) \bibinfo{pages}{1--27}. \bibinfo{note}{\newline[URL]~
  \href{https://ideas.repec.org/p/crs/wpaper/2006-31.html}{https://ideas.repec.org/p/crs/wpaper/2006-31.html}}.
\bibitem[{Gustafsson et~al.(2009)Gustafsson, Hagmann, Nielsen and
  Scaillet}]{MR2516437}
\bibinfo{author}{J.~Gustafsson}, \bibinfo{author}{M.~Hagmann},
  \bibinfo{author}{J.~P. Nielsen}, \bibinfo{author}{O.~Scaillet},
  \bibinfo{title}{Local transformation kernel density estimation of loss
  distributions}, \bibinfo{journal}{J. Bus. Econom. Statist.}
  \bibinfo{volume}{27} (\bibinfo{year}{2009}) \bibinfo{pages}{161--175}.
  \bibinfo{note}{\href{http://www.ams.org/mathscinet-getitem?mr=MR2516437}{MR2516437}}.
\bibitem[{Hagmann and Scaillet(2007)}]{MR2411743}
\bibinfo{author}{M.~Hagmann}, \bibinfo{author}{O.~Scaillet},
  \bibinfo{title}{Local multiplicative bias correction for asymmetric kernel
  density estimators}, \bibinfo{journal}{J. Econometrics} \bibinfo{volume}{141}
  (\bibinfo{year}{2007}) \bibinfo{pages}{213--249}.
  \bibinfo{note}{\href{http://www.ams.org/mathscinet-getitem?mr=MR2411743}{MR2411743}}.
\bibitem[{Hall(1984)}]{MR734096}
\bibinfo{author}{P.~Hall}, \bibinfo{title}{Central limit theorem for integrated
  square error of multivariate nonparametric density estimators},
  \bibinfo{journal}{J. Multivariate Anal.} \bibinfo{volume}{14}
  (\bibinfo{year}{1984}) \bibinfo{pages}{1--16}.
  \bibinfo{note}{\href{http://www.ams.org/mathscinet-getitem?mr=MR734096}{MR734096}}.
\bibitem[{Hall and Wand(1988)}]{MR955204}
\bibinfo{author}{P.~Hall}, \bibinfo{author}{M.~P. Wand},
  \bibinfo{title}{Minimizing {$L_1$} distance in nonparametric density
  estimation}, \bibinfo{journal}{J. Multivariate Anal.} \bibinfo{volume}{26}
  (\bibinfo{year}{1988}) \bibinfo{pages}{59--88}.
  \bibinfo{note}{\href{http://www.ams.org/mathscinet-getitem?mr=MR955204}{MR955204}}.
\bibitem[{Hanebeck(2020)}]{Hanebeck2020master}
\bibinfo{author}{A.~Hanebeck}, \bibinfo{title}{Nonparametric {D}istribution
  {F}unction {E}stimation}, \bibinfo{type}{Master's thesis}, Karlsruher
  Institut f\"ur Technologie, \bibinfo{year}{2020}.
  \bibinfo{note}{\newline[URL]~
  \url{https://core.ac.uk/download/pdf/326703853.pdf}}.
\bibitem[{Hanebeck and Klar(2021)}]{doi:10.1007/s10463-020-00783-y}
\bibinfo{author}{A.~Hanebeck}, \bibinfo{author}{B.~Klar},
  \bibinfo{title}{Smooth distribution function estimation for lifetime
  distributions using {S}zasz-{M}irakyan operators}, \bibinfo{journal}{Ann.
  Inst. Stat. Math.}  (\bibinfo{year}{2021}) \bibinfo{pages}{19 pp.}
  \bibinfo{note}{\href{https://www.doi.org/10.1007/s10463-020-00783-y}{doi:10.1007/s10463-020-00783-y}}.
\bibitem[{Hanif(2013)}]{MR3175804}
\bibinfo{author}{M.~Hanif}, \bibinfo{title}{Local linear estimation of
  jump-diffusion models by using asymmetric kernels}, \bibinfo{journal}{Stoch.
  Anal. Appl.} \bibinfo{volume}{31} (\bibinfo{year}{2013})
  \bibinfo{pages}{956--974}.
  \bibinfo{note}{\href{http://www.ams.org/mathscinet-getitem?mr=MR3175804}{MR3175804}}.
\bibitem[{Harfouche et~al.(2018)Harfouche, Adjabi, Zougab and
  Funke}]{MR3807369}
\bibinfo{author}{L.~Harfouche}, \bibinfo{author}{S.~Adjabi},
  \bibinfo{author}{N.~Zougab}, \bibinfo{author}{B.~Funke},
  \bibinfo{title}{Multiplicative bias correction for discrete kernels},
  \bibinfo{journal}{Stat. Methods Appl.} \bibinfo{volume}{27}
  (\bibinfo{year}{2018}) \bibinfo{pages}{253--276}.
  \bibinfo{note}{\href{http://www.ams.org/mathscinet-getitem?mr=MR3807369}{MR3807369}}.
\bibitem[{Harfouche et~al.(2020)Harfouche, Zougab and Adjabi}]{MR4086907}
\bibinfo{author}{L.~Harfouche}, \bibinfo{author}{N.~Zougab},
  \bibinfo{author}{S.~Adjabi}, \bibinfo{title}{Multivariate generalised gamma
  kernel density estimators and application to non-negative data},
  \bibinfo{journal}{Int. J. Comput. Sci. Math.} \bibinfo{volume}{11}
  (\bibinfo{year}{2020}) \bibinfo{pages}{137--157}.
  \bibinfo{note}{\href{http://www.ams.org/mathscinet-getitem?mr=MR4086907}{MR4086907}}.
\bibitem[{Hirukawa(2010)}]{MR2756441}
\bibinfo{author}{M.~Hirukawa}, \bibinfo{title}{Nonparametric multiplicative
  bias correction for kernel-type density estimation on the unit interval},
  \bibinfo{journal}{Comput. Statist. Data Anal.} \bibinfo{volume}{54}
  (\bibinfo{year}{2010}) \bibinfo{pages}{473--495}.
  \bibinfo{note}{\href{http://www.ams.org/mathscinet-getitem?mr=MR2756441}{MR2756441}}.
\bibitem[{Hirukawa(2018)}]{MR3821525}
\bibinfo{author}{M.~Hirukawa}, \bibinfo{title}{Asymmetric {K}ernel
  {S}moothing}, SpringerBriefs in Statistics, \bibinfo{publisher}{Springer,
  Singapore}, \bibinfo{year}{2018}.
  \bibinfo{note}{\href{http://www.ams.org/mathscinet-getitem?mr=MR3821525}{MR3821525}}.
\bibitem[{Hirukawa et~al.(2020)Hirukawa, Murtazashvili and
  Prokhorov}]{Hirukawa_Murtazashvili_Prokhorov_2020}
\bibinfo{author}{M.~Hirukawa}, \bibinfo{author}{I.~Murtazashvili},
  \bibinfo{author}{A.~Prokhorov}, \bibinfo{title}{Uniform convergence rates for
  nonparametric estimators smoothed by the beta kernel},
  \bibinfo{journal}{Preprint}  (\bibinfo{year}{2020}) \bibinfo{pages}{36 pp.}
  \bibinfo{note}{\href{https://www.econ.ryukoku.ac.jp/~hirukawa/upload/uniform_31dec20_final.pdf}{https://www.econ.ryukoku.ac.jp/~hirukawa/upload/uniform\_31dec20\_final.pdf}}.
\bibitem[{Hirukawa and Sakudo(2014)}]{MR3178361}
\bibinfo{author}{M.~Hirukawa}, \bibinfo{author}{M.~Sakudo},
  \bibinfo{title}{Nonnegative bias reduction methods for density estimation
  using asymmetric kernels}, \bibinfo{journal}{Comput. Statist. Data Anal.}
  \bibinfo{volume}{75} (\bibinfo{year}{2014}) \bibinfo{pages}{112--123}.
  \bibinfo{note}{\href{http://www.ams.org/mathscinet-getitem?mr=MR3178361}{MR3178361}}.
\bibitem[{Hirukawa and Sakudo(2015)}]{MR3304359}
\bibinfo{author}{M.~Hirukawa}, \bibinfo{author}{M.~Sakudo},
  \bibinfo{title}{Family of the generalised gamma kernels: a generator of
  asymmetric kernels for nonnegative data}, \bibinfo{journal}{J. Nonparametr.
  Stat.} \bibinfo{volume}{27} (\bibinfo{year}{2015}) \bibinfo{pages}{41--63}.
  \bibinfo{note}{\href{http://www.ams.org/mathscinet-getitem?mr=MR3304359}{MR3304359}}.
\bibitem[{Hirukawa and Sakudo(2016)}]{doi:10.3390/econometrics4020028}
\bibinfo{author}{M.~Hirukawa}, \bibinfo{author}{M.~Sakudo},
  \bibinfo{title}{Testing symmetry of unknown densities via smoothing with the
  generalized gamma kernels}, \bibinfo{journal}{Econometrics}
  \bibinfo{volume}{4} (\bibinfo{year}{2016}) \bibinfo{pages}{27 pp.}
  \bibinfo{note}{\href{https://doi.org/10.3390/econometrics4020028}{doi:10.3390/econometrics4020028}}.
\bibitem[{Hirukawa and Sakudo(2019)}]{MR3873902}
\bibinfo{author}{M.~Hirukawa}, \bibinfo{author}{M.~Sakudo},
  \bibinfo{title}{Another bias correction for asymmetric kernel density
  estimation with a parametric start}, \bibinfo{journal}{Statist. Probab.
  Lett.} \bibinfo{volume}{145} (\bibinfo{year}{2019})
  \bibinfo{pages}{158--165}.
  \bibinfo{note}{\href{http://www.ams.org/mathscinet-getitem?mr=MR3873902}{MR3873902}}.
\bibitem[{Hjort and Glad(1995)}]{MR1345205}
\bibinfo{author}{N.~L. Hjort}, \bibinfo{author}{I.~K. Glad},
  \bibinfo{title}{Nonparametric density estimation with a parametric start},
  \bibinfo{journal}{Ann. Statist.} \bibinfo{volume}{23} (\bibinfo{year}{1995})
  \bibinfo{pages}{882--904}.
  \bibinfo{note}{\href{http://www.ams.org/mathscinet-getitem?mr=MR1345205}{MR1345205}}.
\bibitem[{Hoang et~al.(2020)Hoang, Pereira, Kupka, Tolosana-Delgado, Frenzel,
  Rudolph and Gutzmer}]{doi:10.14278/rodare.543}
\bibinfo{author}{D.~H. Hoang}, \bibinfo{author}{L.~Pereira},
  \bibinfo{author}{N.~Kupka}, \bibinfo{author}{R.~Tolosana-Delgado},
  \bibinfo{author}{M.~Frenzel}, \bibinfo{author}{M.~Rudolph},
  \bibinfo{author}{J.~Gutzmer}, \bibinfo{title}{Automated mineralogy particle
  dataset: apatite flotation}, \bibinfo{year}{2020}.
  \bibinfo{note}{\href{https://www.doi.org/10.14278/rodare.543}{doi:10.14278/rodare.543}}.
\bibitem[{Hoffmann and Jones(2015)}]{arXiv:1512.03188}
\bibinfo{author}{T.~Hoffmann}, \bibinfo{author}{N.~Jones},
  \bibinfo{title}{Unified treatment of the asymptotics of asymmetric kernel
  density estimators}, \bibinfo{journal}{Preprint}  (\bibinfo{year}{2015})
  \bibinfo{pages}{1--16}.
  \bibinfo{note}{\href{https://arxiv.org/abs/1512.03188}{arXiv:1512.03188}}.
\bibitem[{Hurvich(1985)}]{MR819597}
\bibinfo{author}{C.~M. Hurvich}, \bibinfo{title}{Data-driven choice of a
  spectrum estimate: extending the applicability of cross-validation methods},
  \bibinfo{journal}{J. Amer. Statist. Assoc.} \bibinfo{volume}{80}
  (\bibinfo{year}{1985}) \bibinfo{pages}{933--940}.
  \bibinfo{note}{\href{http://www.ams.org/mathscinet-getitem?mr=MR819597}{MR819597}}.
\bibitem[{Igarashi(2016{\natexlab{a}})}]{MR3463548}
\bibinfo{author}{G.~Igarashi}, \bibinfo{title}{Bias reductions for beta kernel
  estimation}, \bibinfo{journal}{J. Nonparametr. Stat.} \bibinfo{volume}{28}
  (\bibinfo{year}{2016}{\natexlab{a}}) \bibinfo{pages}{1--30}.
  \bibinfo{note}{\href{http://www.ams.org/mathscinet-getitem?mr=MR3463548}{MR3463548}}.
\bibitem[{Igarashi(2016{\natexlab{b}})}]{MR3540109}
\bibinfo{author}{G.~Igarashi}, \bibinfo{title}{Weighted log-normal kernel
  density estimation}, \bibinfo{journal}{Comm. Statist. Theory Methods}
  \bibinfo{volume}{45} (\bibinfo{year}{2016}{\natexlab{b}})
  \bibinfo{pages}{6670--6687}.
  \bibinfo{note}{\href{http://www.ams.org/mathscinet-getitem?mr=MR3540109}{MR3540109}}.
\bibitem[{Igarashi(2018)}]{MR3850066}
\bibinfo{author}{G.~Igarashi}, \bibinfo{title}{Multivariate density estimation
  using a multivariate weighted log-normal kernel}, \bibinfo{journal}{Sankhya
  A} \bibinfo{volume}{80} (\bibinfo{year}{2018}) \bibinfo{pages}{247--266}.
  \bibinfo{note}{\href{http://www.ams.org/mathscinet-getitem?mr=MR3850066}{MR3850066}}.
\bibitem[{Igarashi(2020)}]{MR4076244}
\bibinfo{author}{G.~Igarashi}, \bibinfo{title}{Nonparametric direct density
  ratio estimation using beta kernel}, \bibinfo{journal}{Statistics}
  \bibinfo{volume}{54} (\bibinfo{year}{2020}) \bibinfo{pages}{257--280}.
  \bibinfo{note}{\href{http://www.ams.org/mathscinet-getitem?mr=MR4076244}{MR4076244}}.
\bibitem[{Igarashi and Kakizawa(2014{\natexlab{a}})}]{MR3174309}
\bibinfo{author}{G.~Igarashi}, \bibinfo{author}{Y.~Kakizawa},
  \bibinfo{title}{On improving convergence rate of {B}ernstein polynomial
  density estimator}, \bibinfo{journal}{J. Nonparametr. Stat.}
  \bibinfo{volume}{26} (\bibinfo{year}{2014}{\natexlab{a}})
  \bibinfo{pages}{61--84}.
  \bibinfo{note}{\href{http://www.ams.org/mathscinet-getitem?mr=MR3174309}{MR3174309}}.
\bibitem[{Igarashi and Kakizawa(2014{\natexlab{b}})}]{MR3131281}
\bibinfo{author}{G.~Igarashi}, \bibinfo{author}{Y.~Kakizawa},
  \bibinfo{title}{Re-formulation of inverse {G}aussian, reciprocal inverse
  {G}aussian, and {B}irnbaum-{S}aunders kernel estimators},
  \bibinfo{journal}{Statist. Probab. Lett.} \bibinfo{volume}{84}
  (\bibinfo{year}{2014}{\natexlab{b}}) \bibinfo{pages}{235--246}.
  \bibinfo{note}{\href{http://www.ams.org/mathscinet-getitem?mr=MR3131281}{MR3131281}}.
\bibitem[{Igarashi and Kakizawa(2015)}]{MR3299088}
\bibinfo{author}{G.~Igarashi}, \bibinfo{author}{Y.~Kakizawa},
  \bibinfo{title}{Bias corrections for some asymmetric kernel estimators},
  \bibinfo{journal}{J. Statist. Plann. Inference} \bibinfo{volume}{159}
  (\bibinfo{year}{2015}) \bibinfo{pages}{37--63}.
  \bibinfo{note}{\href{http://www.ams.org/mathscinet-getitem?mr=MR3299088}{MR3299088}}.
\bibitem[{Igarashi and Kakizawa(2018{\natexlab{a}})}]{MR3843043}
\bibinfo{author}{G.~Igarashi}, \bibinfo{author}{Y.~Kakizawa},
  \bibinfo{title}{Generalised gamma kernel density estimation for nonnegative
  data and its bias reduction}, \bibinfo{journal}{J. Nonparametr. Stat.}
  \bibinfo{volume}{30} (\bibinfo{year}{2018}{\natexlab{a}})
  \bibinfo{pages}{598--639}.
  \bibinfo{note}{\href{http://www.ams.org/mathscinet-getitem?mr=MR3843043}{MR3843043}}.
\bibitem[{Igarashi and Kakizawa(2018{\natexlab{b}})}]{MR3833873}
\bibinfo{author}{G.~Igarashi}, \bibinfo{author}{Y.~Kakizawa},
  \bibinfo{title}{Limiting bias-reduced {A}moroso kernel density estimators for
  non-negative data}, \bibinfo{journal}{Comm. Statist. Theory Methods}
  \bibinfo{volume}{47} (\bibinfo{year}{2018}{\natexlab{b}})
  \bibinfo{pages}{4905--4937}.
  \bibinfo{note}{\href{http://www.ams.org/mathscinet-getitem?mr=MR3833873}{MR3833873}}.
\bibitem[{Igarashi and Kakizawa(2020{\natexlab{a}})}]{MR4136585}
\bibinfo{author}{G.~Igarashi}, \bibinfo{author}{Y.~Kakizawa},
  \bibinfo{title}{Higher-order bias corrections for kernel type density
  estimators on the unit or semi-infinite interval}, \bibinfo{journal}{J.
  Nonparametr. Stat.} \bibinfo{volume}{32} (\bibinfo{year}{2020}{\natexlab{a}})
  \bibinfo{pages}{617--647}.
  \bibinfo{note}{\href{http://www.ams.org/mathscinet-getitem?mr=MR4136585}{MR4136585}}.
\bibitem[{Igarashi and Kakizawa(2020{\natexlab{b}})}]{MR3979322}
\bibinfo{author}{G.~Igarashi}, \bibinfo{author}{Y.~Kakizawa},
  \bibinfo{title}{Multiplicative bias correction for asymmetric kernel density
  estimators revisited}, \bibinfo{journal}{Comput. Statist. Data Anal.}
  \bibinfo{volume}{141} (\bibinfo{year}{2020}{\natexlab{b}})
  \bibinfo{pages}{40--61}.
  \bibinfo{note}{\href{http://www.ams.org/mathscinet-getitem?mr=MR3979322}{MR3979322}}.
\bibitem[{Jeon and Kim(2013)}]{MR3130451}
\bibinfo{author}{Y.~Jeon}, \bibinfo{author}{J.~H.~T. Kim}, \bibinfo{title}{A
  gamma kernel density estimation for insurance loss data},
  \bibinfo{journal}{Insurance Math. Econom.} \bibinfo{volume}{53}
  (\bibinfo{year}{2013}) \bibinfo{pages}{569--579}.
  \bibinfo{note}{\href{http://www.ams.org/mathscinet-getitem?mr=MR3130451}{MR3130451}}.
\bibitem[{Jin and Kawczak(2003)}]{Jin_Kawczak_2003}
\bibinfo{author}{X.~Jin}, \bibinfo{author}{J.~Kawczak},
  \bibinfo{title}{{B}irnbaum-{S}aunders and lognormal kernel estimators for
  modelling durations in high frequency financial data}, \bibinfo{journal}{Ann.
  Econ. Finance} \bibinfo{volume}{4} (\bibinfo{year}{2003})
  \bibinfo{pages}{103--124}. \bibinfo{note}{\newline[URL]~
  \href{http://aeconf.com/Articles/May2003/aef040106.pdf}{http://aeconf.com/Articles/May2003/aef040106.pdf}}.
\bibitem[{Jones(1993)}]{doi:10.1007/BF00147776}
\bibinfo{author}{M.~C. Jones}, \bibinfo{title}{Simple boundary correction for
  kernel density estimation}, \bibinfo{journal}{Stat. Comput.}
  \bibinfo{volume}{3} (\bibinfo{year}{1993}) \bibinfo{pages}{135--146}.
  \bibinfo{note}{\href{https://doi.org/10.1007/BF00147776}{doi:10.1007/BF00147776}}.
\bibitem[{Jones and Foster(1993)}]{MR1272163}
\bibinfo{author}{M.~C. Jones}, \bibinfo{author}{P.~J. Foster},
  \bibinfo{title}{Generalized jackknifing and higher order kernels},
  \bibinfo{journal}{J. Nonparametr. Statist.} \bibinfo{volume}{3}
  (\bibinfo{year}{1993}) \bibinfo{pages}{81--94}.
  \bibinfo{note}{\href{http://www.ams.org/mathscinet-getitem?mr=MR1272163}{MR1272163}}.
\bibitem[{Jones and Foster(1996)}]{MR1422417}
\bibinfo{author}{M.~C. Jones}, \bibinfo{author}{P.~J. Foster},
  \bibinfo{title}{A simple nonnegative boundary correction method for kernel
  density estimation}, \bibinfo{journal}{Statist. Sinica} \bibinfo{volume}{6}
  (\bibinfo{year}{1996}) \bibinfo{pages}{1005--1013}.
  \bibinfo{note}{\href{http://www.ams.org/mathscinet-getitem?mr=MR1422417}{MR1422417}}.
\bibitem[{Jones and Henderson(2007)}]{MR2416803}
\bibinfo{author}{M.~C. Jones}, \bibinfo{author}{D.~A. Henderson},
  \bibinfo{title}{Kernel-type density estimation on the unit interval},
  \bibinfo{journal}{Biometrika} \bibinfo{volume}{94} (\bibinfo{year}{2007})
  \bibinfo{pages}{977--984}.
  \bibinfo{note}{\href{http://www.ams.org/mathscinet-getitem?mr=MR2416803}{MR2416803}}.
\bibitem[{Jones et~al.(1995)Jones, Linton and Nielsen}]{MR1354232}
\bibinfo{author}{M.~C. Jones}, \bibinfo{author}{O.~Linton},
  \bibinfo{author}{J.~P. Nielsen}, \bibinfo{title}{A simple bias reduction
  method for density estimation}, \bibinfo{journal}{Biometrika}
  \bibinfo{volume}{82} (\bibinfo{year}{1995}) \bibinfo{pages}{327--338}.
  \bibinfo{note}{\href{http://www.ams.org/mathscinet-getitem?mr=MR1354232}{MR1354232}}.
\bibitem[{Kakizawa(2004)}]{MR2068610}
\bibinfo{author}{Y.~Kakizawa}, \bibinfo{title}{{B}ernstein polynomial
  probability density estimation}, \bibinfo{journal}{J. Nonparametr. Stat.}
  \bibinfo{volume}{16} (\bibinfo{year}{2004}) \bibinfo{pages}{709--729}.
  \bibinfo{note}{\href{http://www.ams.org/mathscinet-getitem?mr=MR2068610}{MR2068610}}.
\bibitem[{Kakizawa(2018)}]{MR3713468}
\bibinfo{author}{Y.~Kakizawa}, \bibinfo{title}{Nonparametric density estimation
  for nonnegative data, using symmetrical-based inverse and reciprocal inverse
  {G}aussian kernels through dual transformation}, \bibinfo{journal}{J.
  Statist. Plann. Inference} \bibinfo{volume}{193} (\bibinfo{year}{2018})
  \bibinfo{pages}{117--135}.
  \bibinfo{note}{\href{http://www.ams.org/mathscinet-getitem?mr=MR3713468}{MR3713468}}.
\bibitem[{Kakizawa(2020)}]{MR4096263}
\bibinfo{author}{Y.~Kakizawa}, \bibinfo{title}{Multivariate non-central
  {B}irnbaum-{S}aunders kernel density estimator for nonnegative data},
  \bibinfo{journal}{J. Statist. Plann. Inference} \bibinfo{volume}{209}
  (\bibinfo{year}{2020}) \bibinfo{pages}{187--207}.
  \bibinfo{note}{\href{http://www.ams.org/mathscinet-getitem?mr=MR4096263}{MR4096263}}.
\bibitem[{Kakizawa(2021{\natexlab{a}})}]{MR4244643}
\bibinfo{author}{Y.~Kakizawa}, \bibinfo{title}{A class of {B}irnbaum-{S}aunders
  type kernel density estimators for nonnegative data},
  \bibinfo{journal}{Comput. Statist. Data Anal.} \bibinfo{volume}{161}
  (\bibinfo{year}{2021}{\natexlab{a}}) \bibinfo{pages}{107249, 18pp.}
  \bibinfo{note}{\href{http://www.ams.org/mathscinet-getitem?mr=MR4244643}{MR4244643}}.
\bibitem[{Kakizawa(2021{\natexlab{b}})}]{MR4279948}
\bibinfo{author}{Y.~Kakizawa}, \bibinfo{title}{Recursive asymmetric kernel
  density estimation for nonnegative data}, \bibinfo{journal}{J. Nonparametr.
  Stat.} \bibinfo{volume}{33} (\bibinfo{year}{2021}{\natexlab{b}})
  \bibinfo{pages}{197--224}.
  \bibinfo{note}{\href{http://www.ams.org/mathscinet-getitem?mr=MR4279948}{MR4279948}}.
\bibitem[{Kakizawa and Igarashi(2017)}]{MR3648359}
\bibinfo{author}{Y.~Kakizawa}, \bibinfo{author}{G.~Igarashi},
  \bibinfo{title}{Inverse gamma kernel density estimation for nonnegative
  data}, \bibinfo{journal}{J. Korean Statist. Soc.} \bibinfo{volume}{46}
  (\bibinfo{year}{2017}) \bibinfo{pages}{194--207}.
  \bibinfo{note}{\href{http://www.ams.org/mathscinet-getitem?mr=MR3648359}{MR3648359}}.
\bibitem[{Kokonendji and Libengu\'{e} Dob\'{e}l\'{e}-Kpoka(2018)}]{MR3885552}
\bibinfo{author}{C.~C. Kokonendji}, \bibinfo{author}{F.~G.~B. Libengu\'{e}
  Dob\'{e}l\'{e}-Kpoka}, \bibinfo{title}{Asymptotic results for continuous
  associated kernel estimators of density functions}, \bibinfo{journal}{Afr.
  Diaspora J. Math.} \bibinfo{volume}{21} (\bibinfo{year}{2018})
  \bibinfo{pages}{87--97}.
  \bibinfo{note}{\href{http://www.ams.org/mathscinet-getitem?mr=MR3885552}{MR3885552}}.
\bibitem[{Kokonendji and Senga~Kiess\'{e}(2011)}]{MR2834036}
\bibinfo{author}{C.~C. Kokonendji}, \bibinfo{author}{T.~Senga~Kiess\'{e}},
  \bibinfo{title}{Discrete associated kernels method and extensions},
  \bibinfo{journal}{Stat. Methodol.} \bibinfo{volume}{8} (\bibinfo{year}{2011})
  \bibinfo{pages}{497--516}.
  \bibinfo{note}{\href{http://www.ams.org/mathscinet-getitem?mr=MR2834036}{MR2834036}}.
\bibitem[{Kokonendji et~al.(2009)Kokonendji, Senga~Kiess\'{e} and
  Balakrishnan}]{MR2549110}
\bibinfo{author}{C.~C. Kokonendji}, \bibinfo{author}{T.~Senga~Kiess\'{e}},
  \bibinfo{author}{N.~Balakrishnan}, \bibinfo{title}{Semiparametric estimation
  for count data through weighted distributions}, \bibinfo{journal}{J. Statist.
  Plann. Inference} \bibinfo{volume}{139} (\bibinfo{year}{2009})
  \bibinfo{pages}{3625--3638}.
  \bibinfo{note}{\href{http://www.ams.org/mathscinet-getitem?mr=MR2549110}{MR2549110}}.
\bibitem[{Kokonendji and Som\'{e}(2018)}]{MR3760293}
\bibinfo{author}{C.~C. Kokonendji}, \bibinfo{author}{S.~M. Som\'{e}},
  \bibinfo{title}{On multivariate associated kernels to estimate general
  density functions}, \bibinfo{journal}{J. Korean Statist. Soc.}
  \bibinfo{volume}{47} (\bibinfo{year}{2018}) \bibinfo{pages}{112--126}.
  \bibinfo{note}{\href{http://www.ams.org/mathscinet-getitem?mr=MR3760293}{MR3760293}}.
\bibitem[{Kokonendji and Som\'{e}(2021)}]{Kokonendji_Some_2021}
\bibinfo{author}{C.~C. Kokonendji}, \bibinfo{author}{S.~M. Som\'{e}},
  \bibinfo{title}{Bayesian bandwidths in semiparametric modelling for
  nonnegative orthant data with diagnostics}, \bibinfo{journal}{Stats}
  \bibinfo{volume}{4} (\bibinfo{year}{2021}) \bibinfo{pages}{162--183}.
  \bibinfo{note}{\href{https://www.doi.org/10.3390/stats4010013}{doi:10.3390/stats4010013}}.
\bibitem[{Kokonendji and Varron(2016)}]{MR3474762}
\bibinfo{author}{C.~C. Kokonendji}, \bibinfo{author}{D.~Varron},
  \bibinfo{title}{Performance of discrete associated kernel estimators through
  the total variation distance}, \bibinfo{journal}{Statist. Probab. Lett.}
  \bibinfo{volume}{110} (\bibinfo{year}{2016}) \bibinfo{pages}{225--235}.
  \bibinfo{note}{\href{http://www.ams.org/mathscinet-getitem?mr=MR3474762}{MR3474762}}.
\bibitem[{Koul and Song(2013)}]{MR3174299}
\bibinfo{author}{H.~L. Koul}, \bibinfo{author}{W.~Song}, \bibinfo{title}{Large
  sample results for varying kernel regression estimates}, \bibinfo{journal}{J.
  Nonparametr. Stat.} \bibinfo{volume}{25} (\bibinfo{year}{2013})
  \bibinfo{pages}{829--853}.
  \bibinfo{note}{\href{http://www.ams.org/mathscinet-getitem?mr=MR3174299}{MR3174299}}.
\bibitem[{Kristensen(2010)}]{MR2587103}
\bibinfo{author}{D.~Kristensen}, \bibinfo{title}{Nonparametric filtering of the
  realized spot volatility: a kernel-based approach},
  \bibinfo{journal}{Econometric Theory} \bibinfo{volume}{26}
  (\bibinfo{year}{2010}) \bibinfo{pages}{60--93}.
  \bibinfo{note}{\href{http://www.ams.org/mathscinet-getitem?mr=MR2587103}{MR2587103}}.
\bibitem[{Kulasekera and Padgett(2006)}]{MR2229885}
\bibinfo{author}{K.~B. Kulasekera}, \bibinfo{author}{W.~J. Padgett},
  \bibinfo{title}{Bayes bandwidth selection in kernel density estimation with
  censored data}, \bibinfo{journal}{J. Nonparametr. Stat.} \bibinfo{volume}{18}
  (\bibinfo{year}{2006}) \bibinfo{pages}{129--143}.
  \bibinfo{note}{\href{http://www.ams.org/mathscinet-getitem?mr=MR2229885}{MR2229885}}.
\bibitem[{Kuruwita et~al.(2010)Kuruwita, Kulasekera and Padgett}]{MR2606717}
\bibinfo{author}{C.~N. Kuruwita}, \bibinfo{author}{K.~B. Kulasekera},
  \bibinfo{author}{W.~J. Padgett}, \bibinfo{title}{Density estimation using
  asymmetric kernels and {B}ayes bandwidths with censored data},
  \bibinfo{journal}{J. Statist. Plann. Inference} \bibinfo{volume}{140}
  (\bibinfo{year}{2010}) \bibinfo{pages}{1765--1774}.
  \bibinfo{note}{\href{http://www.ams.org/mathscinet-getitem?mr=MR2606717}{MR2606717}}.
\bibitem[{Leblanc(2010)}]{MR2662607}
\bibinfo{author}{A.~Leblanc}, \bibinfo{title}{A bias-reduced approach to
  density estimation using {B}ernstein polynomials}, \bibinfo{journal}{J.
  Nonparametr. Stat.} \bibinfo{volume}{22} (\bibinfo{year}{2010})
  \bibinfo{pages}{459--475}.
  \bibinfo{note}{\href{http://www.ams.org/mathscinet-getitem?mr=MR2662607}{MR2662607}}.
\bibitem[{Leblanc(2012)}]{MR2925964}
\bibinfo{author}{A.~Leblanc}, \bibinfo{title}{On the boundary properties of
  {B}ernstein polynomial estimators of density and distribution functions},
  \bibinfo{journal}{J. Statist. Plann. Inference} \bibinfo{volume}{142}
  (\bibinfo{year}{2012}) \bibinfo{pages}{2762--2778}.
  \bibinfo{note}{\href{http://www.ams.org/mathscinet-getitem?mr=MR2925964}{MR2925964}}.
\bibitem[{Lejeune and Sarda(1992)}]{MR1192215}
\bibinfo{author}{M.~Lejeune}, \bibinfo{author}{P.~Sarda},
  \bibinfo{title}{Smooth estimators of distribution and density functions},
  \bibinfo{journal}{Comput. Statist. Data Anal.} \bibinfo{volume}{14}
  (\bibinfo{year}{1992}) \bibinfo{pages}{457--471}.
  \bibinfo{note}{\href{http://www.ams.org/mathscinet-getitem?mr=MR1192215}{MR1192215}}.
\bibitem[{Lepski\u{\i}(1991)}]{MR1147167}
\bibinfo{author}{O.~V. Lepski\u{\i}}, \bibinfo{title}{Asymptotically minimax
  adaptive estimation. {I}. {U}pper bounds. {O}ptimally adaptive estimates},
  \bibinfo{journal}{Teor. Veroyatnost. i Primenen.} \bibinfo{volume}{36}
  (\bibinfo{year}{1991}) \bibinfo{pages}{645--659}.
  \bibinfo{note}{\href{http://www.ams.org/mathscinet-getitem?mr=MR1147167}{MR1147167}}.
\bibitem[{Li et~al.(2019{\natexlab{a}})Li, Xiao and Shi}]{MR3937087}
\bibinfo{author}{X.~Li}, \bibinfo{author}{J.~Xiao}, \bibinfo{author}{J.~Shi},
  \bibinfo{title}{Statistical inference in the partial linear models with the
  inverse {G}aussian kernel}, \bibinfo{journal}{Comm. Statist. Simulation
  Comput.} \bibinfo{volume}{48} (\bibinfo{year}{2019}{\natexlab{a}})
  \bibinfo{pages}{240--263}.
  \bibinfo{note}{\href{http://www.ams.org/mathscinet-getitem?mr=MR3937087}{MR3937087}}.
\bibitem[{Li et~al.(2019{\natexlab{b}})Li, Xiao, Song and Shi}]{MR3975162}
\bibinfo{author}{X.~Li}, \bibinfo{author}{J.~Xiao}, \bibinfo{author}{W.~Song},
  \bibinfo{author}{J.~Shi}, \bibinfo{title}{Local linear regression with
  reciprocal inverse {G}aussian kernel}, \bibinfo{journal}{Metrika}
  \bibinfo{volume}{82} (\bibinfo{year}{2019}{\natexlab{b}})
  \bibinfo{pages}{733--758}.
  \bibinfo{note}{\href{http://www.ams.org/mathscinet-getitem?mr=MR3975162}{MR3975162}}.
\bibitem[{Libengu\'{e} Dob\'{e}l\'{e}-Kpoka and Kokonendji(2017)}]{MR3743306}
\bibinfo{author}{F.~G.~B. Libengu\'{e} Dob\'{e}l\'{e}-Kpoka},
  \bibinfo{author}{C.~C. Kokonendji}, \bibinfo{title}{The mode-dispersion
  approach for constructing continuous associated kernels},
  \bibinfo{journal}{Afr. Stat.} \bibinfo{volume}{12} (\bibinfo{year}{2017})
  \bibinfo{pages}{1417--1446}.
  \bibinfo{note}{\href{http://www.ams.org/mathscinet-getitem?mr=MR3743306}{MR3743306}}.
\bibitem[{Liu and Ghosh(2020)}]{MR4130895}
\bibinfo{author}{B.~Liu}, \bibinfo{author}{S.~K. Ghosh}, \bibinfo{title}{On
  empirical estimation of mode based on weakly dependent samples},
  \bibinfo{journal}{Comput. Statist. Data Anal.} \bibinfo{volume}{152}
  (\bibinfo{year}{2020}) \bibinfo{pages}{107046, 21pp.}
  \bibinfo{note}{\href{http://www.ams.org/mathscinet-getitem?mr=MR4130895}{MR4130895}}.
\bibitem[{Lu(2015)}]{MR3412755}
\bibinfo{author}{L.~Lu}, \bibinfo{title}{On the uniform consistency of the
  {B}ernstein density estimator}, \bibinfo{journal}{Statist. Probab. Lett.}
  \bibinfo{volume}{107} (\bibinfo{year}{2015}) \bibinfo{pages}{52--61}.
  \bibinfo{note}{\href{http://www.ams.org/mathscinet-getitem?mr=MR3412755}{MR3412755}}.
\bibitem[{Ma(2019)}]{Ma2019master}
\bibinfo{author}{X.~Ma}, \bibinfo{title}{On {G}amma {K}ernel {F}unction in
  {R}ecursive {D}ensity {E}stimation}, \bibinfo{type}{Master's thesis},
  Mississippi State University, \bibinfo{year}{2019}.
  \bibinfo{note}{\newline[URL]~ \url{https://hdl.handle.net/11668/14483}}.
\bibitem[{Malec and Schienle(2014)}]{MR3139348}
\bibinfo{author}{P.~Malec}, \bibinfo{author}{M.~Schienle},
  \bibinfo{title}{Nonparametric kernel density estimation near the boundary},
  \bibinfo{journal}{Comput. Statist. Data Anal.} \bibinfo{volume}{72}
  (\bibinfo{year}{2014}) \bibinfo{pages}{57--76}.
  \bibinfo{note}{\href{http://www.ams.org/mathscinet-getitem?mr=MR3139348}{MR3139348}}.
\bibitem[{Manivong(2009)}]{Manivong2009master}
\bibinfo{author}{P.~Manivong}, \bibinfo{title}{Estimating {M}ultinomial {C}ell
  {P}robabilities {U}sing {N}ormalized {B}eta {K}ernels},
  \bibinfo{type}{Master's thesis}, University of Manitoba,
  \bibinfo{year}{2009}. \bibinfo{note}{\newline[URL]~
  \url{http://hdl.handle.net/1993/21676}}.
\bibitem[{Marchant et~al.(2013)Marchant, Bertin, Leiva and Saulo}]{MR3040246}
\bibinfo{author}{C.~Marchant}, \bibinfo{author}{K.~Bertin},
  \bibinfo{author}{V.~Leiva}, \bibinfo{author}{H.~Saulo},
  \bibinfo{title}{Generalized {B}irnbaum-{S}aunders kernel density estimators
  and an analysis of financial data}, \bibinfo{journal}{Comput. Statist. Data
  Anal.} \bibinfo{volume}{63} (\bibinfo{year}{2013}) \bibinfo{pages}{1--15}.
  \bibinfo{note}{\href{http://www.ams.org/mathscinet-getitem?mr=MR3040246}{MR3040246}}.
\bibitem[{Markovich(2018{\natexlab{a}})}]{MR3933005}
\bibinfo{author}{L.~Markovich}, \bibinfo{title}{Light-~and~heavy-tailed density
  estimation by gamma-{W}eibull kernel}, in: \bibinfo{booktitle}{Nonparametric
  statistics}, volume \bibinfo{volume}{250} of \text{\bibinfo{series}{Springer
  Proc. Math. Stat.}}, \bibinfo{publisher}{Springer, Cham},
  \bibinfo{year}{2018}{\natexlab{a}}, pp. \bibinfo{pages}{145--158}.
  \bibinfo{note}{\href{http://www.ams.org/mathscinet-getitem?mr=MR3933005}{MR3933005}}.
\bibitem[{Markovich(2016)}]{MR3520774}
\bibinfo{author}{L.~A. Markovich}, \bibinfo{title}{Gamma kernel estimation of
  the density derivative on the positive semi-axis by dependent data},
  \bibinfo{journal}{REVSTAT} \bibinfo{volume}{14} (\bibinfo{year}{2016})
  \bibinfo{pages}{327--348}.
  \bibinfo{note}{\href{http://www.ams.org/mathscinet-getitem?mr=MR3962366}{MR3962366}}.
\bibitem[{Markovich(2018{\natexlab{b}})}]{MR3962366}
\bibinfo{author}{L.~A. Markovich}, \bibinfo{title}{Gamma kernel estimates for
  multivariate density and its partial derivative with respect to dependent
  data [in {R}ussian]}, \bibinfo{journal}{Fundam. Prikl. Mat.}
  \bibinfo{volume}{22} (\bibinfo{year}{2018}{\natexlab{b}})
  \bibinfo{pages}{145--177}.
  \bibinfo{note}{\href{http://www.ams.org/mathscinet-getitem?mr=MR3962366}{MR3962366}}.
\bibitem[{Markovich(2021)}]{doi:10.1007/s10958-021-05325-2}
\bibinfo{author}{L.~A. Markovich}, \bibinfo{title}{Nonparametric estimation of
  multivariate density and its derivative by dependent data using gamma
  kernels}, \bibinfo{journal}{J. Math. Sci.} \bibinfo{volume}{254}
  (\bibinfo{year}{2021}) \bibinfo{pages}{550--573}.
  \bibinfo{note}{\href{https://www.doi.org/10.1007/s10958-021-05325-2}{doi:10.1007/s10958-021-05325-2}}.
\bibitem[{Lafaye~de Micheaux and Ouimet(2020)}]{arXiv:2011.14893}
\bibinfo{author}{P.~Lafaye~de Micheaux}, \bibinfo{author}{F.~Ouimet},
  \bibinfo{title}{A study of seven asymmetric kernels for the estimation of
  cumulative distribution functions}, \bibinfo{journal}{Preprint}
  (\bibinfo{year}{2020}) \bibinfo{pages}{1--38}.
  \bibinfo{note}{\href{https://arxiv.org/abs/2011.14893}{arXiv:2011.14893}}.
\bibitem[{Minc and Sathre(6465)}]{MR162751}
\bibinfo{author}{H.~Minc}, \bibinfo{author}{L.~Sathre}, \bibinfo{title}{Some
  inequalities involving {$(r!)^{1/r}$}}, \bibinfo{journal}{Proc. Edinburgh
  Math. Soc. (2)} \bibinfo{volume}{14} (\bibinfo{year}{1964/65})
  \bibinfo{pages}{41--46}.
  \bibinfo{note}{\href{http://www.ams.org/mathscinet-getitem?mr=MR162751}{MR162751}}.
\bibitem[{Mnatsakanov and Sarkisian(2012)}]{MR2929784}
\bibinfo{author}{R.~Mnatsakanov}, \bibinfo{author}{K.~Sarkisian},
  \bibinfo{title}{Varying kernel density estimation on {$\mathbb{R_+}$}},
  \bibinfo{journal}{Statist. Probab. Lett.} \bibinfo{volume}{82}
  (\bibinfo{year}{2012}) \bibinfo{pages}{1337--1345}.
  \bibinfo{note}{\href{http://www.ams.org/mathscinet-getitem?mr=MR2929784}{MR2929784}}.
\bibitem[{Mnatsakanov and Ruymgaart(2012)}]{MR2903523}
\bibinfo{author}{R.~M. Mnatsakanov}, \bibinfo{author}{F.~H. Ruymgaart},
  \bibinfo{title}{Moment density estimation for positive random variables},
  \bibinfo{journal}{Statistics} \bibinfo{volume}{46} (\bibinfo{year}{2012})
  \bibinfo{pages}{215--230}.
  \bibinfo{note}{\href{http://www.ams.org/mathscinet-getitem?mr=MR2903523}{MR2903523}}.
\bibitem[{Mombeni et~al.(2019)Mombeni, Masouri and
  Akhoond}]{Mombeni_et_al_2019_accepted}
\bibinfo{author}{H.~A. Mombeni}, \bibinfo{author}{B.~Masouri},
  \bibinfo{author}{M.~R. Akhoond}, \bibinfo{title}{Asymmetric kernels for
  boundary modification in distribution function estimation},
  \bibinfo{journal}{In press. To appear in REVSTAT}  (\bibinfo{year}{2019})
  \bibinfo{pages}{1--27}.
  \bibinfo{note}{\href{https://www.ine.pt/revstat/pdf/Asymmetrickernelsforboundarymodificationindistributionfunctionestimation.pdf}{Available
  here}}.
\bibitem[{Mousa et~al.(2016)Mousa, Hassan and Fathi}]{MR3544186}
\bibinfo{author}{A.~M. Mousa}, \bibinfo{author}{M.~K. Hassan},
  \bibinfo{author}{A.~Fathi}, \bibinfo{title}{A new non parametric estimator
  for pdf based on inverse gamma distribution}, \bibinfo{journal}{Comm.
  Statist. Theory Methods} \bibinfo{volume}{45} (\bibinfo{year}{2016})
  \bibinfo{pages}{7002--7010}.
  \bibinfo{note}{\href{http://www.ams.org/mathscinet-getitem?mr=MR3544186}{MR3544186}}.
\bibitem[{Ng et~al.(2011)Ng, Tian and Tang}]{MR2830563}
\bibinfo{author}{K.~W. Ng}, \bibinfo{author}{G.-L. Tian},
  \bibinfo{author}{M.-L. Tang}, \bibinfo{title}{Dirichlet and {R}elated
  {D}istributions}, Wiley Series in Probability and Statistics,
  \bibinfo{publisher}{John Wiley \& Sons, Ltd., Chichester},
  \bibinfo{year}{2011}.
  \bibinfo{note}{\href{http://www.ams.org/mathscinet-getitem?mr=MR2830563}{MR2830563}}.
\bibitem[{Ouimet(2021{\natexlab{a}})}]{MR4287788}
\bibinfo{author}{F.~Ouimet}, \bibinfo{title}{Asymptotic properties of
  {B}ernstein estimators on the simplex}, \bibinfo{journal}{J. Multivariate
  Anal.} \bibinfo{volume}{185} (\bibinfo{year}{2021}{\natexlab{a}})
  \bibinfo{pages}{104784, 20 pp.}
  \bibinfo{note}{\href{http://www.ams.org/mathscinet-getitem?mr=MR4287788}{MR4287788}}.
\bibitem[{Ouimet(2021{\natexlab{b}})}]{arXiv:2006.11756}
\bibinfo{author}{F.~Ouimet}, \bibinfo{title}{On the boundary properties of
  {B}ernstein estimators on the simplex}, \bibinfo{journal}{Preprint}
  (\bibinfo{year}{2021}{\natexlab{b}}) \bibinfo{pages}{1--11}.
  \bibinfo{note}{\href{https://arxiv.org/abs/2006.11756}{arXiv:2006.11756}}.
\bibitem[{Ouimet(2021{\natexlab{c}})}]{MR4213687}
\bibinfo{author}{F.~Ouimet}, \bibinfo{title}{On the {L}e {C}am distance between
  {P}oisson and {G}aussian experiments and the asymptotic properties of {S}zasz
  estimators}, \bibinfo{journal}{J. Math. Anal. Appl.} \bibinfo{volume}{499}
  (\bibinfo{year}{2021}{\natexlab{c}}) \bibinfo{pages}{125033, 18pp.}
  \bibinfo{note}{\href{http://www.ams.org/mathscinet-getitem?mr=MR4213687}{MR4213687}}.
\bibitem[{Ouimet(2021{\natexlab{d}})}]{arXiv:2104.04882}
\bibinfo{author}{F.~Ouimet}, \bibinfo{title}{A symmetric matrix-variate normal
  local approximation for the {W}ishart distribution and some applications},
  \bibinfo{journal}{Preprint}  (\bibinfo{year}{2021}{\natexlab{d}})
  \bibinfo{pages}{1--12}.
  \bibinfo{note}{\href{https://arxiv.org/abs/2104.04882}{arXiv:2104.04882}}.
\bibitem[{Ouimet(2022)}]{doi:10.1002/sta4.410}
\bibinfo{author}{F.~Ouimet}, \bibinfo{title}{A multivariate normal
  approximation for the {D}irichlet density and some applications},
  \bibinfo{journal}{Stat} \bibinfo{volume}{11} (\bibinfo{year}{2022})
  \bibinfo{pages}{12 pp.}
  \bibinfo{note}{\href{https://www.doi.org/10.1002/sta4.410}{doi:10.1002/sta4.410}}.
\bibitem[{Pereira et~al.(2021)Pereira, Frenzel, Hoang, Tolosana-Delgado,
  Rudolph and Gutzmer}]{doi:10.1016/j.mineng.2021.107054}
\bibinfo{author}{L.~Pereira}, \bibinfo{author}{M.~Frenzel},
  \bibinfo{author}{D.~H. Hoang}, \bibinfo{author}{R.~Tolosana-Delgado},
  \bibinfo{author}{M.~Rudolph}, \bibinfo{author}{J.~Gutzmer},
  \bibinfo{title}{Computing single-particle flotation kinetics using automated
  mineralogy data and machine learning}, \bibinfo{journal}{Miner. Eng.}
  \bibinfo{volume}{170} (\bibinfo{year}{2021}) \bibinfo{pages}{107054, 10 pp.}
  \bibinfo{note}{\href{https://doi.org/10.1016/j.mineng.2021.107054}{doi:10.1016/j.mineng.2021.107054}}.
\bibitem[{Prakasa~Rao(1983)}]{MR0740865}
\bibinfo{author}{B.~L.~S. Prakasa~Rao}, \bibinfo{title}{Nonparametric
  {F}unctional {E}stimation}, \bibinfo{publisher}{Academic Press, New York},
  \bibinfo{year}{1983}.
  \bibinfo{note}{\href{http://www.ams.org/mathscinet-getitem?mr=MR0740865}{MR0740865}}.
\bibitem[{Renault and Scaillet(2004)}]{doi:10.1016/j.jbankfin.2003.10.018}
\bibinfo{author}{O.~Renault}, \bibinfo{author}{O.~Scaillet}, \bibinfo{title}{On
  the way to recovery: A nonparametric bias free estimation of recovery rate
  densities}, \bibinfo{journal}{J. Bank. Finance} \bibinfo{volume}{28}
  (\bibinfo{year}{2004}) \bibinfo{pages}{2915--2931}.
  \bibinfo{note}{\href{https://doi.org/10.1016/j.jbankfin.2003.10.018}{doi:10.1016/j.jbankfin.2003.10.018}}.
\bibitem[{Salha(2012)}]{doi:10.12988/pms.2014.4616}
\bibinfo{author}{R.~B. Salha}, \bibinfo{title}{Hazard rate function estimation
  using inverse {G}aussian kernel}, \bibinfo{journal}{IUG Journal of Natural
  and Engineering Studies} \bibinfo{volume}{20} (\bibinfo{year}{2012})
  \bibinfo{pages}{73--84}.
  \bibinfo{note}{\href{http://dx.doi.org/10.12988/pms.2014.4616}{doi:10.12988/pms.2014.4616}}.
\bibitem[{Saulo et~al.(2013)Saulo, Leiva, Ziegelmann and
  Marchant}]{doi:10.1007/s00477-012-0684-8}
\bibinfo{author}{H.~Saulo}, \bibinfo{author}{V.~Leiva}, \bibinfo{author}{F.~A.
  Ziegelmann}, \bibinfo{author}{C.~Marchant}, \bibinfo{title}{A nonparametric
  method for estimating asymmetric densities based on skewed birnbaum-saunders
  distributions applied to environmental data}, \bibinfo{journal}{Stoch.
  Environ. Res. Risk Assess.} \bibinfo{volume}{27} (\bibinfo{year}{2013})
  \bibinfo{pages}{1479--1491}.
  \bibinfo{note}{\href{https://doi.org/10.1007/s00477-012-0684-8}{doi:10.1007/s00477-012-0684-8}}.
\bibitem[{Scaillet(2004)}]{MR2053071}
\bibinfo{author}{O.~Scaillet}, \bibinfo{title}{Density estimation using inverse
  and reciprocal inverse {G}aussian kernels}, \bibinfo{journal}{J. Nonparametr.
  Stat.} \bibinfo{volume}{16} (\bibinfo{year}{2004}) \bibinfo{pages}{217--226}.
  \bibinfo{note}{\href{http://www.ams.org/mathscinet-getitem?mr=MR2053071}{MR2053071}}.
\bibitem[{Schach et~al.(2019)Schach, Buchmann, Tolosana-Delgado, Lei{\ss}ner,
  Kern, {van den Boogaart}, Rudolph and
  Peuker}]{doi:10.1016/j.mineng.2019.03.026}
\bibinfo{author}{E.~Schach}, \bibinfo{author}{M.~Buchmann},
  \bibinfo{author}{R.~Tolosana-Delgado}, \bibinfo{author}{T.~Lei{\ss}ner},
  \bibinfo{author}{M.~Kern}, \bibinfo{author}{K.~G. {van den Boogaart}},
  \bibinfo{author}{M.~Rudolph}, \bibinfo{author}{U.~A. Peuker},
  \bibinfo{title}{Multidimensional characterization of separation processes –
  {P}art 1: {I}ntroducing kernel methods and entropy in the context of mineral
  processing using {SEM}-based image analysis}, \bibinfo{journal}{Miner. Eng.}
  \bibinfo{volume}{137} (\bibinfo{year}{2019}) \bibinfo{pages}{78--86}.
  \bibinfo{note}{\href{https://www.doi.org/10.1016/j.mineng.2019.03.026}{doi:10.1016/j.mineng.2019.03.026}}.
\bibitem[{Schucany and Sommers(1977)}]{MR448691}
\bibinfo{author}{W.~R. Schucany}, \bibinfo{author}{J.~P. Sommers},
  \bibinfo{title}{Improvement of kernel type density estimators},
  \bibinfo{journal}{J. Amer. Statist. Assoc.} \bibinfo{volume}{72}
  (\bibinfo{year}{1977}) \bibinfo{pages}{420--423}.
  \bibinfo{note}{\href{http://www.ams.org/mathscinet-getitem?mr=MR448691}{MR448691}}.
\bibitem[{Schuster(1985)}]{MR797636}
\bibinfo{author}{E.~F. Schuster}, \bibinfo{title}{Incorporating support
  constraints into nonparametric estimators of densities},
  \bibinfo{journal}{Comm. Statist. A---Theory Methods} \bibinfo{volume}{14}
  (\bibinfo{year}{1985}) \bibinfo{pages}{1123--1136}.
  \bibinfo{note}{\href{http://www.ams.org/mathscinet-getitem?mr=MR797636}{MR797636}}.
\bibitem[{Scott(2015)}]{MR3329609}
\bibinfo{author}{D.~W. Scott}, \bibinfo{title}{Multivariate {D}ensity
  {E}stimation}, Wiley Series in Probability and Statistics,
  \bibinfo{publisher}{John Wiley \& Sons, Inc., Hoboken, NJ},
  \bibinfo{edition}{second} edition, \bibinfo{year}{2015}.
  \bibinfo{note}{\href{http://www.ams.org/mathscinet-getitem?mr=MR3329609}{MR3329609}}.
\bibitem[{Scott and Wand(1991)}]{MR1118245}
\bibinfo{author}{D.~W. Scott}, \bibinfo{author}{M.~P. Wand},
  \bibinfo{title}{Feasibility of multivariate density estimates},
  \bibinfo{journal}{Biometrika} \bibinfo{volume}{78} (\bibinfo{year}{1991})
  \bibinfo{pages}{197--205}.
  \bibinfo{note}{\href{http://www.ams.org/mathscinet-getitem?mr=MR1118245}{MR1118245}}.
\bibitem[{Senga~Kiess\'{e} and Cuny(2014)}]{MR3215725}
\bibinfo{author}{T.~Senga~Kiess\'{e}}, \bibinfo{author}{H.~E. Cuny},
  \bibinfo{title}{Discrete triangular associated kernel and bandwidth choices
  in semiparametric estimation for count data}, \bibinfo{journal}{J. Stat.
  Comput. Simul.} \bibinfo{volume}{84} (\bibinfo{year}{2014})
  \bibinfo{pages}{1813--1829}.
  \bibinfo{note}{\href{http://www.ams.org/mathscinet-getitem?mr=MR3215725}{MR3215725}}.
\bibitem[{Serfling(1980)}]{MR0595165}
\bibinfo{author}{R.~J. Serfling}, \bibinfo{title}{Approximation {T}heorems of
  {M}athematical {S}tatistics}, Wiley Series in Probability and Mathematical
  Statistics, \bibinfo{publisher}{John Wiley \& Sons, Inc., New York},
  \bibinfo{year}{1980}.
  \bibinfo{note}{\href{http://www.ams.org/mathscinet-getitem?mr=MR0595165}{MR0595165}}.
\bibitem[{Shi and Song(2016)}]{MR3494026}
\bibinfo{author}{J.~Shi}, \bibinfo{author}{W.~Song}, \bibinfo{title}{Asymptotic
  results in gamma kernel regression}, \bibinfo{journal}{Comm. Statist. Theory
  Methods} \bibinfo{volume}{45} (\bibinfo{year}{2016})
  \bibinfo{pages}{3489--3509}.
  \bibinfo{note}{\href{http://www.ams.org/mathscinet-getitem?mr=MR3494026}{MR3494026}}.
\bibitem[{Som\'{e} and Kokonendji(2016)}]{MR3499725}
\bibinfo{author}{S.~M. Som\'{e}}, \bibinfo{author}{C.~C. Kokonendji},
  \bibinfo{title}{Effects of associated kernels in nonparametric multiple
  regressions}, \bibinfo{journal}{J. Stat. Theory Pract.} \bibinfo{volume}{10}
  (\bibinfo{year}{2016}) \bibinfo{pages}{456--471}.
  \bibinfo{note}{\href{http://www.ams.org/mathscinet-getitem?mr=MR3499725}{MR3499725}}.
\bibitem[{Som\'e and Kokonendji(2021)}]{doi:10.1080/02664763.2021.1881456}
\bibinfo{author}{S.~M. Som\'e}, \bibinfo{author}{C.~C. Kokonendji},
  \bibinfo{title}{Bayesian selector of adaptive bandwidth for multivariate
  gamma kernel estimator on $[0,\infty)^d$}, \bibinfo{journal}{J. Appl. Stat.}
  (\bibinfo{year}{2021}) \bibinfo{pages}{1--22}.
  \bibinfo{note}{\href{https://www.doi.org/10.1080/02664763.2021.1881456}{doi:10.1080/02664763.2021.1881456}}.
\bibitem[{Som\'{e} et~al.(2016)Som\'{e}, Kokonendji and Ibrahim}]{MR3567789}
\bibinfo{author}{S.~M. Som\'{e}}, \bibinfo{author}{C.~C. Kokonendji},
  \bibinfo{author}{M.~Ibrahim}, \bibinfo{title}{Associated kernel discriminant
  analysis for multivariate mixed data}, \bibinfo{journal}{Electron. J. Appl.
  Stat. Anal.} \bibinfo{volume}{9} (\bibinfo{year}{2016})
  \bibinfo{pages}{385--399}.
  \bibinfo{note}{\href{http://www.ams.org/mathscinet-getitem?mr=MR3567789}{MR3567789}}.
\bibitem[{Song et~al.(2019)Song, Hou and Zhou}]{doi:10.1515/snde-2018-0001}
\bibinfo{author}{Y.~Song}, \bibinfo{author}{W.~Hou}, \bibinfo{author}{S.~Zhou},
  \bibinfo{title}{Variance reduction estimation for return models with jumps
  using gamma asymmetric kernels}, \bibinfo{journal}{Stud. Nonlinear Dyn.
  Econ.} \bibinfo{volume}{23} (\bibinfo{year}{2019}) \bibinfo{pages}{20180001}.
  \bibinfo{note}{\href{https://doi.org/10.1515/snde-2018-0001}{doi:10.1515/snde-2018-0001}}.
\bibitem[{Stadtm\"uller(1983)}]{MR0726014}
\bibinfo{author}{U.~Stadtm\"uller}, \bibinfo{title}{Asymptotic distributions of
  smoothed histograms}, \bibinfo{journal}{Metrika} \bibinfo{volume}{30}
  (\bibinfo{year}{1983}) \bibinfo{pages}{145--158}.
  \bibinfo{note}{\href{http://www.ams.org/mathscinet-getitem?mr=MR0726014}{MR0726014}}.
\bibitem[{Stadtm\"uller(1986)}]{MR0858109}
\bibinfo{author}{U.~Stadtm\"uller}, \bibinfo{title}{Asymptotic properties of
  nonparametric curve estimates}, \bibinfo{journal}{Period. Math. Hungar.}
  \bibinfo{volume}{17} (\bibinfo{year}{1986}) \bibinfo{pages}{83--108}.
  \bibinfo{note}{\href{http://www.ams.org/mathscinet-getitem?mr=MR0858109}{MR0858109}}.
\bibitem[{Star-Lack et~al.(2009)Star-Lack, Sun, Kaestner, Hassanein, Virshup,
  Berkus and Oelhafen}]{doi:10.1117/12.811578}
\bibinfo{author}{J.~Star-Lack}, \bibinfo{author}{M.~Sun},
  \bibinfo{author}{A.~Kaestner}, \bibinfo{author}{R.~Hassanein},
  \bibinfo{author}{G.~Virshup}, \bibinfo{author}{T.~Berkus},
  \bibinfo{author}{M.~Oelhafen}, \bibinfo{title}{Efficient scatter correction
  using asymmetric kernels}, \bibinfo{journal}{Proc. SPIE}
  \bibinfo{volume}{7258} (\bibinfo{year}{2009}) \bibinfo{pages}{12pp.}
  \bibinfo{note}{\href{https://doi.org/10.1117/12.811578}{doi:10.1117/12.811578}}.
\bibitem[{Steele(1997)}]{MR1422018}
\bibinfo{author}{M.~Steele}, \bibinfo{title}{Probability {T}heory and
  {C}ombinatorial {O}ptimization}, volume~\bibinfo{volume}{69} of
  \text{\bibinfo{series}{CBMS-NSF Regional Conference Series in Applied
  Mathematics}}, \bibinfo{publisher}{Society for Industrial and Applied
  Mathematics (SIAM), Philadelphia, PA}, \bibinfo{year}{1997}.
  \bibinfo{note}{\href{http://www.ams.org/mathscinet-getitem?mr=MR1422018}{MR1422018}}.
\bibitem[{Tanabe and Sagae(1992)}]{MR1157720}
\bibinfo{author}{K.~Tanabe}, \bibinfo{author}{M.~Sagae}, \bibinfo{title}{An
  exact {C}holesky decomposition and the generalized inverse of the
  variance-covariance matrix of the multinomial distribution, with
  applications}, \bibinfo{journal}{J. Roy. Statist. Soc. Ser. B}
  \bibinfo{volume}{54} (\bibinfo{year}{1992}) \bibinfo{pages}{211--219}.
  \bibinfo{note}{\href{http://www.ams.org/mathscinet-getitem?mr=MR1157720}{MR1157720}}.
\bibitem[{Tang and Yang(2017)}]{doi:10.3233/JCM-170731}
\bibinfo{author}{F.-X. Tang}, \bibinfo{author}{Y.-F. Yang},
  \bibinfo{title}{Research of color image segmentation algorithm based on
  asymmetric kernel density estimation}, \bibinfo{journal}{J. Comput. Methods
  Sci. Engin.} \bibinfo{volume}{17} (\bibinfo{year}{2017})
  \bibinfo{pages}{455--462}.
  \bibinfo{note}{\href{https://doi.org/10.3233/JCM-170731}{doi:10.3233/JCM-170731}}.
\bibitem[{Tenbusch(1994)}]{MR1293514}
\bibinfo{author}{A.~Tenbusch}, \bibinfo{title}{Two-dimensional {B}ernstein
  polynomial density estimators}, \bibinfo{journal}{Metrika}
  \bibinfo{volume}{41} (\bibinfo{year}{1994}) \bibinfo{pages}{233--253}.
  \bibinfo{note}{\href{http://www.ams.org/mathscinet-getitem?mr=MR1293514}{MR1293514}}.
\bibitem[{Terrell and Scott(1980)}]{MR585714}
\bibinfo{author}{G.~R. Terrell}, \bibinfo{author}{D.~W. Scott},
  \bibinfo{title}{On improving convergence rates for nonnegative kernel density
  estimators}, \bibinfo{journal}{Ann. Statist.} \bibinfo{volume}{8}
  (\bibinfo{year}{1980}) \bibinfo{pages}{1160--1163}.
  \bibinfo{note}{\href{http://www.ams.org/mathscinet-getitem?mr=MR585714}{MR585714}}.
\bibitem[{Tromp(1937)}]{Tromp_1937}
\bibinfo{author}{K.~Tromp}, \bibinfo{title}{{N}eue {W}ege f\"ur die
  {B}eurteilung der {A}ufbereitung von {S}teinkohlen},
  \bibinfo{journal}{Gl\"uckauf} \bibinfo{volume}{6} (\bibinfo{year}{1937})
  \bibinfo{pages}{125--131}.
\bibitem[{Vitale(1975)}]{MR0397977}
\bibinfo{author}{R.~A. Vitale}, \bibinfo{title}{{B}ernstein polynomial approach
  to density function estimation}, in: \bibinfo{booktitle}{Statistical
  {I}nference and {R}elated {T}opics}, \bibinfo{publisher}{Academic Press, New
  York}, \bibinfo{year}{1975}, pp. \bibinfo{pages}{87--99}.
  \bibinfo{note}{\href{http://www.ams.org/mathscinet-getitem?mr=MR0397977}{MR0397977}}.
\bibitem[{Wansouw\'{e} et~al.(2015)Wansouw\'{e}, Kokonendji and
  Kolyang}]{MR3337328}
\bibinfo{author}{W.~E. Wansouw\'{e}}, \bibinfo{author}{C.~C. Kokonendji},
  \bibinfo{author}{D.~T. Kolyang}, \bibinfo{title}{Nonparametric estimation for
  probability mass function with {D}isake: an {R} package for discrete
  associated kernel estimators}, \bibinfo{journal}{ARIMA Rev. Afr. Rech.
  Inform. Math. Appl.} \bibinfo{volume}{19} (\bibinfo{year}{2015})
  \bibinfo{pages}{1--23}.
  \bibinfo{note}{\href{http://www.ams.org/mathscinet-getitem?mr=MR3337328}{MR3337328}}.
\bibitem[{Weglarczyk(2018)}]{doi:10.1051/itmconf/20182300037}
\bibinfo{author}{S.~Weglarczyk}, \bibinfo{title}{Kernel density estimation and
  its application}, \bibinfo{journal}{ITM Web of Conferences: XLVIII Seminar of
  Applied Mathematics} \bibinfo{volume}{23} (\bibinfo{year}{2018})
  \bibinfo{pages}{8pp.}
  \bibinfo{note}{\href{https://doi.org/10.1051/itmconf/20182300037}{doi:10.1051/itmconf/20182300037}}.
\bibitem[{Xiao et~al.(2019)Xiao, Li and Shi}]{MR3978111}
\bibinfo{author}{J.~Xiao}, \bibinfo{author}{X.~Li}, \bibinfo{author}{J.~Shi},
  \bibinfo{title}{Estimation in a semiparametric partially linear
  errors-in-variables model with inverse {G}aussian kernel},
  \bibinfo{journal}{Comm. Statist. Theory Methods} \bibinfo{volume}{48}
  (\bibinfo{year}{2019}) \bibinfo{pages}{4394--4424}.
  \bibinfo{note}{\href{http://www.ams.org/mathscinet-getitem?mr=MR3978111}{MR3978111}}.
\bibitem[{Xu(2014)}]{MR3200997}
\bibinfo{author}{S.~Xu}, \bibinfo{title}{Asymmetric kernel density estimation
  based on grouped data with applications to loss model},
  \bibinfo{journal}{Comm. Statist. Simulation Comput.} \bibinfo{volume}{43}
  (\bibinfo{year}{2014}) \bibinfo{pages}{657--672}.
  \bibinfo{note}{\href{http://www.ams.org/mathscinet-getitem?mr=MR3200997}{MR3200997}}.
\bibitem[{Yilmaz(2011)}]{doi:10.1007/s00138-009-0237-4}
\bibinfo{author}{A.~Yilmaz}, \bibinfo{title}{Kernel-based object tracking using
  asymmetric kernels with adaptive scale and orientation selection},
  \bibinfo{journal}{Machine Vision and Applications} \bibinfo{volume}{22}
  (\bibinfo{year}{2011}) \bibinfo{pages}{255--268}.
  \bibinfo{note}{\href{https://doi.org/10.1007/s00138-009-0237-4}{doi:10.1007/s00138-009-0237-4}}.
\bibitem[{Yin and Hao(2007)}]{doi:10.1109/ICMLC.2007.4370716}
\bibinfo{author}{X.-F. Yin}, \bibinfo{author}{Z.-F. Hao},
  \bibinfo{title}{Adaptative kernel density estimation using beta kernel},
  \bibinfo{journal}{Proc. 6th Int. Conf. on Machine Learning and Cybernetics}
  (\bibinfo{year}{2007}) \bibinfo{pages}{19--22}.
  \bibinfo{note}{\href{https://doi.org/10.1109/ICMLC.2007.4370716}{doi:10.1109/ICMLC.2007.4370716}}.
\bibitem[{Yuan-ming et~al.(2011)Yuan-ming, Wei, Yi-ning and
  Guo-xuan}]{doi:10.1109/ICMT.2011.6001976}
\bibinfo{author}{D.~Yuan-ming}, \bibinfo{author}{W.~Wei},
  \bibinfo{author}{L.~Yi-ning}, \bibinfo{author}{Z.~Guo-xuan},
  \bibinfo{title}{Enhanced mean shift tracking algorithm based on evolutive
  asymmetric kernel}, \bibinfo{journal}{International Conference on Multimedia
  Technology}  (\bibinfo{year}{2011}) \bibinfo{pages}{5394--5398}.
  \bibinfo{note}{\href{https://doi.org/10.1109/ICMT.2011.6001976}{doi:10.1109/ICMT.2011.6001976}}.
\bibitem[{Zhang(2010)}]{MR2595129}
\bibinfo{author}{S.~Zhang}, \bibinfo{title}{A note on the performance of the
  gamma kernel estimators at the boundary}, \bibinfo{journal}{Statist. Probab.
  Lett.} \bibinfo{volume}{80} (\bibinfo{year}{2010}) \bibinfo{pages}{548--557}.
  \bibinfo{note}{\href{http://www.ams.org/mathscinet-getitem?mr=MR2595129}{MR2595129}}.
\bibitem[{Zhang and Karunamuni(1998)}]{MR1649872}
\bibinfo{author}{S.~Zhang}, \bibinfo{author}{R.~J. Karunamuni},
  \bibinfo{title}{On kernel density estimation near endpoints},
  \bibinfo{journal}{J. Statist. Plann. Inference} \bibinfo{volume}{70}
  (\bibinfo{year}{1998}) \bibinfo{pages}{301--316}.
  \bibinfo{note}{\href{http://www.ams.org/mathscinet-getitem?mr=MR1649872}{MR1649872}}.
\bibitem[{Zhang and Karunamuni(2000)}]{MR1752313}
\bibinfo{author}{S.~Zhang}, \bibinfo{author}{R.~J. Karunamuni},
  \bibinfo{title}{On nonparametric density estimation at the boundary},
  \bibinfo{journal}{J. Nonparametr. Statist.} \bibinfo{volume}{12}
  (\bibinfo{year}{2000}) \bibinfo{pages}{197--221}.
  \bibinfo{note}{\href{http://www.ams.org/mathscinet-getitem?mr=MR1752313}{MR1752313}}.
\bibitem[{Zhang and Karunamuni(2010)}]{MR2598955}
\bibinfo{author}{S.~Zhang}, \bibinfo{author}{R.~J. Karunamuni},
  \bibinfo{title}{Boundary performance of the beta kernel estimators},
  \bibinfo{journal}{J. Nonparametr. Stat.} \bibinfo{volume}{22}
  (\bibinfo{year}{2010}) \bibinfo{pages}{81--104}.
  \bibinfo{note}{\href{http://www.ams.org/mathscinet-getitem?mr=MR2598955}{MR2598955}}.
\bibitem[{Ziane et~al.(2015)Ziane, Adjabi and Zougab}]{MR3350456}
\bibinfo{author}{Y.~Ziane}, \bibinfo{author}{S.~Adjabi},
  \bibinfo{author}{N.~Zougab}, \bibinfo{title}{Adaptive {B}ayesian bandwidth
  selection in asymmetric kernel density estimation for nonnegative
  heavy-tailed data}, \bibinfo{journal}{J. Appl. Stat.} \bibinfo{volume}{42}
  (\bibinfo{year}{2015}) \bibinfo{pages}{1645--1658}.
  \bibinfo{note}{\href{http://www.ams.org/mathscinet-getitem?mr=MR3350456}{MR3350456}}.
\bibitem[{Ziane et~al.(2018)Ziane, Zougab and Adjabi}]{MR3754719}
\bibinfo{author}{Y.~Ziane}, \bibinfo{author}{N.~Zougab},
  \bibinfo{author}{S.~Adjabi}, \bibinfo{title}{Birnbaum-{S}aunders
  power-exponential kernel density estimation and {B}ayes local bandwidth
  selection for nonnegative heavy tailed data}, \bibinfo{journal}{Comput.
  Statist.} \bibinfo{volume}{33} (\bibinfo{year}{2018})
  \bibinfo{pages}{299--318}.
  \bibinfo{note}{\href{http://www.ams.org/mathscinet-getitem?mr=MR3754719}{MR3754719}}.
\bibitem[{Ziane et~al.(2021)Ziane, Zougab and Adjabi}]{MR4223858}
\bibinfo{author}{Y.~Ziane}, \bibinfo{author}{N.~Zougab},
  \bibinfo{author}{S.~Adjabi}, \bibinfo{title}{Body tail adaptive kernel
  density estimation for nonnegative heavy-tailed data},
  \bibinfo{journal}{Monte Carlo Methods Appl.} \bibinfo{volume}{27}
  (\bibinfo{year}{2021}) \bibinfo{pages}{57--69}.
  \bibinfo{note}{\href{http://www.ams.org/mathscinet-getitem?mr=MR4223858}{MR4223858}}.
\bibitem[{Zougab and Adjabi(2016)}]{MR3456321}
\bibinfo{author}{N.~Zougab}, \bibinfo{author}{S.~Adjabi},
  \bibinfo{title}{Multiplicative bias correction for generalized
  {B}irnbaum-{S}aunders kernel density estimators and application to
  nonnegative heavy tailed data}, \bibinfo{journal}{J. Korean Statist. Soc.}
  \bibinfo{volume}{45} (\bibinfo{year}{2016}) \bibinfo{pages}{51--63}.
  \bibinfo{note}{\href{http://www.ams.org/mathscinet-getitem?mr=MR3456321}{MR3456321}}.
\bibitem[{Zougab et~al.(2013)Zougab, Adjabi and Kokonendji}]{MR3169297}
\bibinfo{author}{N.~Zougab}, \bibinfo{author}{S.~Adjabi},
  \bibinfo{author}{C.~C. Kokonendji}, \bibinfo{title}{Adaptive smoothing in
  associated kernel discrete functions estimation using {B}ayesian approach},
  \bibinfo{journal}{J. Stat. Comput. Simul.} \bibinfo{volume}{83}
  (\bibinfo{year}{2013}) \bibinfo{pages}{2219--2231}.
  \bibinfo{note}{\href{http://www.ams.org/mathscinet-getitem?mr=MR3169297}{MR3169297}}.
\bibitem[{Zougab et~al.(2014)Zougab, Adjabi and Kokonendji}]{MR3178355}
\bibinfo{author}{N.~Zougab}, \bibinfo{author}{S.~Adjabi},
  \bibinfo{author}{C.~C. Kokonendji}, \bibinfo{title}{Bayesian estimation of
  adaptive bandwidth matrices in multivariate kernel density estimation},
  \bibinfo{journal}{Comput. Statist. Data Anal.} \bibinfo{volume}{75}
  (\bibinfo{year}{2014}) \bibinfo{pages}{28--38}.
  \bibinfo{note}{\href{http://www.ams.org/mathscinet-getitem?mr=MR3178355}{MR3178355}}.
\bibitem[{Zougab et~al.(2016)Zougab, Adjabi and Kokonendji}]{MR3453033}
\bibinfo{author}{N.~Zougab}, \bibinfo{author}{S.~Adjabi},
  \bibinfo{author}{C.~C. Kokonendji}, \bibinfo{title}{Comparison study to
  bandwidth selection in binomial kernel estimation using {B}ayesian
  approaches}, \bibinfo{journal}{J. Stat. Theory Pract.} \bibinfo{volume}{10}
  (\bibinfo{year}{2016}) \bibinfo{pages}{133--153}.
  \bibinfo{note}{\href{http://www.ams.org/mathscinet-getitem?mr=MR3453033}{MR3453033}}.
\bibitem[{Zougab et~al.(2018)Zougab, Harfouche, Ziane and Adjabi}]{MR3819800}
\bibinfo{author}{N.~Zougab}, \bibinfo{author}{L.~Harfouche},
  \bibinfo{author}{Y.~Ziane}, \bibinfo{author}{S.~Adjabi},
  \bibinfo{title}{Multivariate generalized {B}irnbaum-{S}aunders kernel density
  estimators}, \bibinfo{journal}{Comm. Statist. Theory Methods}
  \bibinfo{volume}{47} (\bibinfo{year}{2018}) \bibinfo{pages}{4534--4555}.
  \bibinfo{note}{\href{http://www.ams.org/mathscinet-getitem?mr=MR3819800}{MR3819800}}.

\end{thebibliography}

\end{document}